\documentclass{amsart}
\usepackage[margin=3.5cm]{geometry}
\usepackage{xcolor}

\usepackage{amsfonts, amsthm, amssymb,verbatim,amscd,amsmath, blindtext}

\usepackage{multirow}
\usepackage{graphicx}
\usepackage{enumitem}
\usepackage{amssymb}
\usepackage{ytableau}
\usepackage{tikz}
\usepackage{tikz-cd}
\usepackage{float, pifont}
\usepackage{mathtools}
\usepackage[pagebackref,colorlinks=true,citecolor=blue,linkcolor=black]{hyperref}

\newtheorem{theorem}{Theorem}[section]
\newtheorem{lemma}[theorem]{Lemma}
\newtheorem{corollary}[theorem]{Corollary}
\newtheorem{proposition}[theorem]{Proposition}
\newtheorem{conjecture}[theorem]{Conjecture}
\newtheorem{claim}[theorem]{Claim}
\newtheorem{notation}[theorem]{Notation}

\theoremstyle{remark}
\newtheorem{remark}[theorem]{Remark}
\newtheorem{question}[theorem]{Question}

\theoremstyle{definition}
\newtheorem{definition}[theorem]{Definition}
\newtheorem{example}[theorem]{Example}
\newtheorem{algorithm}[theorem]{Algorithm}

\DeclareMathOperator{\rank}{rank}

\DeclareMathOperator{\lcm}{lcm}

\DeclareMathOperator{\pd}{pd}

\DeclareMathOperator{\sbridge}{sb}

\DeclareMathOperator{\sink}{sink}
\DeclareMathOperator{\iron}{iron}

\newcommand{\G}{\mathcal{G}}

\newcommand{\D}{\mathcal{D}}
\newcommand{\T}{\mathcal{T}}
\newcommand{\C}{\mathcal{C}}

\newcommand{\ZZ}{{\mathbb Z}}

\newcommand{\cmark}{\ding{51}}
\newcommand{\xmark}{\ding{55}}

\def\T{{\mathcal T}}

\def\G{{\mathcal G}}
\def\P{{\mathcal P}}

\def\G{{\mathcal G}}

\def\w{{\bf w}}

\def\1{{\bf 1}}
\def\0{{\bf 0}}

\begin{document}
\title{\textbf{Barile-Macchia Resolutions}}

\author{Trung Chau}
\address{Department of Mathematics, University of Utah, 155 South 1400 East, Salt Lake City, UT~84112, USA}
\email{trung.chau@utah.edu}

\author{Selvi Kara$^{\ast}$}
\address{Science Research Initiative, University of Utah, 155 South 1400 East, Salt Lake City, UT~84112, USA}
\email{selvi.kara@utah.edu}
\date{}

\maketitle

\begin{abstract}
  We construct cellular resolutions for monomial ideals via discrete Morse theory. In particular, we develop an algorithm to create homogeneous acyclic matchings and we call the cellular resolutions induced from these matchings  Barile-Macchia resolutions. These resolutions are minimal for edge ideals of weighted oriented forests and (most) cycles. As a result, we provide recursive formulas for graded Betti numbers and projective dimension. Furthermore, we compare Barile-Macchia resolutions to those created by Batzies and Welker and some well-known simplicial resolutions. Under certain assumptions, whenever the above resolutions are minimal, so are Barile-Macchia resolutions. 
\end{abstract}

\tableofcontents

\section{Introduction}

There are several combinatorial constructions to produce free resolutions of monomial ideals over a polynomial ring $R=\Bbbk[x_1,\ldots, x_n]$. Such constructions rarely yield minimal ones; Taylor \cite{Tay66} and Lyubeznik resolutions \cite{Ly88} are classic such examples.  Thus, it is of interest to find new and nice constructions that produce free resolutions which are minimal for large classes of ideals, which is the main objective of our paper. A general approach is to associate a monomial ideal $I$ to a combinatorial object that supports a minimal free resolution of $R/I$.  This has been a fruitful approach for a myriad of ideals including stable ideals \cite{AHH98, EK90, Mer10, Pe94, Sin08}, generic ideals \cite{BPS98,BS98,MSY00}, matroid ideals of a finite projective space \cite{Nov00,NPS02}, shellable ideals \cite{BW02}, quadratic ideals with minimal linear resolutions \cite{Ho07}, edge ideals of Ferrers graphs and their specializations \cite{CN08-2, CN08}, and edge ideals of the complements of cycles \cite{Bi11}.

Our focus is on Batzies and Welker's approach from  \cite{BW02}  that is based on discrete Morse theory.  In \cite{BW02},  they developed a method to ``trim" a cellular resolution induced from a regular CW-complex, and their method relies on Chari's reformulation \cite{Cha00} of Forman’s discrete Morse theory \cite{Fo94}. In particular, they showed that \textit{homogeneous acyclic matchings } induce cellular resolutions. Moreover, they provided a way to produce such matchings for any monomial ideal and proved that the free resolutions induced by them are minimal for generic and shellable ideals.  Inspired by this work, researchers constructed minimal free resolutions for powers of edge ideals of paths \cite{EN13}, edge ideals of forests \cite{BM20}, and  powers of square-free monomial ideals of projective dimension one \cite{CEFMMSS22}. Structure of these ideals heavily informs their homogeneous acyclic matchings.

In this paper, we generalize Barile and Macchia's  approach \cite{BM20} to all monomial ideals, and adopt some of their terminology. Specifically, we provide the \textit{Barile-Macchia algorithm} that produces  homogeneous acyclic matchings (Algorithm \ref{algorithm1} and Theorem \ref{thm:proofalg}).  We call a matching produced this way a \emph{Barile-Macchia matching} and the corresponding induced free resolution a \textit{Barile-Macchia resolution}. Barile-Macchia matchings give rise to an interesting classes of monomial ideals called  \emph{bridge-friendly} (see Section \ref{subsec:contents}). A monomial ideal is bridge-friendly if, roughly speaking, it behaves \emph{well} under our algorithm. More importantly, the Barile-Macchia resolutions of bridge-friendly ideals are always minimal (see Theorem \ref{thm:friendlyToMinimal}).  In addition, in many occasions, we show that Barile-Macchia matchings are \textit{Batzies-Welker}. The main class of ideals we consider for bridge-friendliness are the edge ideals of weighted oriented graphs.

A \emph{weighted oriented graph} is a triple $\D=(V,E,\w)$ where $V$ is the vertex set, $E$ is the set of directed edges, and $\w: V\to \mathbb{N}^{+}$ is a weight function on the vertices. Let $ R:=\Bbbk[\D]$ denote the polynomial ring on the vertices. The \emph{edge ideal} of $\D$ is denoted $I(\D)$ and defined as
$$I(\D)= ( xy^{\w(y)}\mid (x,y) \in E ) \subseteq R.$$ 
When all vertices have weight one, $I(\D)$ is the edge ideal of an unweighted unoriented graph which was introduced by Villarreal \cite{villarreal1990cohen} and has been studied extensively since (see \cite{banerjee2017regularity} for a survey). Edge ideals of weighted oriented graphs are relatively newer objects and there has been a growing literature on their algebraic properties and invariants \cite{gimenez2018symbolic, ha2019edge, pitones2019monomial, kara2022algebraic}. As a first application of our methods, we consider forests and show that their edge ideals are bridge-friendly (Theorem \ref{thm:morsefriendlyforest}). Next, we consider edge ideals of weighted oriented cycles. The story for this class of ideals is more involved.  We study these cycles by partitioning them into \textit{classic} and \textit{non-classic} cycles (see Section \ref{sec:cycle}). Edge ideals of non-classic cycles are bridge-friendly (Theorem \ref{morsepropertycycle}). Although this is not necessarily the case for edge ideals of classic cycles (Subsection \ref{subsec:classic}), we prove that their minimal resolutions can be obtained from those of bridge-friendly ideals via the mapping cone construction (Corollary \ref{cor:resofclassic}). In these cases, we deduce recursive (and explicit) formulas for Betti numbers and projective dimension.

A cycle (or path/forest) is called \emph{naturally oriented} if all of its edges are oriented in the same direction. Our methods yield to a powerful property which we call \emph{independence of directions} for weighted oriented cycles and paths. Specifically, given a weighted oriented cycle (resp, path), one can find a naturally oriented cycle (resp, path) such that the total Betti numbers of the two edge ideals coincide (Theorems \ref{indecycle} and \ref{indepath}).  Unfortunately, weighted oriented forests do not have this property in general (Example \ref{notindeforest}).  We also show it is often possible to find a weighted oriented path (resp, cycle) from a given weighted oriented cycle (resp, path) such that the total Betti numbers of their edge ideals are equal (Theorem \ref{thm:identicalBetti}). 

In a more general direction, we compare Barile-Macchia resolutions with  Scarf complexes, Taylor and Lyubeznik resolutions (see Section \ref{sec:comparison}). First, we provide an ideal whose Barile-Macchia, Taylor, and Lyubeznik resolutions and Scarf complex are all non-isomorphic and the only minimal one among these is a Barile-Macchia resolution (Theorem \ref{exampleofmultipleres}). Next, we study when these resolutions or the Scarf complex coincide with a Barile-Macchia resolution. The most immediate case is the Taylor resolutions as they  coincide with Barile-Macchia resolutions when the Barile-Macchia matching is the empty set.   In the case of Lyubeznik resolutions,  while there are examples where the two resolutions are not comparable in general, we identify sufficient conditions under which Barile-Macchia resolutions are closer to minimal ones (Theorem \ref{BridgevsLyu}). In addition, we prove that Barile-Macchia resolutions coincide minimally with the Scarf complex under a natural assumption, recovering a result of Yuzvinsky \cite[Proposition 4.4]{Yu99} (Theorem \ref{YuzvinskyisMorseminimal}). Finally, we generalize Barile-Macchia resolutions in the same way Batzies and Welker generalize Lyubeznik resolutions in \cite{BW02}. We briefly compare these two generalized versions (Theorem ~\ref{BridgevsLyu3}).

Our paper is structured as follows: Section \ref{section2} contains the background regarding discrete Morse theory (Subsection \ref{subsection2.1}), the Barile-Macchia algorithm (Subsection \ref{subsec:AMA}), Barile-Macchia matchings, and the notion of \emph{bridge-friendliness} (Subsection \ref{subsec:contents}). We present a criterion to check bridge-friendliness in Lemma \ref{lem:morsefri2} which proves to be quite useful in our work. In addition, we show that the corresponding Barile-Macchia resolution of a bridge-friendly ideal is minimal in Theorem \ref{thm:friendlyToMinimal}. In Section \ref{sec:forests}, we focus on edge ideals of weighted oriented forests and  prove Theorem \ref{thm:morsefriendlyforest}.  Edge ideals of weighted oriented cycles are studied in Section \ref{sec:cycle}. In particular, we focus on independence of directions in  Subsection \ref{subsec:independence}, non-classic cycles in  Subsection \ref{subsec:non-classic}, and classic cycles in Subsection \ref{subsec:classic}. In Section \ref{sec:comparison}, we provide a comparison between Barile-Macchia resolutions and some well-known simplicial resolutions and complexes.  In Section \ref{sec:final}, we  introduce  some open questions and conjectures.

\section{Barile-Macchia matchings}\label{section2}

\subsection{Preliminaries}\label{subsection2.1}

We recall some of the basic discrete Morse theory notions from  \cite{BW02}. Let $X$ be a CW-complex, $P$ a poset and $f\colon X^{(*)}\to P$ an order-preserving map where $X^{(*)}$ denotes the set of cells of $X$ and is ordered by $\sigma \leq \sigma'$ whenever $\sigma \subseteq \text{cl}(\sigma')$. Consider the directed graph $G=(V,E)$ where $V$ is the set of cells of $X$ and $E$ is the set of directed edges given by $\sigma\to \sigma'$ for $\sigma'\leq  \sigma$ and $|\sigma'|=|\sigma|-1$. We adopt the notation $(\sigma, \sigma ')$ to denote a directed edge  from $\sigma$ to $\sigma '$. For $A\subseteq E$, let $G^A$ be the directed graph obtained from $G$ by reversing the edges in $A$, i.e., $V(G^A)=V(G)$ and  $E(G^A)=(E \setminus A) \cup \{(\sigma', \sigma ) \mid (\sigma, \sigma') \in A\}.$

\begin{definition}\label{def:acyclicmatch}
    A subset $A \subseteq E$ is called an  \textit{$f$-homogeneous acyclic matching} if it satisfies the following conditions:
    \begin{enumerate}
        \item (\emph{matching}) Any cell appears in at most one edge of $A$.
        \item (\emph{acyclicity}) The edge set $E(G^A)$ does not contain a directed cycle.
        \item (\emph{$f$-homogeneity}) If $(\sigma,  \sigma') \in A$, then $f(\sigma)=f(\sigma')$.
    \end{enumerate}
In this case, a cell is called $A$-\textit{critical} if does not appear in any of the edges of $A$. When there is no confusion, we will simply use the term \textit{critical}.
\end{definition}

For a directed edge $(\sigma, \sigma') \in E(G_A)$, we set 
$$m(\sigma,\sigma')=\begin{cases}
    -[\sigma':\sigma] & \text{ if } (\sigma', \sigma) \in A,\\
     ~~[\sigma:\sigma'] & \text{ otherwise}
\end{cases}
$$
where $[\sigma:\sigma'] $ is the coefficient of $\sigma'$ in the differential of the cellular complex of $X$. A \textit{gradient path} $\mathcal{P}$  from $\sigma_1$ to $\sigma_t$ is a directed path $\mathcal{P}\colon \sigma_1 \to \sigma_2\to \cdots \to \sigma_t$ in  $G^A_X$. Set
\[ m(\mathcal{P})=m(\sigma_1,\sigma_2)\cdots m(\sigma_{t-1}, \sigma_t).\]

In this paper, we only focus on finding $f$-homogeneous acyclic matchings of the Taylor complex. For the rest of the paper, unless otherwise stated, let $R$ denote a polynomial ring in $N$ variables over a field $\Bbbk$, and $I$ be a monomial ideal with its set of minimal monomial generators $\G(I)$. Impose a total ordering   $(>_I)$ on $\G(I)$ and let $X$ be the full simplex whose vertices are labelled by the elements of $\G(I)$. Then, cells of $X$ correspond to subsets of $\G(I)$. We will treat a subset $\sigma$ of $\G(I)$ as an ordered set with respect to $(>_I)$. Set $\lcm(\sigma)$ to be the least common multiple of elements in $\sigma$. Recall that $X$ is a $\mathbb{Z}^N$-graded complex that induces a complex $\mathcal{F}$ where $\mathcal{F}_r$ is the free $R$-module with a basis indexed by all subsets of $\G(I)$ of cardinality $r$, and the differentials $\partial_r\colon \mathcal{F}_r\to \mathcal{F}_{r-1}$ are  defined by
\[ \partial_r(\sigma) =\sum_{\substack{\sigma'\subseteq \sigma,\\|\sigma'|=r-1}} [\sigma:\sigma'] \frac{\lcm(\sigma)}{\lcm(\sigma')} \sigma'. \]
The complex $\mathcal{F}$ is well-known to be a resolution \cite{Tay66}, hence is called the \textit{Taylor resolution} of $R/I$.

We are now ready to define the resolution induced by an $f$-homogeneous acyclic matching $A$.

\begin{definition}
    If there exists a commutative diagram of poset maps
    \[
    \begin{tikzcd}
        X^{(*)} \arrow[d, "\lcm"] \arrow[rd, "f"]
        & \\
        \mathbb{Z}^N
        & \arrow[l, dashed, "g"]  P,
    \end{tikzcd}
    \]
    then $f$ is called an \emph{$\lcm$-compatible $P$-grading} of $X$. 
\end{definition}

\begin{theorem}\cite[Proposition 2.2, Proposition 3.1,  Lemma 7.7]{BW02} \label{thm:morseres}
    Let $I$ be a monomial ideal and $f$ an $\lcm$-compatible $P$-grading of $X$. Then any $f$-homogeneous acyclic matching $A$ induces a cellular resolution $\mathcal{F}_A$ where $(\mathcal{F}_A)_r$ is the free $R$-module with a basis indexed by all critical subsets   of cardinality $r$ and the differentials are the maps $\partial_r^A:(\mathcal{F}_A)_r\to (\mathcal{F}_A)_{r-1}$ defined by
    \[ \partial_r^A(\sigma) =\sum_{\substack{\sigma'\subseteq \sigma,\\|\sigma'|=r-1}} [\sigma:\sigma'] \sum_{\substack{\sigma'' \text{ critical,}\\ |\sigma''|=r-1 }} \sum_{\substack{\mathcal{P} \text{ gradient path}\\ \text{from } \sigma' \text{ to }\sigma''}} m(\mathcal{P}) \frac{\lcm(\sigma)}{\lcm(\sigma'')} \sigma''. \]
    The resulting (cellular) free resolution $\mathcal{F}_A$ is called a \emph{Morse resolution} of $R/I$ associated to $A$.
\end{theorem}

Thus, one can bound the projective dimension and graded Betti numbers using critical subsets.

\begin{corollary}\label{morseminimabetti}
	Let $I$ be a monomial ideal and $A$ an $\lcm$-homogeneous acyclic matching. Then
	\begin{enumerate}
		\item[(a)] $\pd(R/I) \leq \{ |\sigma| : \sigma \text{ is a critical subset }  \}$, and 
		\item[(b)] $\beta_{r,\mathbf{a}}(R/I) \leq \Big| \{ \sigma : \sigma  \text{ is a critical subset } , |\sigma|=r, \lcm (\sigma) =\mathbf{x}^\mathbf{a} \} \Big| $. 
	\end{enumerate}
	where $\mathbf{x}^\mathbf{a}\coloneqq x_1^{a_1}\dots x_N^{a_N}$. Moreover, equalities in (a) and (b) hold if and only if the Morse resolution $\mathcal{F}_A$ is minimal. 
\end{corollary}

The following simple sufficient condition is shown to be effective to deduce the minimality of Morse resolutions \cite{BW02,BM20,faridi22}.

\begin{definition}
	An $\lcm$-homogeneous acyclic matching  is called \textit{Batzies-Welker matching} if $\lcm(\sigma)\neq \lcm(\sigma')$ for any two critical subsets $\sigma,\sigma'$ of $\G(I)$. 
\end{definition} 

The next result follows from Corollary \ref{morseminimabetti}.

\begin{theorem}\label{thm:BWmatching}
	Let $I$ be a monomial ideal. If $I$ has a Batzies-Welker matching, then the corresponding Morse resolution of $R/I$ is minimal. Moreover, 
	\[
	\beta_{r,\mathbf{a}}(R/I)=\begin{cases}
		1 & \text{if there exists a critical subset $\sigma\subseteq \G(I)$ such that } |\sigma|=r \text{ and } \lcm (\sigma) =\mathbf{x}^\mathbf{a},\\
		0 & \text{otherwise.}  
	\end{cases}
	\]
\end{theorem}

\subsection{Barile-Macchia Algorithm}\label{subsec:AMA}

Throughout the paper, let $[k]$ denote the set $\{1,2,\cdots, k\}$ for any integer $k$.

\begin{definition}
Let $\sigma$ be a subset of $\G(I)$. A monomial $m\in \G(I)$ is called a \textit{bridge} of $\sigma$ if $m\in \sigma$ and $\lcm (\sigma \setminus \{m\})=\lcm (\sigma)$.
\end{definition}

\begin{definition} 
Let $B$ be the collection of all subsets of $\G(I)$. The \emph{smallest bridge function}, denoted by $\sbridge$, is a map $\sbridge: B \to \G(I) \cup \{\emptyset\}$ where $\sbridge(\sigma)$ is the smallest bridge of $\sigma$ if it has a bridge and $\emptyset$ otherwise.
\end{definition}

In what follows, we define an iterative algorithm that produces an $\lcm$-homogeneous acyclic matching. We implemented this algorithm in {\tt Macaulay2}, and our code is available online at~\cite{github}. After describing the steps of the algorithm, we show that the matching is indeed acyclic. To ease the notation, we use  $A \setminus a$ and $A\cup a$ instead of $A \setminus \{ a\}$ and $A\cup \{a\}$, respectively.
 
\begin{algorithm}\label{algorithm1}
    {\sf Let $A=\emptyset$. Set $\Omega=\{\text{all  subsets of } \G(I) \text{ with cardinality at least } 3\}.$
    \begin{enumerate}[label=(\arabic*)]
        \item Pick a subset $\sigma$ of maximal cardinality in $\Omega$. 
        \item  Set
        \[ \Omega \coloneqq \Omega \setminus \{\sigma, \sigma \setminus \sbridge(\sigma)\}. \]
        If  $\sbridge(\sigma)\neq \emptyset$, add the directed edge $(\sigma , \sigma \setminus \sbridge(\sigma))$ to $A$. If $\Omega\neq \emptyset$, return to step (1).
        \item Whenever there exist distinct directed edges $(\sigma, \sigma \setminus \sbridge(\sigma))$ and $(\sigma', \sigma' \setminus \sbridge(\sigma'))$ in $A$ such that 
            $$\sigma \setminus \sbridge(\sigma) = \sigma' \setminus \sbridge(\sigma'),$$
            then 
            \begin{itemize}
                \item if $ \sbridge(\sigma') >_I \sbridge(\sigma)$, remove $(\sigma', \sigma' \setminus \sbridge(\sigma'))$ from $A$,
                \item  otherwise, remove $(\sigma, \sigma \setminus \sbridge(\sigma))$ from $A$.
            \end{itemize}
    \end{enumerate}}
\end{algorithm}

This process eventually terminates since there are only finitely many subsets to consider. It is straightforward from  steps (2) and (3) that $A$ is an $\lcm$-homogeneous matching. In order to show it is acyclic, we will employ the next lemma.

\begin{lemma}\label{clm:directedcycle}
        If $A$ is not acyclic, then there exists a directed cycle $\mathcal{C}$ of the form
        $$\mathcal{C}: \tau_1\to \sigma_1 \to \tau_2\to \sigma_2 \to \cdots \to \tau_k \to \sigma_k \to \tau_1$$
    where $k\geq 2$ and
        \begin{enumerate}
            \item[(a)] for each $i \in [k]$, we have $\tau_{i+1}=\sigma_{i+1} \setminus \sbridge(\sigma_{i+1}) =\sigma_{i}\setminus m_i$  for some $m_i\in \G(I)$, and
            \item[(b)]  $\lcm (\sigma_i)= \lcm(\tau_j)$ for all $ i,j \in [k]$.
        \end{enumerate}
        Here we use the convention  $k+1=1$.
    \end{lemma} 

\begin{proof}
    Suppose $A$ is not acyclic. Then, $G^A$ has a directed cycle $\mathcal{C}$. Note that directed edges of $G^A$ are either in $E \setminus A$ or of the form $(\sigma', \sigma)$ for $(\sigma, \sigma') \in A$.  Since $A$ is a matching, the directed cycle $\mathcal{C}$ cannot contain two consecutive directed edges of the second form. If $\mathcal{C}$ has directed edges $(\tau_1,\tau_2),  \ldots, (\tau_{k},\tau_{k+1})$ from $E \setminus A$, one has $|\tau_{i+1}|=|\tau_{i}|-1$ for each $i \in [k] $. In order to circle back to $\tau_1$ along  $\mathcal{C}$, one must travel through $k$ many directed edges of the second form. Since it is not possible to have two successive directed edges of the second form, the directed cycle $\mathcal{C}$ must be formed by alternating directed edges of the first and second forms. Thus, $\mathcal{C}$  is of the  form
$$\tau_1\to \sigma_1 ~\textcolor{blue}{\to} ~\tau_2\to \sigma_2 ~\textcolor{blue}{\to} ~ \cdots \to \tau_k \to \sigma_k ~\textcolor{blue}{\to}~  \tau_1$$
where $k+1=1$ and, for each $i\in [k]$, we have
\begin{itemize}
    \item $(\sigma_i, \tau_{i+1})= (\sigma_i, \sigma_i\setminus m_i) \in E \setminus A$ for some $m_{i}\in \G(I)$  (blue edges),
    \item $(\tau_i,\sigma_i)=(\sigma_i \setminus \sbridge(\sigma_i), \sigma_i)$ where $(\sigma_i,\tau_i) \in A$ (black edges).
\end{itemize}

Next, observe that $k \geq 2$. Otherwise, if $k=1$, then the directed edges of $\mathcal{C}$ are $(\sigma,\tau)$ and $(\tau, \sigma)$. Since the edges alternate between the first and second form, we must have  $(\sigma,\tau) \in \big(E\setminus A \big) \cap A$, a contradiction. We conclude the proof with two observations. Since $(\sigma_i, \tau_i)\in A$ for each $i\in [k]$, we have $\lcm(\sigma_i)= \lcm(\tau_i)$ for each $i$. In addition, since $\tau_{i+1}= \sigma_i \setminus m_i$ for each $i\in [k]$, we have $\lcm (\tau_{i+1}) | \lcm (\sigma_i)$. As a result, we obtain $(b)$ from the following:
\[ \lcm(\sigma_1)=\lcm(\tau_1) \mid \lcm(\sigma_k)= \lcm(\tau_k) \mid  \cdots \mid \lcm(\sigma_2) = \lcm(\tau_2) \mid \lcm(\sigma_1).    \] 
\end{proof}

\begin{theorem}\label{thm:proofalg}
   The collection of directed edges $A$ obtained from Algorithm \ref{algorithm1} is an $\lcm$-homogeneous acyclic matching.
\end{theorem}
\begin{proof}
    For the sake of contradiction, suppose $A$ is not acyclic. Then, there exists a directed cycle $\mathcal{C}$  in $G^A$ satisfying  $(a)$ and $(b)$ from Lemma \ref{clm:directedcycle}. In particular, each directed edge of  $\mathcal{C}$ is of the form $(\tau_i, \sigma_i)=(\sigma_i \setminus \sbridge(\sigma_i), \sigma_i)$ or $(\sigma_i, \tau_{i+1})= (\sigma_i, \sigma_i \setminus m_i)$. A direct edge of the first form is obtained by adding  a monomial $\sbridge(\sigma_i)$ and the second form is obtained by removing a monomial $m_i$. 
    
    Set $S=\{\sbridge(\sigma_1), \sbridge(\sigma_2), \ldots, \sbridge(\sigma_k)\}$. If an edge is formed by adding a monomial, then that monomial must be an element of $S$. Hence, without loss of generality, we may assume  $\sbridge(\sigma_1)$ is the largest element in $S$ with respect to the total ordering $(>_I)$ on $\G(I)$. 
    
    It follows from Lemma \ref{clm:directedcycle} that the directed edge $(\sigma_1, \tau_2)$ is obtained by removing $m_1$. Since $(\tau_1,\sigma_1)$ is a directed edge of $\mathcal{C}$, the monomial $m_1$ is eventually added back in an edge of the first form. Thus, $m_1 \in S$, i.e., $m_1= \sbridge(\sigma_j)$ for some $j \in [k]$. Note that $j\neq 1$. Otherwise, it would imply $\tau_1=\tau_2$, which is impossible by the construction of $G^A$. Furthermore, by Lemma \ref{clm:directedcycle}, we have 
    $$\lcm(\tau_{2}\cup m_1)=\lcm(\sigma_1)=\lcm(\tau_{2}) = \lcm (\sigma_1 \setminus m_1)$$
    where the first and last equality follows from $(a)$ and the second one follows from $(b)$, and hence $m_1$ is a bridge of $\sigma_1$. Thus, $m_1= \sbridge(\sigma_j) >_I \sbridge(\sigma_1)$, a contradiction. Therefore, $A$ is acyclic.
\end{proof}

\begin{definition}
    An $\lcm$-homogeneous acyclic matching induced from Algorithm \ref{algorithm1} is called a \textit{Barile-Macchia matching}. A resolution induced from a Barile-Macchia matching is called a \textit{Barile-Macchia resolution}.
\end{definition}

We conclude this section with the following example to illustrate  Algorithm \ref{algorithm1} and present the corresponding Barile-Macchia resolution.

   \begin{example}\label{ex:alg1}
     Set $R=\Bbbk[x,y,z,w]$, $I=(xw,xy,yz,zw)$ and consider the total ordering $xw> xy> yz> zw$ on $\G(I)$. We apply Algorithm \ref{algorithm1}:
  
  {\sf Set $A=\emptyset$ and 
  	$$\Omega:=\{ \{xw,xy,yz,zw\}, \{xy,yz,zw\},\{xw,yz, zw\},  \{xw,xy,zw\},\{xw,xy,yz\}\}.$$
  \begin{itemize}
      \item (step 1) Pick an element of maximum cardinality in $\Omega$: $\sigma= \{xw,xy,yz,zw\}$. 
      \item (step 2) $\sbridge(\sigma)=zw$. Then $$A=A\cup \{ (\{xw,xy,yz,zw\},\{xw,xy,yz\})\}$$ and
      $$\Omega:=\{  \{xy,yz,zw\},\{xw,yz, zw\},  \{xw,xy,zw\}\}$$
      \item (step 1) Pick an element of maximum cardinality in $\Omega$: $\sigma=\{xw,yz,zw\}$. 
      \item (step 2)  $\sbridge(\sigma)=zw$. Then $$A=A\cup \{ (\{xw,yz,zw\}, \{xw,yz\})\}$$ and
        \[ \Omega\coloneqq \{ \{xw,xy,zw\},\{xw,yz,zw\} \}. \]
        \item (step 1) Pick an element of maximum cardinality in $\Omega$: $\sigma=\{xw,xy,zw\}$.
        \item (step 2)  $\sbridge(\sigma)=xw$. Then $$A=A\cup \{( (\{xw,xy,zw\},\{xy,zw\}) \}$$ and
        \[ \Omega\coloneqq \{ \{xy,yz,zw\} \}. \]
       \item (step 1) Pick an element of maximum cardinality in $\Omega$: $\sigma=\{xy,yz,zw\}$.
       \item (step 2)  $\sbridge(\sigma)=yz$. Then 
       \begin{align*}
           A=A\cup \{     (\{xy,yz,zw\},\{xy,zw\})\}
       \end{align*} 
       and $ \Omega\coloneqq \emptyset$.
    
        \item Since $\Omega=\emptyset.$ proceed to step 3.
        \item (step 3) There is only one pair of directed edges with the same target:
        $$(\underbrace{\{xw,xy,zw\}}_{\sigma'}, \{xy,zw\}) \text{ and } (\underbrace{\{xy,yz,zw\}}_{\sigma}, \{xy,zw\}).$$ 
        Since $\sbridge(\sigma')=xw>_I yz=\sbridge(\sigma)$, remove the former edge from $A$. Then 
         $$A=\{(\{xw,xy,yz,zw\},\{xw,xy,yz\}), (\{xw,yz,zw\}, \{xw,yz\}), (\{xy,yz,zw\},\{xy,zw\})\}.$$ 
         \item Terminate.
  \end{itemize}
  }
      Critical subsets include one  of cardinality zero (the empty set), four of cardinality $1$, four of cardinality $2$, and one of cardinality $3$. By Theorem \ref{thm:morseres}, the matching $A$ induces a cellular resolution of $R/I$:
    \[ 0 \leftarrow R \leftarrow R^4\leftarrow R^4 \leftarrow R \leftarrow 0. \]
 In the following table, we compute the least common multiples of all the critical subsets and see that they are all different.

    \begin{center}
        \begin{tabular}{|c|c|c|c|c|c|}
        \hline
        & $\{xw,xy\}$ & $\{xw,zw\}$ & $\{xy,yz\}$ & $\{yz,zw\}$ & $\{xw,xy,zw\}$ \\ \hline
        lcm & $xyw$       & $xzw$       & $xyz$       & $yzw$       & $xyzw$         \\ \hline
    \end{tabular}
    \end{center}
Hence, this resolution is minimal. In particular, this matching is Batzies-Welker. Note that there is no need to consider subsets of cardinality $1$ in the algorithm because for any $m,m',m''\in \G(I)$, we must have $\lcm(m)\neq \lcm(\{m',m''\})$. 
\end{example}

\begin{remark}\label{Taylorismorseminimal}
    Algorithm \ref{algorithm1} always produces a resolution shorter than the Taylor resolution except for when none of the subsets of $\G(I)$ have a bridge, which happens exactly when the latter is minimal. In this case, Algorithm \ref{algorithm1} produces $A=\emptyset$ and  $\mathcal{F}_{A}$ coincides with the Taylor resolution. 
\end{remark}

Barile-Macchia resolutions of $R/I$ depend on the choice of total orderings on $\G(I)$. Two different total orderings may produce different Barile-Macchia matchings and the corresponding Barile-Macchia resolutions are not necessarily isomorphic, as can be seen in the following example.

\begin{example}\label{example:DependOnordering}
   Consider the ideal $I=(x^2y^2,y^2z^2,xz^2,x^2z)$ with the following total orderings $(>_1)$ and $(>_2)$:
   $$x^2y^2>_1 y^2z^2>_1 xz^2>_1 x^2z,$$
   $$\text{and } xz^2>_2 x^2z>_2 x^2y^2>_2 y^2z^2.$$
   Then the Barile-Macchia resolutions induced from the two orderings are
    \begin{align*}
        \mathcal{F}_1: &~~0\leftarrow R\leftarrow R^4\leftarrow R^3 \leftarrow 0,\text{ and }\\
        \mathcal{F}_2:& ~~0\leftarrow R\leftarrow R^4\leftarrow R^4\leftarrow R\leftarrow 0,
    \end{align*}
    where $\mathcal{F}_i$ is with respect to $(>_i)$ for $i\in \{1,2\}$. One can verify that  $\mathcal{F}_1$ is minimal while $\mathcal{F}_2$ is not.
\end{example}

\subsection{Contents of a Barile-Macchia matching}\label{subsec:contents}

Let $A$ be a Barile-Macchia matching throughout this section. Terminology introduced in this subsection is inspired by Barile and Macchia's work in \cite{BM20}.

\begin{definition}
For each directed edge $(\sigma,\tau)$ in $A$, we call $\sigma$ a \textit{type-2} element in $A$ and  $\tau$ a \textit{type-1} element in $A$. Moreover, for each directed edge $(\sigma, \tau)$ added to $A$ in step (2) of Algorithm \ref{algorithm1}, we call $\sigma$ a \textit{potentially-type-2} element in $A$. This directed edge does not necessarily appear in the final elements in $A$. If it does, then $\sigma$ is type-2.
\end{definition}

\begin{example}\label{ex:types}
Let $I$ be the ideal from Example \ref{ex:alg1}. The edges in $A$ are
$$\{\big(\textcolor{blue}{\{xw,xy,yz,zw\}},\{xw,xy,yz\}\big), \big(\textcolor{blue}{\{xw,yz,zw\}}, \{xw,yz\}\big), \big(\textcolor{blue}{\{xy,yz,zw\}},\{xy,zw\}\big)\}.$$

Then, the type-2 elements of $A$ are $\{xw,xy,yz,zw\}, \{xw,yz,zw\}$, and $\{xy,yz,zw\}$ and type-1 elements are $\{xw,xy,yz\}, \{xw,yz\}$, and $\{xy,zw\} $. Recall from Example \ref{ex:alg1} that the edge $(\{xw,xy,zw\}, \{xy,zw\})$ was added to $A$ in step (2) but removed at  step (3) of the algorithm. Thus, the subset $\{xw,xy,zw \}$ is the only  potentially-type-2 element which is not type-2.
\end{example}

\begin{remark}\label{rem:types}
It follows from step (2) of Algorithm \ref{algorithm1} that a subset of $\G(I)$  cannot be both type-1 and (potentially-) type-2. On the other hand, all type-2 subsets of $\G(I)$  are potentially-type-2 while the reverse is not always true. Moreover, if a subset of $\G(I)$  has a bridge, it must be either potentially-type-2 or type-1.
\end{remark}

In addition to producing $A$ via Algorithm \ref{algorithm1}, one can identify whether a subset $\sigma$ of $\G(I)$ is an element of $A$ and its type only by analyzing the structure of $\sigma$. The remainder of this subsection is devoted to the characterization of these types of elements in $A$.

\begin{definition}
Let $m,m'\in \G(I)$ and $\sigma$ be a subset of $\G(I)$.

\begin{enumerate}
    \item If $m>_I m'$, we say $m$ \textit{dominates} $m'$.
    \item  The monomial $m$ is called a \textit{gap} of $\sigma$ if 
    \begin{enumerate}
        \item[(a)] $m\notin \sigma$ and
        \item[(b)] $\lcm (\sigma \cup m)=\lcm (\sigma)$; in other words, $m$ is a bridge of $\sigma \cup m$.
    \end{enumerate}
    \item The monomial $m$ is called a \textit{true gap} of $\sigma$ if 
        \begin{enumerate}
        \item[(a)]  it is a gap of $\sigma$ and 
        \item[(b)]  the subset $\sigma \cup m$ has no new bridges dominated by $m$. In other words, if $m'$ is a bridge of $\sigma \cup m$ such that $ m >_I m'$, then $m'$ is a bridge of $\sigma$.
    \end{enumerate}
\end{enumerate}
\end{definition}

\begin{example}\label{ex:bgtg}
    Consider the ideal $I=(xw,xy,yz,zw)$ from Example \ref{ex:alg1} with the total ordering $xw> xy> yz> zw$.
    \begin{enumerate}
        \item[(a)] Every element of $\sigma=\{xw,xy,yz,zw\}$ is a bridge because omitting any element from this subset does not change the lcm. The smallest bridge of $\sigma$ is $\sbridge(\sigma)= zw$. It is immediate that $\sigma $ has no gap or true gap.
        \item[(b)] Consider the subset $\sigma_1=\{xw,xy,yz\}$. The only bridge of $\sigma_1$ is $xy$ and only gap of $\sigma_1$ is $zw$.  Moreover, $zw$ is a true gap of $\sigma_1$ because all the new bridges of $\sigma =\sigma_1 \cup zw$ are $xw,yz,zw$ and $zw$  is the smallest bridge among them.
        \item[(c)] Next, consider the subset $\sigma_2=\{xw,xy,zw\}$. The only bridge of $\sigma_2$ is $xw$ and   $yz$ is the only gap of $\sigma_2$. However,  $yz$ is not a true gap of $\sigma_2$ because one of the new bridges of $\sigma=\sigma_2\cup yz$ is $zw$ and  $yz> zw$.
    \end{enumerate}
\end{example}

Notions of a bridge and a gap are  \emph{almost} complementary. One can view a gap as the opposite of a bridge. However, the opposite of the smallest bridge is not the smallest gap and it is more like the smallest true gap. 

The following proposition will be used frequently in the remainder of this subsection as we will be mostly working with true gaps not dominating any bridges.

\begin{proposition}\label{prop:truegap}
    A monomial $m$ is a gap of $\sigma$ such that  $\sbridge(\sigma \cup m)=m$ if and only if $m$ is a true gap of $\sigma$ that does not dominate any bridges of $\sigma$.
\end{proposition}

\begin{proof}
Suppose $m$ is a gap of $\sigma$ such that $\sbridge(\sigma \cup m)=m$. Then $\sigma \cup m$ has no new bridges dominated by $m$. Thus, $m$ is a true gap of $\sigma$. In addition, $m$ does not dominate any bridges of $\sigma$ since $\sbridge(\sigma \cup m)=m$.

Conversely, suppose $m$ is a true gap of $\sigma$ that does not dominate any bridges of $\sigma$. Then, $m$ must be a bridge of  $\sigma \cup m$. If there exists a new bridge $m'$ of $\sigma \cup m$, we must have $m' >_I m$. Otherwise, if $m >_I m'$, then $m'$ is a bridge of $\sigma$ by the definition of true gaps, a contradiction. Therefore, all the bridges (new and old) of $\sigma \cup m$ must be $m$ or dominate $m$, i.e., $\sbridge(\sigma \cup m)=m$.
\end{proof}

\begin{proposition}\label{prop:truegap2}
Assume $m=\sbridge(\sigma)$ for some $m \in \G(I)$ and subset $\sigma$ of $\G(I)$. Then, $m$ is a true gap of $\sigma\setminus m$ that does not dominate any bridges of $\sigma \setminus m$.
\end{proposition}
 
\begin{proof}
   Observe that $m$ is a gap of $\sigma \setminus m$ because $m$ is a bridge of $\sigma$. Furthermore, $m$ is the smallest bridge of $\sigma=(\sigma \setminus m)\cup m$. Then, by Proposition \ref{prop:truegap}, $m$ is a true gap of $\sigma\setminus m$ not dominating any bridges of $\sigma\setminus m$.
\end{proof}

\begin{proposition}\label{prop:truegap3}
    If $m$ is the smallest true gap of $\sigma$ that does not dominate any bridges, then $m=\sbridge(\sigma\cup m)$ and $m$ does not dominate any true gaps of $\sigma \cup m$.
\end{proposition}
\begin{proof}
It follows from Proposition \ref{prop:truegap} that $\sbridge(\sigma \cup m)= m$. It remains to show that any true gap of $\sigma \cup m$ dominates $m$. On the contrary, suppose $m'$ is a true gap of $\sigma\cup m$ such that  $m>_I m'$. Note that $m'$ does not dominate any bridges of $\sigma \cup m$ since $m$ does not. Then,  we have
    \begin{itemize}
        \item $m'\notin \sigma\cup m$ and $\lcm(\sigma\cup \{m , m'\})=\lcm(\sigma\cup m)$ as $m'$ is a gap of $\sigma \cup m$, and
        \item $\sbridge(\sigma\cup \{m , m'\})= m'$ by Proposition \ref{prop:truegap}.  
    \end{itemize}
    
      The second item implies that any $m''\in \sigma \cup \{m , m'\}$ such that $m'>_I m''$ cannot be bridge of $\sigma \cup \{m , m'\}$. Hence, $m''$ is not a bridge  of $\sigma \cup m'$ either.  Thus, $\sbridge(\sigma \cup m')= m'$. It suffices to show that $m'$ is a gap of $\sigma$ because this implies that $m'$ is a true gap of $\sigma$  not dominating any bridges of $\sigma$ by Proposition \ref{prop:truegap} which is in contradiction with our assumption on $m$.  Indeed, since $m$ and $m'$ are gaps of $\sigma$ and $\sigma \cup m$, respectively, we have 
        \[ \lcm(\sigma)\mid \lcm(\sigma \cup m') \mid\lcm(\sigma \cup \{m , m'\}) =\lcm(\sigma \cup m)=  \lcm(\sigma). \]
        This means $\lcm(\sigma \cup m') = \lcm(\sigma)$, i.e., $m'$ is a gap of ~$\sigma$.
\end{proof}

Now we are ready to characterize type-1 and (potentially-)type-2 elements in $A$.

\begin{theorem}\label{thm:alltypes}
Let $\sigma$ be a subset of $\G(I)$. Then
   \begin{enumerate}
        \item[(a)] $\sigma$ is type-1 if and only if it has a true gap not dominating any bridges. 
        \item[(b)] $\sigma$ is potentially-type-2  if and only if  it has a bridge not dominating any true gaps. 
        \item[(c)] $\sigma$ is type-2  if and only if  
        \begin{enumerate}
            \item[(i)] it has a bridge that does not dominate any true gaps and
            \item[(ii)] whenever there exists a potentially-type-2  subset $\tau$ of $\G(I)$  where $\sigma \setminus \sbridge(\sigma)=\tau \setminus \sbridge(\tau)$, we have $\sbridge(\tau) \geq_I \sbridge(\sigma)$.
        \end{enumerate}
   \end{enumerate} 
\end{theorem}

\begin{proof}
Let $n=|\sigma|$, the cardinality of $\sigma$. We prove these statements using descending induction on $n$. If $n=|\G(I)|$, then $\sigma = \G(I)$. It is clear from the definition that $\sigma$ does not have any gaps. If $\sigma$ has a bridge, then $(\sigma, \sigma \setminus \sbridge(\sigma)) \in A$ by Algorithm \ref{algorithm1}. Thus, $\sigma$ is type-2. Otherwise, $A= \emptyset$ and the Taylor resolution is minimal. Hence, the theorem holds for the base case $n=|\G(I)|$. Suppose the statement of the theorem holds for all subsets of $\G(I)$  of cardinality $n+1$. 

\begin{enumerate}
        \item[(a)] If $\sigma$ is type-1, then there is a directed edge $(\tau, \sigma) \in A$ where $\sigma = \tau \setminus \sbridge(\tau)$. It follows from Proposition \ref{prop:truegap2} (a) that $\sbridge(\tau)$ is a true gap of $\sigma$ not dominating any bridges of $\sigma$. For the reverse implication, suppose $\sigma$ has a true gap which does not dominate any bridges of $\sigma$ and $m$ is its smallest true gap. Then $\sbridge(\sigma\cup m)=m$ and $m$ does not dominate any true gaps of $\sigma \cup m$ by Proposition \ref{prop:truegap3}. Since $|\sigma\cup m|=n+1$, the subset $\sigma \cup m$ is potentially-type-2 by the induction hypothesis. This guarantees the existence of a directed edge towards $\sigma$ in $A$. Thus, $\sigma$ is type-1.
        
        \item[(b)] If $\sigma$ is potentially-type-2 , then $\sigma$ has a bridge and it is not type-1 by step (2) of Algorithm \ref{algorithm1}. If $\sigma$ has a true gap, it dominates a bridge of $\sigma$ by part (a). Then, $\sbridge(\sigma)$ does not dominate any true gap of $\sigma$. If $\sigma$ has no true gaps, then the statement is immediate. For the reverse direction, suppose $\sigma$ has a bridge not dominating any true gaps of $\sigma$. Note that $\sigma$ is not type-1 because any true gap of $\sigma$ dominates $\sbridge(\sigma)$ by part (a). Then, the directed edge $(\sigma, \sigma \setminus \sbridge(\sigma))$ is added to $A$ in step (2) of Algorithm \ref{algorithm1} and thus $\sigma$ is potentially-type-2. 
        \item[(c)] This statement follows from part (b) and step (3) of Algorithm \ref{algorithm1}.  \qedhere 
\end{enumerate}
\end{proof}

\begin{example}\label{ex:cattypes}
  Let $I$ be the ideal from Example \ref{ex:alg1}. In Example \ref{ex:bgtg}, we discussed bridges, gaps and true gaps for some subsets of $\G(I)$. In this example, we determine the types of the subsets  considered in Example \ref{ex:bgtg} using our characterization from Theorem \ref{thm:alltypes}. 
  
    \begin{enumerate}
        \item[(a)] The subset $\sigma=\{xw,xy,yz,zw\}$ has a bridge but it has no true gaps. Therefore, it is potentially-type-2 and it is, in fact, clear that $\sigma$ is type-2.
        \item[(b)] The only bridge of $\sigma=\{xw,xy,yz\}$ is $xy$ and $zw$ is its only true gap. Since $xy$ dominates $zw$, $\sigma$ is type-1. 
        \item[(c)] The subset $\sigma=\{xw,xy,zw\}$ has only one bridge $xw$ and it has no true gaps. Hence, $\sigma$ is  potentially-type-2. However, $\sigma$ is not type-2 since there exists $\tau=\{xy,yz,zw\}$ such that $\sigma \setminus \sbridge(\sigma)=\tau \setminus \sbridge(\tau)$ and $\sbridge(\sigma) > \sbridge(\tau) $.
    \end{enumerate}
\end{example}

Given a type-1 element $\sigma$ in $A$, one can determine which $m \in \G(I)$ is added back to $\sigma$ so that $\sigma \cup m$ is the associated type-2 element, i.e., $(\sigma \cup m, \sigma) \in A$.  We will describe those elements in the following remark while providing an explanation for why a potentially-type-2 element is not necessarily type-2 using the notions of bridges, gaps, and true gaps.

\begin{remark}
Consider a subset $\sigma$ of $\G(I)$. 

\begin{itemize}
    \item Suppose $\sigma$ has a true gap not dominating any bridges, i.e., $\sigma$ is type-1. Let $m$ be the smallest true gap of $\sigma$. Then, $m$ is the smallest bridge of $\sigma \cup m$ and it does not dominate any true gaps. Furthermore, $\sigma \cup m$  is type-2 since $(\sigma \cup m, \sigma) \in A$. 
    \item Suppose $\sigma$ has a bridge not dominating any true gaps, i.e., $\sigma$ is potentially-type-2. Let $\sbridge(\sigma)=m$. One may expect $m$ to be the smallest true gap of $\sigma \setminus \sbridge(\sigma)$. Unfortunately, this is not necessarily the case because a gap of $\sigma$ smaller than $m$ may become a true gap of $\sigma \setminus m$. To see this, consider the subset $\sigma$ from part (c) of Example \ref{ex:cattypes}. Notice that $\sbridge(\sigma)= xw$ and the smallest true gap of $ \sigma \setminus xw= \{xy, zw\}$ is $yz$ which was a gap of $\sigma$ with $xw> yz$.

\end{itemize}
\end{remark}

As we see in the statement of Theorem \ref{thm:alltypes}, the classification of  potentially-type-2 elements in terms of bridges and true gaps is relatively simpler than that of type-2 elements (assuming one can identify bridges and true gaps). In what follows, we introduce a class of ideals whose type-2 and potentially-type-2 elements coincide.

\begin{definition}
   A monomial ideal $I$ is called \textit{bridge-friendly} if there exists a total ordering $(>_I)$ on $\G(I)$ such that all potentially-type-2 subsets of $\G(I)$  are type-2. In other words, $I$ is bridge-friendly  if a subset $\sigma$ of $\G(I)$ is type-2 exactly when it has a bridge not dominating any  true gaps.
\end{definition}

The critical subsets associated with bridge-friendly ideals can be characterized in a nice way. This follows from Theorem \ref{thm:alltypes}.

\begin{corollary}\label{cor:critical}
    If $I$ is bridge-friendly, then the critical subsets of $\G(I)$  are exactly  the  ones with no bridges and no true gaps.
\end{corollary}

This nice characterization of critical subsets of bridge-friendly ideals allows us to understand their minimal free resolutions.

\begin{theorem}\label{thm:friendlyToMinimal}
If $I$ is bridge-friendly with respect to $(>_I)$, then the corresponding Barile-Macchia resolution of $R/I$ is minimal.
\end{theorem}

\begin{proof}
	Suppose the corresponding Barile-Macchia resolution of  $R/I$ is not minimal. Then, by Theorem \ref{thm:morseres}, there exist a critical subset $\sigma$ and its subset $\sigma'$ where $|\sigma'|=|\sigma|-1$ and there is a gradient path from $\sigma'$ to $\sigma''$ for a critical subset $\sigma''$ with $|\sigma'|=|\sigma''|$ such that $\lcm(\sigma)=\lcm(\sigma'')$. By similar arguments as in the proof of Claim \ref{clm:directedcycle}, $\lcm$'s of all the subsets along this gradient path must coincide. In particular, we have $\lcm(\sigma')=\lcm(\sigma'')=\lcm(\sigma)$. Thus $\sigma$ has a bridge, contradicting Corollary ~\ref{cor:critical}.
\end{proof}

As discussed earlier, one can obtain the minimality of Morse resolutions via bridge-friendliness or Batzies-Walker matchings. A natural question is whether these two notions are related. In the following two examples, we provide examples to show that one notion does not imply the other.

\begin{example}
	Consider the monomial ideal $I=(xy,yz,zx)$ in $R=\Bbbk[x,y,z]$ with a total ordering $xy>yz>zx$. The corresponding Barile-Macchia matching of $R/I$ is
	\[
	A=\{(\{xy,yz,zx\},\{xy,yz\})\}.
	\]
	Since there is only one subset to consider in Algorithm \ref{algorithm1}, ideal $I$ is bridge-friendly.  Note that $I$ is bridge-friendly with respect to any total ordering. On the other hand, $A$ is not Batzies-Welker because $\{xy,zx\}$ and $\{yz,zx\}$ are critical and they have the same $\lcm$. More generally, this implies that none of the lcm-homogeneous acyclic mathcings of $R/I$ is Batzies-Welker.
\end{example}

\begin{example}\label{notmorseminimalexample}
Let $I$ be the ideal from Example \ref{ex:alg1}. It follows from Example \ref{ex:types} that $I$ is not bridge-friendly with respect to the given total ordering because  the subset $\{xw, xy,zw\}$ is potentially-type-2 but not type-2.  Note that the corresponding matching is Batzies-Welker. 

In this example, we use Corollary \ref{cor:critical} to prove that $I$ is not bridge-friendly with respect to any total ordering on $\G(I)$. For the sake of contradiction, suppose $I$ is bridge-friendly. Let $m_1=xw, ~m_2=xy, ~m_3=yz, ~m_4=zw$ and we will take indices modulo $4$. Note that a subset $\sigma$ of $\G(I)$ of cardinality $3$ must contain three consecutive elements $m_i, m_{i+1},$ and $m_{i+2}$ for some $1\leq i\leq 4$. In particular, such $\sigma$ has a bridge $m_{i+1}$ and hence not critical by Corollary \ref{cor:critical}. Then the maximum cardinality of a critical subset is 2 and thus $\pd_R(R/I)\leq 2$, which is a contradiction because  $\pd_R(R/I)=3$ from Example \ref{ex:alg1}.  Therefore, $I$ is not bridge-friendly.
\end{example}

In fact, this example can be generalized to a $(3n+1)$-cycle. 

\begin{proposition}\label{ex:notfriendlycycle}
  Let $I$ be the edge ideal of a $(3n+1)$-cycle for $n\geq 1$, i.e.,
    \[I=(x_1x_2,x_2x_3,\dots, x_{3n}x_{3n+1},x_{3n+1}x_1).\] 
    Then $I$ is not bridge-friendly.
\end{proposition}

\begin{proof}
    For the sake of contradiction, suppose $I$ is bridge-friendly. We consider indices modulo $(3n+1)$ in this proof. Let $m_i=x_ix_{i+1}$ for each $i\in [3n+1]$. One can show that any subset $\sigma$ of $\G(I)$  of cardinality at least $2n+1$ contains  $m_j,m_{j+1},m_{j+2}$ for some $j \in[3n+1]$. Hence $\sigma$ has a bridge, which, in particular, implies that $\sigma$ is not critical since $I$ is  bridge-friendly.  Thus, $\pd_R(R/I)\leq 2n$, which contradicts the fact that $\pd_R(R/I)=2n+1$ (\cite[Proposition 5.0.6]{Bou10}). Therefore, $I$ is not bridge-friendly.
\end{proof}

In the next sections, we show that many interesting classes of monomial ideals are bridge-friendly. In what follows, we provide an equivalent condition to check the bridge-friendliness of an ideal which will prove to be useful later.

\begin{lemma}\label{lem:morsefri}
     A monomial ideal $I$ is bridge-friendly with respect to $(>_I)$ if and only if there exists no monomial $m \in \G(I)$ such that $m$ is a true gap of $\sigma \setminus \sbridge(\sigma)$ and $\sbridge(\sigma) >_I m$ for any potentially-type-2  element $\sigma$.
\end{lemma}

\begin{proof}
   We will prove both directions with contraposition. For the forward implication, suppose that there exists a monomial $m$ and a potentially-type-2  subset $\sigma$  of $\G(I)$  such that $m$ is a true gap of $\sigma \setminus \sbridge(\sigma)$ and  $\sbridge(\sigma) >_I m$. Then $m$ is not a true gap of $\sigma$ by Theorem \ref{thm:alltypes}. We may  assume $m$ is the smallest such monomial. Let $\tau= \sigma \setminus \sbridge(\sigma)$. Then, $m$ is the smallest true gap of $\tau$ and it does not dominate any bridges of $\tau$. It follows from Proposition \ref{prop:truegap3} that $m=\sbridge(\tau \cup m)$ and it does not dominate any true gaps of  $\tau\cup m$. Thus, $\tau\cup m$ is potentially-type-2 by Theorem \ref{thm:alltypes}. Since 
   \[
   \sigma \setminus \sbridge(\sigma) = (\tau\cup m)\setminus\sbridge(\tau\cup m),
   \]
    the ideal $I$ is not bridge-friendly. 
    
	For the reverse implication, suppose $I$ is not bridge-friendly with respect to $(>_I)$. Then
	there exist two different potentially-type-2 subsets $\sigma$ and $\tau$ such that  $\sigma\setminus \sbridge(\sigma)=\tau\setminus \sbridge(\tau)$. Without loss of generality, we assume that $\sbridge(\tau) >_I \sbridge(\sigma)$. Then $\sbridge(\sigma)$ is a true gap of $\sigma\setminus \sbridge(\sigma)=\tau\setminus \sbridge(\tau)$ not dominating any of its bridges by Proposition \ref{prop:truegap2}, which completes the proof.
\end{proof}

\begin{lemma}\label{lem:morsefri2}
Let $I$ be a monomial ideal  and $(>_I)$ a total ordering on $\G(I)$. If $m$ is a true gap of $\sigma\cup m'$ whenever $m>_I m'$ and $m$ is a true gap of $\sigma$, then  $I$ is bridge-friendly with respect to $(>_I)$.
\end{lemma}

\begin{proof}
Let $\sigma$ be a potentially-type-2 subset  of $\G(I)$. If there exists a true gap $m$ of $\sigma\setminus \sbridge(\sigma)$ such that $\sbridge(\sigma)>_I m$, then $m$ is a true gap of $\sigma$ by the assumption of this lemma, contradicting Theorem \ref{thm:alltypes}. Therefore, $I$ is bridge-friendly from  Lemma \ref{lem:morsefri}. 
\end{proof}

We conclude this section with the following result:

\begin{theorem}\label{thm:tensormorse}
    If $I$ and $J$ are bridge-friendly ideals of $R=\Bbbk[x_1,\ldots, x_r]$ and $S=\Bbbk[y_1,\ldots, y_s]$, respectively. Then $I+J$ is a bridge-friendly ideal of $R\otimes_\Bbbk S$.
\end{theorem}

\begin{proof}
   Suppose $I$ and $J$ are bridge-friendly ideals with respect to the total orderings $(>_I)$ on $\G(I)$ and $(>_J)$ on $\G(J)$, respectively. Let $(>)$ be the total ordering on $\G(I+J)=\G(I)\cup \G(J)$ defined as follows:    for all $f,f'\in \G(I)$ and $g,g' \in \G(J)$, 
   \begin{itemize}
       \item $f>f'$ whenever $f>_I f',$
       \item $g>g'$ whenever $g>_J g'$, and
       \item $f>g$.
   \end{itemize}

One can decompose any subset $\sigma$ of $I+J$ uniquely as $\sigma=\sigma_I \sqcup \sigma_J$ where $\sigma_I$ and $\sigma_J$ are subsets of $\G(I)$ and $\G(J)$, respectively. Note that a monomial $m\in \G(I)$ is a bridge/gap/true gap of $\sigma$ if and only if $m$ is a bridge/gap/true gap of $\sigma_I$.  The same statement holds for $m\in \G(J)$. With these observations, we can conclude  that $I+J$ is bridge-friendly and bridge-minimal with respect to $(>)$.
\end{proof}

\begin{remark}
    A direct corollary of Theorem \ref{thm:tensormorse} is that if $\mathcal{F}_A$ and $\mathcal{F}_B$ are the (minimal) Barile-Macchia resolutions of $R/I$ and $R/J$, respectively, then $\mathcal{F}_A \otimes_\Bbbk \mathcal{F}_B$ is a (minimal) Barile-Macchia resolution of $(R \otimes_\Bbbk S)/(I+J)$.
\end{remark}

\section{Minimal Free Resolutions of Edge Ideals of Weighted Oriented Forests}\label{sec:forests}

The main objects of this section are edge ideals of weighted oriented forests. In particular, we will show that these ideals are bridge-friendly. We start by  introducing some of the fundamental definitions.

Let $\D$ be a weighted oriented graph. Abusing notation, we  write $xy \in E(\D)$ for the directed edge between the vertices $x$ and $y$ without specifying the orientation. We will use the notation $m_{xy}$ to denote the monomial associated to this directed edge (without the orientation information). We denote the underlying unweighted unoriented graph of $\D$ by $G_{\D}$. The vertices of $\D$ and $G_{\D}$ are the same while the edges of $G_{\D}$  have no orientation. We use $\{x,y\}$ to denote the unoriented edge between the vertices $x$ and $y$, and $(x,y)$ to denote the edge oriented from vertex $x$ to vertex $y$. In our setting, all the underlying graphs are finite and simple, i.e., no loops and no multiple edges are allowed. 

\begin{definition}\label{def:naturallyoriented}
Let $\T=(V(\T),E(\T),\w)$ be a weighted oriented rooted tree with a root vertex $x_0$. Note that $G_{\T}$ has no cycles since it is a tree.  We call $\T$ a \emph{weighted naturally oriented tree} if each edge of $\T$ is oriented towards the vertex that is further away from $x_0$. A \textit{weighted naturally oriented forest} is a disjoint union of weighted naturally oriented trees. 
\end{definition}

Due to Theorem \ref{thm:tensormorse}, it suffices to consider trees when dealing with Barile-Macchia resolutions of weighted oriented forests. For the remainder of this section, we assume $\T$ is a weighted naturally oriented tree with a root vertex $x_0$. 

Let $\rank x$ denote the distance between $x\in V(\T)$ and the root vertex, and $V(\T)_d$ be the  collection of all vertices of rank $d$ where $V(\T)_d:=\{x^{(d)}_{1}, \ldots, x^{(d)}_{n_d}\}$. For the remainder of this section, we consider the following  variable ordering in $R$ based on vertex ranks:
\[ x_0>x^{(1)}_{1}>\cdots >x^{(1)}_{n_1} >x^{(2)}_{1}>\cdots>x^{(2)}_{n_2}>\cdots\]
Note that vertices of the same rank are ordered based on their labels.  We order the edges of $\T$ with respect to this variable ordering: For two edges $xy$ and $zw$ in $\T$, we write $xy > zw$ if $\rank(x) < \rank(z)$ or if $x$ and $z$ are of the same rank while $\rank(y)<\rank(w)$. Let $(>_I)$ denote the total ordering on $\G(I(\T))$ defined as follows:  $ m_{xy}>_I m_{zw} $ if $ xy> zw$  where $xy,zw\in E(\T)$. To ease the notation, we write $\G(\D)$ for $\G(I(\D))$ whenever $\D$ is a (weighted oriented) graph.

\begin{definition}
   A vertex $x\in V(\T)$ is called a \textit{predecessor} of $z\in V(\T)$ if  $(x,z)\in E(\T)$.
\end{definition}

\begin{remark}\label{rmk:forest}
A subtle detail particular to trees is that each vertex (except the root)  has a unique predecessor. In particular, if there are  two different edges $xy,xz\in E(\T)$ with $(y,x)\in E(\T)$, we have $(x,z) \in E(\T)$.  
\end{remark}

\begin{example}\label{exampleforest}
    Consider the weighted naturally oriented tree $\T$ given in Figure \ref{fig:forest0}.
    \begin{figure}[H]
    \centering
    \includegraphics[width=0.3\textwidth]{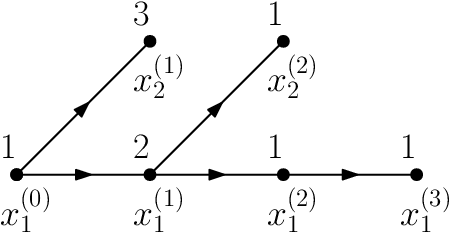}
    \caption{A weighted naturally oriented tree}
    \label{fig:forest0}
\end{figure}

   \noindent Then the ideal $I(\T)$ is generated by the following monomials:

    \[\small{m_{x_0x^{(1)}_{1}}=x_0(x^{(1)}_{1})^2,~~ m_{x_0x^{(1)}_{2}}=x_0(x^{(1)}_{2})^3,~~ m_{x^{(1)}_{1}x^{(2)}_{1}}=x^{(1)}_{1}x^{(2)}_{1},
    }\]
    \vskip0.05cm
    \[\small{ m_{x^{(1)}_{1}x^{(2)}_{2}}=x^{(1)}_{1}x^{(2)}_{2},~~ m_{x^{(2)}_{1}x^{(3)}_{1}}=x^{(2)}_{1}x^{(3)}_{1}}.   \]
   The total ordering on the generators of $I(\T)$ is 
    \[m_{x_0x^{(1)}_{1}}>_I m_{x_0x^{(1)}_{2}}>_I m_{x^{(1)}_{1}x^{(2)}_{1}}>_I m_{x^{(1)}_{1}x^{(2)}_{2}}>_I m_{x^{(2)}_{1}x^{(3)}_{1}}.\]
    Note that degrees of the generators do not play any role in this total ordering. 
\end{example}

We next introduce the concept of blocks which will be useful in characterizing bridges, gaps, and true gaps of our ideals. Let $P_{ee'}$ denote the unique path between two distinct edges $e, e' \in E(\T)$ which admits $e$ and $e'$ as its first and last edges. Inheriting the orientations and weights from $\T$, $P_{ee'}$ is a weighted oriented path and is called an \textit{induced path} of $\T$.  It is immediate that $\G(P_{ee'})\subseteq \G(\T)$. Set $E(P_{ee'})=\{x_1x_2,x_2x_3,\dots, x_nx_{n+1}\}$ and denote the corresponding monomials by $m_1,\ldots, m_n$ where $m_i=m_{x_ix_{i+1}}$ for $i \in [n]$. The monomials $m_{1}$ and $m_{n}$ are called the \textit{ends} of $\G(P_{ee'})$. Note that edges of $P_{ee'}$ are not necessarily oriented  in the same direction.

\begin{definition}\label{blockdefinitionforest}
   Let $e,e'\in E(\T)$ be two distinct edges. The set $\G(P_{ee'})=\{m_1,\ldots, m_n\}$  is called a \textit{potential block} if
    \[m_p \mid \lcm(m_{p-1},m_{p+1})
    \]
    for any $2\leq p\leq n-1$.  If $\G(P_{ee'})$  is maximal (with respect to inclusion) among  all potential blocks, then it is called a \textit{block} and  its ends $m_1$ and $m_n$ are called \textit{blockends}.
\end{definition}

\begin{example}\label{exampleforest2}
    Consider the weighted oriented tree $\T$ from  Example \ref{exampleforest}. Recall that each block comes from a path between two vertices of $\T$ and the monomials are written in the order of the successive edges along the corresponding path. Then the blocks of $\T$ and their blockends (underlined) are given as follows:
    \begin{itemize}
        \item $P_1= \{\underline{m_{x_0x^{(1)}_{1}}},\underline{ m_{x_0x^{(1)}_{2}}}\}$
        \item $P_2= \{\underline{m_{x_0x^{(1)}_{1}}}, \underline{m_{x^{(1)}_{1}x^{(2)}_{2}}}\}$
        \item $P_3= \{\underline{m_{x_0x^{(1)}_{1}}}, m_{x^{(1)}_{1}x^{(2)}_{1}}, \underline{m_{x^{(2)}_{1}x^{(3)}_{1}}}\}$
        \item $P_4= \{\underline{m_{x^{(1)}_{1}x^{(2)}_{2}}}, m_{x^{(1)}_{1}x^{(2)}_{1}}, \underline{m_{x^{(2)}_{1}x^{(3)}_{1}}}\}$
    \end{itemize}
 The set  $\{m_{x^{(1)}_{1}x^{(2)}_{2}}, m_{x^{(1)}_{1}x^{(2)}_{1}}\}$ is  a potential block but not a block since it is contained in the fourth block $P_4$. 
\end{example}

\begin{definition}\label{inthesameblock}
   Let $e_1,e_2,\ldots, e_n \in E(\T)$. We say $m_{e_1}, m_{e_2},\dots, m_{e_n}$ are \textit{in the same block} if there exists a block that contains all of them. There may be multiple blocks containing a given collection of elements of $\G(\T)$.

   Let $e_1,e_2,\ldots, e_n \in E(\T)$ such that $m_{e_1}, \dots, m_{e_n}$ are in the same block, say $\G(P)$. Suppose we have another edge $e_{n+1} \in E(\T)$. We say \textit{$m_{e_{n+1}}$ \textit{is in the same block as} $m_{e_1}, \dots, m_{e_n}$} if there exists a block, say $\G(P')$, that contains all of them. Note that $P$ and $P'$ share directed edges but they are not necessarily the same. 
\end{definition}

We are now ready to characterize bridges, gaps, and true gaps of $I(\T)$.

\begin{proposition}\label{prop:bridgeforest}
Let $\sigma$ be a subset  of $\G(\T)$. Consider an edge $(x,z)\in E(\T)$. The monomial $m_{xz}$ is 
    \begin{enumerate}
        \item[(a)] a bridge of $\sigma$ iff there exist edges $xy,zw\in E(\T)$ such that $m_{xy}, m_{xz}, m_{zw}\in \sigma$ are in the same block.
        \item[(b)]  a gap of $\sigma$ iff there exist edges $ xy,zw\in E(\T)$ such that $m_{xy}, m_{zw}\in \sigma $ and $m_{xz}\notin \sigma$ are in the same block.
        \item[(c)] a true gap of $\sigma$ iff it is a gap and the following conditions hold for each pair $ xy,zw\in E(\T)$ satisfying the condition in (b): 
        \begin{enumerate}
        \item[(i)] Either
        \begin{itemize}
            \item the monomial $m_{zw}$ is the only monomial in $\sigma$ that is in the same block as $m_{xy}, m_{xz}, m_{zw}$ and  divisible by $w$, or 
            \item there exist $m_{zz'}, m_{ww'}\in \sigma$ such that   $m_{zz'},m_{ww'}, m_{zw}$ are in the same block. Here $zz', ww', zw$ are distinct.
        \end{itemize}
        \item[(ii)] If $z>y$, then either
        \begin{itemize}
            \item the monomial $m_{xy}$ is the only monomial in $\sigma$ that is in the same block as  $m_{xy}, m_{xz}, m_{zw}$ and divisible by $y$, or 
            \item there exist $m_{xx'}, m_{yy'}\in \sigma$ such that   $m_{xx'},m_{yy'}, m_{xy}$ are in the same block. Here $xx', yy', xy$ are distinct.
        \end{itemize}
    \end{enumerate}
    \end{enumerate}
    Note that $xy, xz, zw$ are distinct in all the statements.
\end{proposition}

\begin{remark}
    Note that $xz>zw$ in our set-up. On the other hand, it is possible to have $xz > xy$ or  $xy> xz$ as it is shown in Figure \ref{fig:forest2}. For the remainder of this section, we will use this convention for the edges $xy,xz,zw$.
\end{remark}

\begin{figure}[ht]
    \centering
    \includegraphics[width=0.65\textwidth]{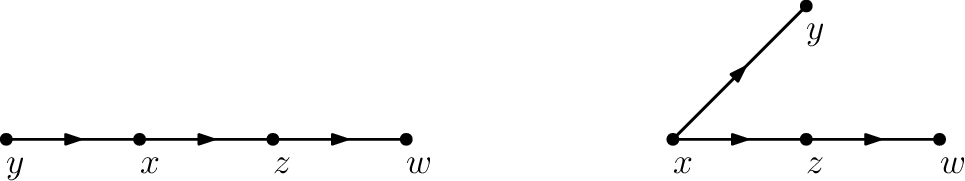}
    \caption{Two possibilities for the edge $xy$}
    \label{fig:forest2}
\end{figure}

\begin{proof}
$(a)$ Assume $m_{xz}$ is a bridge of $\sigma$. By definition, we have $m_{xz} \mid \lcm(\sigma\setminus m_{xz})$. Then there exist $y,w\in V(\T)$ such that $m_{xy}, m_{zw}\in \sigma$ and $m_{xz}\mid \lcm(m_{xy},m_{zw})$. By definition, $\{m_{xy},m_{xz},m_{zw}\}$ is a potential block and thus all these monomials are in the same block. The other direction is immediate from the definition of a bridge. One can prove $(b)$ using similar arguments.

Before moving on to part $(c)$, we prove the following claim:

\begin{claim}\label{clm:newbridge}
    Consider an edge $(x,z)\in E(\T)$ and a subset $\sigma$  of $\G(\T)$. Assume $m_{xz}$ is a gap of $\sigma$, i.e., there exist  edges $ xy,zw\in E(\T)$ such that $m_{xy}, m_{zw}\in \sigma $ and $m_{xz}\notin \sigma$ are in the same block. Moreover, assume that $\sigma \cup m_{xz}$ has a new bridge $m$ dominated by $m_{xz}$. Then $m$ is either $m_{xy}$ or $m_{zw}$. Note that  the first case implies that $z>y$.
\end{claim}

\begin{proof}[Proof of Claim \ref{clm:newbridge}]
    We prove this statement by contradiction. Suppose $\sigma \cup m_{xz}$ has a new bridge $m$ dominated by $m_{xz}$ such that  $m\neq m_{xy}, m_{zw}$.  Observe that $m$ must be of the form $m_{zw'}$ or $m_{xz'}$ for some $w',z'\in V(\T)$ where $w' \neq w$ and $z'<z$. We consider these two cases separately and obtain a contradiction by showing that $m$ must be a bridge of $\sigma$ for each situation. 
    
    Suppose $m=m_{zw'}$ where $w'\neq w$. Since $m$ is a bridge of $\sigma \cup m_{xz}$, we have $m_{zw'}\in \sigma$ and there exists $m_{w'u}\in \sigma$ such that $m_{xz}, m_{zw'}, m_{w'u}$ are in the same block  by Proposition \ref{prop:bridgeforest} $(a)$. Hence
    \[m_{zw'} \mid \lcm(m_{xz},m_{w'u}) \mid \lcm(m_{xy},m_{zw},m_{w'u}), \]
    which implies $m_{zw'} \mid \lcm(m_{zw},m_{w'u})$. Therefore $m_{zw'}$ is a bridge of $\sigma$.

    Suppose $m=m_{xz'}$ where $z'<z$. Since $m$ is a bridge of $\sigma\cup m_{xz}$, we have $m_{xz'}\in \sigma$ and there exists $m_{z'w'}\in \sigma$ such that $m_{xz}, m_{xz'}, m_{z'w'}$ are in the same block  by Proposition \ref{prop:bridgeforest} $(a)$. Hence
    \[m_{xz'} \mid \lcm(m_{xz},m_{z'w'}) \mid \lcm(m_{xy},m_{zw},m_{w'u}), \]
    which implies $m_{xz'} \mid \lcm(m_{xy},m_{z'w'})$. Therefore $m_{xz'}$ is a bridge of $\sigma$.
\end{proof}

$(c)$ We start by considering the forward direction.  Assume $m_{xz}$ is a gap of $\sigma$ and either $(i)$ or $(ii)$ fails. Our goal is to show $m_{xz}$ cannot be a true gap of $\sigma$ under these assumptions. First, suppose $(i)$ fails. Then, there exist $ m_{ww'}\in \sigma$ that is in the same block as $m_{xy},m_{xz},m_{zw}$ and there  is no $ m_{zz'} \in \sigma$ such that $m_{zz'},m_{ww'}, m_{zw}$ are in the same block where $zw \neq zz', ww'$.  The second condition guarantees that $ m_{zw}$ is not a bridge of $\sigma$ while the first condition implies that it is a bridge of $\sigma\cup m_{xz}$. Since $ m_{xz} $ dominates $ m_{zw}$, it is not a true gap of $\sigma$ by definition. Similar arguments apply when $(ii)$ fails.

Now we prove the reverse direction. Assume $m_{xz}$ is a gap such that $(i)$ and $(ii)$ hold for each pair $xy,zw \in E(\T)$ satisfying the condition in $(b)$. By definition of true gaps, we need to show that $\sigma\cup m_{xz}$ does not have any new  bridges dominated by $m_{xz}$.  For the sake of contradiction, suppose $\sigma\cup m_{xz}$  has such a bridge $m$. Then $m$ must be $m_{zw}$ or $m_{xy}$ where $z>y$ by Claim \ref{clm:newbridge}. Suppose $m=m_{zw}$. Due to $(i)$, we have two cases. We first present the statement of each case and then obtain a contradiction based on the statement.
        \begin{itemize}
            \item \emph{The monomial $m_{zw}$ is the only monomial in $\sigma$ that is in the same block as $m_{xy}, m_{xz}, m_{zw}$ and is divisible by $w$.} Since $m_{zw}$ is a bridge of $\sigma \cup m_{xz}$ but not a bridge of $\sigma$, each pair satisfying (a) for $m_{zw}$ must involve $m_{xz}$. Thus there exists $wu \in E(\T)$ such that  $m_{wu}\in \sigma$ and $m_{xz},m_{zw},m_{wu}$ are in the same block. Recall that $m_{xy},m_{xz},m_{zw}$ are all in the same block as $m_{xz}$ is a gap of $\sigma$. Thus, $m_{xy},m_{xz},m_{zw},m_{wu}$ form a potential block. Therefore $m_{wu}$ is in the same block as as $m_{xy}, m_{xz}, m_{zw}$ and it is divisible by $w$. This is in contradiction to the statement of this case.
            \item \emph{There exist $ m_{zz'}, m_{ww'}\in \sigma$ such that $m_{zz'}, m_{ww'}, m_{zw}$ are in the same block where $zw\neq zz', ww'$.} Then $m_{zw}$ is a bridge of $\sigma$, a contradiction.
        \end{itemize}
Similar arguments apply for the case $m=m_{xy}$ where $z>y$ using $(ii)$ instead of $(i)$. 
\end{proof}

In particular, if a monomial $m$ is a true gap of some subset $\sigma$  of $\G(\T)$, then $m$ is still a true gap of $\sigma\cup m'$ as long as $m$ dominates $m'$. Thus the main theorem of this section follows immediately from Lemma \ref{lem:morsefri2}.

\begin{theorem}\label{thm:morsefriendlyforest}
    The ideal $I(\T)$ is bridge-friendly.
\end{theorem}


%

%
Next we will compute all the multi-graded Betti numbers of $R/I(\T)$. Specifically, we will show that the Barile-Macchia matching considered in this section is Batzies-Welker. First we provide a few auxiliary lemmas which will be used in the proof. 

\begin{lemma} (Forest-Bridge Lemma)
    \label{bridgelemma1}
    Let $\sigma$ be a subset  of $\G(\T)$ and $(x,z) in E(\T)$.  Assume $m_{xy},m_{xz}\in \sigma$ and $\sigma$ does not have any bridge. If  $m_{zw}$ is in the same block as $m_{xy}, m_{xz}$ where $(z,w)\in E(\T)$, then $m_{zw}\notin \sigma$.
\end{lemma}

\begin{proof}
If $m_{zw}$ is in the same block as $m_{xy}, m_{xz}$ where $(z,w)\in E(\T)$, then $m_{xz}\mid \lcm(m_{xy},m_{zw})$ by definition. Since $\sigma$ has no bridges, we must have $m_{zw} \notin \sigma$.
\end{proof}

\begin{lemma} (Forest-True-Gap Lemma)
    \label{truegaplemma1}
     Let $\sigma$ be a subset  of $\G(\T)$ and $(x,z) \in E(\T)$. Let $m_{xz}$ be a gap of $\sigma$.  Assume that $\sigma$ has no true gaps.   Then either of the following is true for any pair $xy,zw\in E(\T)$ satisfying the condition in Proposition \ref{prop:bridgeforest} (b):
\begin{enumerate}
        \item[(i)] There exists $(w,w') \in E(\T)$ such that $m_{ww'}\in \sigma$ and it is in the same block as $m_{xy},m_{xz},m_{zw}$. Moreover, there exists no $m_{zz'} \in \sigma$ such that $m_{zz'}, m_{ww'}, m_{zw}$ are in the same block. Here $zz',ww',zw$ are distinct.
        \item[(ii)] If  $z>y$, there exists $(y,y') \in E(\T)$ such that $m_{yy'}\in \sigma$ and it is in the same block as $m_{xy},m_{xz},m_{zw}$. Moreover, there exists no $m_{xx'} \in \sigma$ such that $m_{xx'}, m_{yy'}, m_{xy}$ are in the same block. Here $xx',yy', xy$ are distinct.
\end{enumerate} 
\end{lemma}
\begin{proof}
   It follows from Proposition \ref{prop:bridgeforest} $(c)$.
\end{proof}

\begin{theorem}\label{thm:morseminimalforest}
    The ideal $I(\T)$ has a Batzies-Welker matching. In particular, if  there exists $\sigma\subseteq \G(\T)$ such that $\sigma$ has no bridges or true gaps where $|\sigma|=r$ and $\lcm (\sigma) =\mathbf{x}^\mathbf{a}$, then $\beta_{r,\mathbf{a}}(R/I(\T))=1$. Otherwise, $\beta_{r,\mathbf{a}}(R/I(\T))=0$.
\end{theorem}
\begin{proof}
    Let $\sigma$ and $\sigma'$ be  different critical subsets of $\G(\T)$ with respect to the total ordering $(>_I)$. Recall that $I(\T)$ is bridge-friendly by Theorem \ref{thm:morsefriendlyforest}. Then  $\sigma$ and $\sigma'$ have no bridges and no true gaps by Corollary \ref{cor:critical}. 
    
    Since $\sigma$ and $\sigma'$ are different, there exists  some $(x,z)\in E(\T)$ such that $m_{xz}\in \sigma$ but $m_{xz}\notin \sigma'$. We say that an edge  $e$ satisfies the  \textbf{separation condition} if $m_e \in \sigma$ and $m_e \notin \sigma'$. Let  $xz$ be the smallest edge satisfying the separation condition. Suppose  $\lcm(\sigma)=\lcm(\sigma')$ for the sake of contradiction. Then $m_{xz}$ is a gap of $\sigma'$ because $m_{xz} \mid \lcm(\sigma')$. It then follows from  Proposition \ref{prop:bridgeforest} $(b)$ that there exist $xy,zw\in E(\T)$ such that monomials $m_{xy},m_{xz},m_{zw}$ are in the same block and $m_{xy},m_{zw}\in \sigma'$. Since $\sigma'$ has no true gaps, we have the following scenarios by  \hyperref[truegaplemma1]{Forest-True-Gap Lemma}:
    \begin{enumerate}
        \item[(i)] There exists $(w,s) \in E(\T)$ such that $m_{ws}\in \sigma'$ and it is in the same block as $m_{xy},m_{xz},m_{zw}$. Moreover, there exists no $m_{zw'} \in \sigma'$ such that $m_{zw'}, m_{ws}, m_{zw}$ are in the same block. 
      \begin{enumerate}
          \item[(a)] Suppose $m_{zw}\in \sigma$.  Then $m_{ws}\notin \sigma$ by \hyperref[bridgelemma1]{Forest-Bridge Lemma}. Note that $m_{ws} \mid \lcm(\sigma)$ as $m_{ws} \in \sigma'$. Hence $m_{ws}$ is a gap of $\sigma$ and thus by Proposition \ref{prop:bridgeforest} $(b)$, there exist $st, ws' \in E(\T)$ such that $m_{st}, m_{ws'} \in \sigma$ and $m_{st}, m_{ws'}, m_{ws}$ are in the same block. Suppose $ws'\neq zw$. Then 
          \[ m_{zw} \mid \lcm(m_{xz},m_{ws}) \mid \lcm(m_{xz},m_{st},m_{ws'}), \]
          which implies $m_{zw}\mid \lcm(m_{xz},m_{ws'})$. Therefore $m_{zw}$ is a bridge of $\sigma$, a contradiction. Now we suppose $ws'=zw$, i.e., $s'=z$. In particular, this means $m_{zw},m_{ws},m_{st}$ are in the same block. Thus, $m_{st} \notin \sigma'$ by \hyperref[bridgelemma1]{Forest-Bridge Lemma}. Note that $xz$ and $st$ satisfies the separation condition, and $xy>st$. This contradicts the minimality of the edge $xz$.
\begin{figure}[ht]
    \centering
    \includegraphics[width=0.45\textwidth]{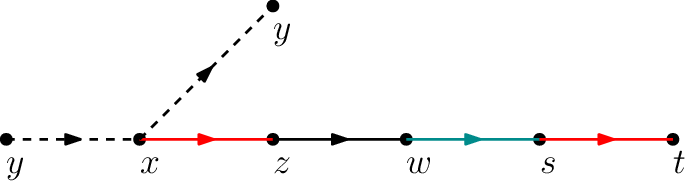}
    \caption{A figure depicting (a) (red edges are not in $\sigma'$, teal edge is not in $\sigma$)}
    \label{fig:exforest1}
\end{figure}

          \item[(b)] Suppose $m_{zw}\notin \sigma$. Note that $m_{zw} \mid \lcm(\sigma)$ as $m_{zw} \in \sigma'$. Hence $m_{zw}$ is a gap of $\sigma$ and thus by Proposition \ref{prop:bridgeforest} $(b)$, there exist $ ws', zw' \in E(\T)$ such that $m_{ws'}, m_{zw'} \in \sigma$ and $m_{ws'}, m_{zw'}, m_{zw}$ are in the same block.

          Suppose $zw'\neq xz$, i.e., $(z,w')\in E(\T)$. We know that $m_{zw}=z^aw^b$ for some integers $a,b$. Observe that in this case,
          \[ m_{zw} \mid \lcm(m_{xz},m_{ws}) \text{ and } m_{zw}\mid \lcm(m_{zw'},m_{ws'}). \]
           In particular, this means $z^a\mid m_{zw'}$ and $w^b\mid m_{ws}$. Therefore $m_{zw}\mid \lcm(m_{zw'},m_{ws})$. In other words, $m_{zw'}, m_{ws}, m_{zw}$ are in the same block, and hence $m_{zw'}\notin \sigma'$ by our assumption. Note that $xz$  and $zw'$ satisfies the separation condition, but $xz>zw'$. This contradicts the minimality of the edge $xz$. 
           \begin{figure}[ht]
    \centering
    \includegraphics[width=0.35\textwidth]{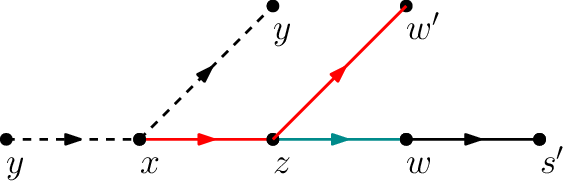}
    \caption{A figure depicting the first case  $zw'\neq xz$ in (b)}
    \label{fig:exforest2}
\end{figure}

Now we suppose  $zw'= xz$, i.e., $w'=x$. Since $xz>ws'$ and $ws'\in \sigma$, we must have $m_{ws'} \in \sigma'$.  Observe that $m_{zw}$ is a gap of $\sigma$. It follows from \hyperref[truegaplemma1]{Forest-True-Gap Lemma} $(i)$ that there exists $(s',t) \in E(\T)$  such that $m_{s't} \in \sigma$ and it is in the same block as $m_{xz}, m_{zw}, m_{ws'}$. Note that part (ii) of the lemma is not applicable to this case since $x>w$. Then, by \hyperref[bridgelemma1]{Forest-Bridge Lemma}, we must have $m_{s't} \notin \sigma'$. Then $s't$ satisfies the separation condition. Since $xz> s't$, we obtain a contradiction. 
\begin{figure}[ht]
    \centering
    \includegraphics[width=0.45\textwidth]{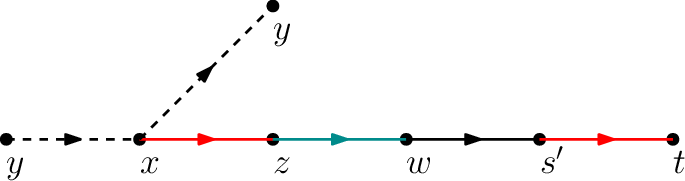}
    \caption{A figure depicting the second case $w'=x$ in (b)}
    \label{fig:exforest2'}
\end{figure}
      \end{enumerate} 
        \item[(ii)] If $z>y$, there exists $(y,y') \in E(\T)$ such that $m_{yy'}\in \sigma'$ and it is in the same block as $m_{xy},m_{xz},m_{zw}$. Moreover, there exists no $m_{xx'} \in \sigma$ such that $m_{xx'}, m_{yy'}, m_{xy}$ are in the same block. The proof of this case uses the same arguments as $(i)$ with the only difference that we have two subcases based on whether $m_{xy} \in \sigma$ and we work along the vertices $x,z,y$ instead of $x,z,w$. \qedhere 
          \begin{figure}[H]
    \centering
    \includegraphics[width=0.25\textwidth]{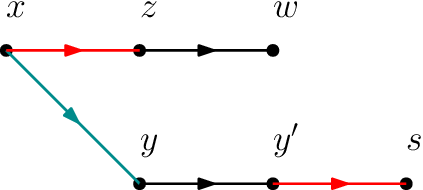}
    \caption{A figure depicting (ii)}
    \label{fig:exforest3}
\end{figure}    
    \end{enumerate}
\end{proof}

As another application, we obtain recursive formulas for the graded and total Betti numbers of $I(\T)$ which generalize those from \cite{Jac04} and \cite{JM05}. Before providing these formulas, we characterize the blockends:

\begin{proposition}\label{blockendforest}
	The monomial $m_{x^{(d)}_px^{(d+1)}_q}$ is a blockend if and only if either of the following holds:
	\begin{enumerate}[label=(\arabic*)]
		\item[(a)] Either $x^{(d)}_p$ or $x^{(d+1)}_q$ is a leaf. 
		\item[(b)] The vertex $x^{(d+1)}_q$ has a non-simple weight, i.e., $\w(x^{(d+1)}_q)\geq 2$.
	\end{enumerate}
\end{proposition}

\begin{proof}
	
	Let $e$ denote the directed edge $x^{(d)}_px^{(d+1)}_q$. Suppose $m_e$ is a blockend of a block  $\G(P)$ where
	\[E(P)=\{x_1x_2,x_2x_3,\dots, x_nx_{n+1}\}\]
	such that $m_e= m_1$ or $m_e=m_n$. Suppose that $m_e= m_1$.  For the sake of contradiction, suppose neither $x^{(d)}_p$ nor $x^{(d+1)}_q$ is a leaf and $\w(x^{(d+1)}_q)=1$. In particular, $m_1=x_1x_2$ and $x_1$ is not a leaf, i.e., there exist a directed edge $e_1$ incident to $x_1$ where $e_1 \neq e$. Let $P_1$ be the path obtained by extending $P$ at $m_1$ so that it contains $e_1$. Then $\G(P_1)=\{m_{e_1},m_1,m_2,\ldots, m_n\}$. Observe that  $\lcm(m_{e_1}, m_2)$  is divisible by $m_1$ because $x_1 |m_{e_1}$ and $x_2 |m_{2}$. Therefore, $\G(P_1)$ is a potential block which contains $\G(P)$, a contradiction. The other case $m_e=m_n$ is treated similarly.
	
	For the reverse direction, suppose $(a)$ or $(b)$ holds for $m_e$. We make the following claim whose proof is provided later.
	\begin{claim}\label{clm:blockend}
		The monomial $m_e$ does not divide $ \lcm(\sigma \setminus m_e)$ for any subset  $\sigma$ of $\G(\T)$.
	\end{claim}
	Suppose the claim holds. Observe that there exists a potential block $\G(P)$ that contains $m_e$ where $P$ is an induced path of $\T$.  Then by Claim \ref{clm:blockend} and the definition of a potential block, $m_e$ must be one of the two ends of $\G(P)$ and it is not possible to extend $\G(P)$ to a block at the end $m_e$.  If the extension can be done at the other end of $\G(P)$, extend and continue in this fashion. This process terminates after a finite number of steps at a block $\G(P')$ which admits $m_e$ as one of its blockends. 
\end{proof}

\begin{proof}[Proof of Claim \ref{clm:blockend}]
	Since $m_e$ is a blockend, then it satisfies Proposition \ref{blockendforest} (a) or (b). If (a) holds, then $m_e$ is the only monomial in $\G(\T)$ that is divisible by $x^{(d)}_p$ or $x^{(d+1)}_q$. Suppose (b) holds. Since $\T$ is a tree, the directed edge $(x^{(d)}_p,x^{(d+1)}_q)$ is the only edge directed towards $x^{(d+1)}_q$. Thus $m_e$ is the only monomial in $\G(\T)$ that is divisible by $(x^{(d+1)}_q)^2$. Hence, the statement holds.  
\end{proof}

We introduce some more notations. Let $p=\max \{\rank (x) : x \in V(\T)\}$. Let $v_1$ be one of the vertices of rank $p$ in $\T$ and $v$ be the predecessor of $v_1$. Denote the neighbors of $v$ by $v_1,\ldots, v_n$. Note that at most one of the neighbors of $v$ is not a leaf. If all of them are leaves, then $v=x_0$ and the ideal $I(\T)$ is quite simple to study.  Suppose exactly one neighbor of $v$ is not a leaf and denote it by $v_n$. Let $w_1,\ldots, w_k$ be the neighbors of $v_n$ other than $v$.  If $v_n$ has no predecessor, then $I(\T)$ is still a simpler ideal to study. So, suppose $w_k$ is the predecessor of $v_n$. 
\begin{figure}[ht]
    \centering
    \includegraphics[width=0.4\textwidth]{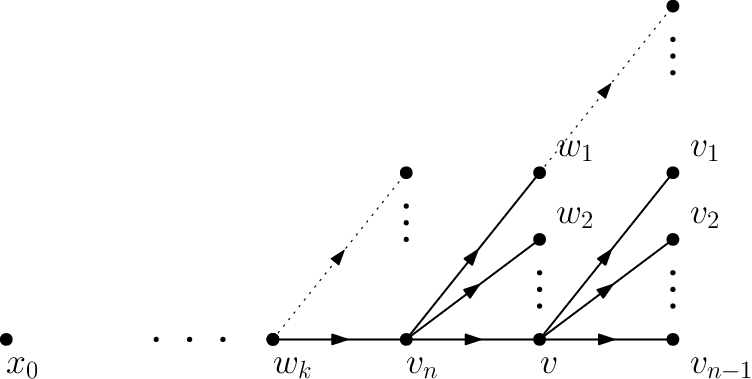}
    \caption{Our set-up}
    \label{fig:forestsetup}
\end{figure}

Let $W \subseteq V(\T)$. Let $\T\setminus W$ denote the weighted naturally oriented tree obtained from $\T$ by deleting all the vertices in $W$ and the edges incident to those vertices. Let $\T_1:= \T \setminus \{v_1\}, \T_2:= \T\setminus \{v,v_1,\ldots, v_n\}$, and $ \T_3:= \T_2 \setminus \{w_1,\ldots, w_k\}$.

\begin{theorem}\label{thm:bettiforest1}
   Let $M=\{m_{vv_2}, \ldots, m_{vv_n}\}$. Assume $\w(v)=1$. Then, we have
    \begin{align*}
        \beta_{r,d}(R/I(\T)) &=\beta_{r,d}(R/I(\T_1)) + \sum_{
       S\subseteq M } \beta_{r-(|S|+1),d-d' }(R/I(\T_2)) 
    \end{align*}
 for all indices $r,d$  where 
   $$d' = 1+ \w(v_{1})+ \sum_{\substack{m_{vv_i}\in S\\
    i\neq n}} \w(v_i)+ \begin{cases}
    1  &    \text{ if } m_{vv_n}\in S,\\
    0 &  \text{ otherwise}.
   \end{cases}$$
\end{theorem}

\begin{proof}
Recall that the Barile-Macchia resolution of $R/I(\T)$ is minimal by Theorems \ref{thm:friendlyToMinimal} and \ref{thm:morsefriendlyforest}. Then,  by Corollary \ref{morseminimabetti},  the $(r,d)^{\text{th}}$ graded Betti number of $R/I(\T)$ is the number of critical subsets of $\G(\T)$ with cardinality $r$ and degree $d$. Let $\sigma$ be a critical subset of $\G(\T)$ such that $|\sigma|=r$ and $\deg (\lcm (\sigma))=d$. Our approach is based on considering whether the monomial $m_{vv_1}$ belongs to $\sigma$.

(a) Suppose $m_{vv_1}\notin \sigma.$ Then $\sigma$ is critical subset of $\G(\T)$ if and only if it is a critical subset of $\G(\T_1)$. In this case, the number of such subsets of $\G(\T)$ equals $\beta_{r,d}(R/I(\T_1))$.

(b) Suppose $m_{vv_1}\in \sigma$. Recall that our variable ordering does not have a specific rule on how to order vertices of the same rank. So, we may assume that $v$ is the smallest among vertices of rank $p-1$. Then, none of the monomials $m_{v_nw_i}$ belongs to $\sigma$ for $i\in [k]$. Otherwise, $m_{vv_n}$ is either a bridge or a true gap of $\sigma$. This is not possible  by Corollary \ref{cor:critical} since $\sigma$ is a critical subset. Then  $\sigma'=\sigma \setminus \{m_{vv_1},m_{vv_2},\dots, m_{vv_n} \}$ is a critical subset of $\G(\T_2)$.
        
        Conversely, we can extend a critical subset of $\G(\T_2)$ to obtain a critical subset of $\G((I(\T))$. Let $\sigma'$ be a critical subset of $\G(\T_2)$ with cardinality $r-(j+1)$ where $0\leq j\leq n-1$. Let $S \subseteq M$ such that $|S|=j$. Suppose the following claim holds:
        
        \begin{claim}\label{claim:betti}
        Let $\sigma ''= \sigma' \cup \{m_{vv_1}\} \cup S$. Then,  $\sigma''$ is a critical subset of $\G(\T)$ where $|\sigma''|=r$.    
        \end{claim} 
        
        Then $\sigma=\sigma' \cup\{m_{vv_1}\} \cup S$ for some critical subset $\sigma'$ of $I(\T_2)$ with $|\sigma'|=r-(j+1)$ and a subset $S \subseteq M$ with $|S|=j$.  Let $S=\{m_{vv_{i_1}},\ldots, m_{vv_{i_j}}\}$ where $2 \leq i_1< i_2< \cdots < i_j \leq n$. Then
       \[\lcm(\sigma)=\begin{cases}
            \lcm(\sigma')\cdot v \cdot v_1^{\w(v_1)}\cdot v_n \cdot \prod_{k=1}^{j-1} v_{i_k}^{\w(v_{i_k})} &  \text{ if } i_j=n, \\\\
            \lcm(\sigma')\cdot v \cdot v_1^{\w(v_1)} \cdot \prod_{k=1}^{j}v_{i_k}^{\w(v_{i_k})}  &  \text{ otherwise } 
        \end{cases}
        \]
     and   
        \[ \deg (\lcm(\sigma))= d \iff \deg (\lcm(\sigma'))=\begin{cases}
         d- (2+\w(v_1) + \sum_{k=1}^{j}\w(v_{i_k}))   & \text{ if } i_j=n,\\
         d- (1+\w(v_1) + \sum_{k=1}^{j}\w(v_{i_k}))   & \text{ otherwise. } 
        \end{cases}\]
        Hence, the statement holds.
\end{proof}
\begin{proof}[Proof of the Claim \ref{claim:betti}] Consider the weighted oriented tree that only contains the directed edges of $\T$ involving the vertices $\{v_1,v_2,\ldots, v_n\}.$  This tree and $\T_2$ are disconnected. Thus, $\sigma''$ does not have any bridges. Next, our goal is to show that $\sigma''$ has no true gaps. Suppose $\sigma''$ has a true gap.  We claim that $m_{vv_i}$ is not a (true) gap of $\sigma''$ for any $i\in [n]$. Suppose $m:=m_{vv_i} \notin \sigma''$ for some $i\in [n]$. If $i=n$, then $m$ does not divide $\lcm(\sigma'' \setminus m)$ since $m_{v_nw_j} \notin \sigma''$ for any $j\in[k]$. If $i\in [n-1]$, then $m$ is a blockend and $\lcm(\sigma'' \setminus m)$ is not divisible by $m$ by Claim \ref{clm:blockend}. Thus, $m$ is not a (true) gap of $\sigma''$. Therefore, a true gap of $\sigma''$ is of the form $m_{v_nw_i}$ for some $1\leq i\leq k$. Then there exists $e_1=vv_n, e_2=w_iw \in E(\T)$ such that $m_{e_1},m_{e_2} \in \sigma''$ and $m_{e_1}, m_{v_nw_i},  m_{e_2}$ are all in the same block. Since $w_i>v$, we have the following two cases for the pair $e_1, e_2$ by Proposition \ref{prop:bridgeforest} $(c)$ $(ii)$:
        \begin{itemize}
            \item The monomial $m_{vv_n}$ is the only monomial in $\sigma''$ that is in the same block as $ m_{e_1}, m_{v_nw_i}, m_{e_2}$ and divisible by $v$. This is not possible because $m_{vv_1} \in \sigma''$ and it is in the same block as those monomials because $\w(v)=1$. 
            \item There exists $m_{v_nx'}, m_{vy'} \in \sigma''$ such that $m_{v_nx'},m_{vy'}, m_{vv_n}$ are all in the same block where $x'\neq v$ and $y'\neq v_n$. Since $x'$ is a neighbor of $v_n$, it must be one of the $w_i$'s, which is a contradiction since $m_{v_nw_i} \notin \sigma''$.    \qedhere
        \end{itemize} 
\end{proof}

\begin{remark}
    The formula given in Theorem \ref{thm:bettiforest1} coincides with \cite[Theorem 9.3.15]{Jac04} when $\w(x)=1$ for all $x\in V(\T)$.
\end{remark}

Let  $N_i$ be the set of all neighbors of $w_i$  other than $v_n$ for $i \in [k]$.  By reordering the vertices $w_1,\ldots, w_k$, we may assume that the first $l$ of them  have simple weights and $w_{l+1},\ldots, w_{k-1}$ have non-simple weights for $l \geq 0$. Let $M=\{m_{vv_2}, \ldots, m_{vv_{n-1}}\}$ and
$$M'=M\cup \{m_{v_nw_1},\ldots,m_{v_nw_k} \} \cup \{m_{w_iw_{i'}} :  w_{i'} \in N_i \text{ for } l<i\leq k-1\}.$$

\begin{theorem}\label{thm:bettiforest2}
Assume $\w(v)\geq 2$ and $\w(v_n)=1$. 
    Then, we have
 \[
        \beta_{r,d}(R/I(\T)) =\beta_{r,d}(R/I(\T_1)) + \sum_{S\subseteq M} \beta_{r-(|S|+1),d-d_1 }(R/I(\T_2)) + \sum_{S'\subseteq M'} \beta_{r-(|S'|+2),d-d_2 }(R/I(\T_3)) 
\]
    for all indices $r,d$  such that
    \begin{align*}
       d_1 & =1+ \w(v)+\w(v_{1})+ \sum_{m_{vv_i}\in S} \w(v_i)\\ 
       d_2&=1+\w(v)+\w(v_1)+ \sum_{xy^{\w(y) }\in S' } \w(y)+d'
    \end{align*}
   
    where $d'= \#\{i\mid m_{v_nw_i}\notin \sigma \text{ and } m_{w_iw_{i'}}\in \sigma \text{ for some } w_{i'} \in N_i \text{ where } l<i<k\}$. 
\end{theorem}

\begin{proof}
   As in the proof of the previous theorem, it suffices to partition the set of critical subsets of $\G(\T)$ of cardinality $r$ and degree $d$ appropriately.  Let $\sigma$ be a critical subset of $\G(\T)$ with $|\sigma|=r$ and $\deg(\lcm(\sigma))=d$. As before, our approach is based on whether $m_{vv_1}$ or $m_{vv_n}$ belong to $\sigma$. If $m_{vv_1}\notin \sigma$, then the number of such subsets equals $\beta_{r,d}(R/I(\T_1))$. Suppose $m_{vv_1}\in \sigma$. Next we consider whether  the monomial $m_{vv_n}$ is in $\sigma$.

(a) Suppose $m_{vv_n}\notin \sigma$.  Note that none of the $m_{vv_1},\ldots, m_{vv_{n-1}}$ is in a block $P$ with $|P|\geq 3$  since $\w(v)\geq 2$. Observe that $\sigma \setminus \{m_{vv_i}\}$ has no true gaps or bridges for $i\in [n-1]$. Moreover, $\sigma' \cup \{m_{vv_i}\}$ has no true gaps or bridges for a critical subset $\sigma' \subseteq \sigma$ of $I(\T_2)$ and  $i\in [n-1]$. Then $\sigma$ can be expressed as $\sigma = \sigma' \cup \{m_{vv_1}\} \cup S$ where $\sigma'$ a critical subset of $\G(\T_2)$ and $S \subseteq M$. Thus, $\deg(\lcm(\sigma))=d$ if and only if $\deg(\lcm(\sigma'))=d-d_1$ where $d_1=1+\w(v)+\w(v_1) + \sum_{m_{vv_i}\in S}\w(v_{i})$. Therefore, the number of such critical subsets for this case equals $\beta_{r-(|S|+1),d-d_1}(R/I(\T_2))$.

(b) Suppose $m_{vv_n}\in \sigma$. The proof of this case uses similar arguments as in the second case in the proof of Theorem \ref{thm:bettiforest1}. In that case, we start by observing $\sigma$ does not contain any monomial $m_{w_iw_{i'}}$ where  $i\in [l]\cup \{k\}$ and $w_{i'} \neq v_n$ is a neighbor of $w_i$.  Otherwise, $m_{v_nw_i}$ would be a bridge or a true gap of $\sigma$ for the corresponding $i$. This is shown first for $i=k$ and then iteratively build up for the others starting from $i=l$. Following the same steps, one can verify that $\sigma$ is of the form $\sigma' \cup\{m_{vv_1},m_{vv_n}\} \cup S'$ where $\sigma'$ is a critical subset of $\G(\T_3)$ and $S' \subseteq M'$. We omit the details to avoid repetition. 

Note that $\lcm(\sigma)$ is divisible by $v^{\w(v)}v_1^{\w(v_1)}v_n$ and we subtract $\w(y)$ from $d=\deg(\lcm(\sigma))$ for each $m_{xy} = xy^{\w(y)} \in S'$. Additionally, we subtract $1$ from $d$ for each $i$ whenever $m_{v_nw_i} \notin \sigma$ and $m_{w_iw_{i'}} \in \sigma$ for some $w_{i'} \in N_i$ where $l<i<k$. Let $d'$ be the total number of such $i$'s. Thus, the number of such critical subsets in this case  equals  $\beta_{r-(|S'|+2),d-d_2}(R/I(\T_3))$ where $d_2= \w(v)+\w(v_1) +\sum_{xy^{\w(y)} \in T} \w(y)+d'$.
\end{proof}

\begin{theorem}\label{thm:bettiforest3}
    Assume $\w(v)\geq 2$ and $\w(v_n)\geq 2$.  Then, for all indices $r,d$, we have

        \[\beta_{r,d}(R/I(\T)) =\beta_{r,d}(R/I(\T_1)) + \sum_{S\subseteq M} \beta_{r-(|S|+1),d-d_1 }(R/I(\T_2)) +d'\]
    where $d'$ is the number of critical subsets of $\G(\T)$ with cardinality $r$ and degree $d$ containing $vv_1$ and $ vv_n$.
\end{theorem}

\begin{proof}
    The arguments of this proof are similar to those from the previous two recursive formulas. Let $\sigma$ be a critical subset of $\G(\T)$ with $|\sigma|=r$ and $\deg(\lcm(\sigma))=d$. As we have seen earlier, if $m_{vv_1}\notin \sigma$, then the number of such subsets $\sigma$ equals  $\beta_{r,d}(R/I(\T_1))$. Suppose $m_{vv_1}\in \sigma$. If $m_{vv_n} \notin \sigma$, then we employ the same arguments from the proof of Theorem \ref{thm:bettiforest2}.  The remaining case is $m_{vv_1}\in \sigma$ and $m_{vv_n}\in \sigma$. This concludes the proof.
\end{proof}

One can continue to peel off $\T$ in this way to obtain more recursive formulas but we will stop here to avoid repetition. If one focuses on total Betti numbers instead of graded Betti numbers, formulas have simpler expressions.

\begin{theorem}\label{thm:totalbettiforest}
    For any index $r$ we have
    \[ \beta_r(R/I(\T))=\begin{cases} 
    \beta_r(R/I(\T_1)) +\sum_{j=0}^{n-1} \binom{n-1}{j}\beta_{r-(j+1)}(R/I(\T_2)) & \text{ if } \w(v)=1,\\
    \beta_r(R/I(\T_1))+\beta_{r-1}(R/I(\T_1))& \text{ if } \w(v)\geq 2.
    \end{cases}
    \]
    Moreover,
    \[ \pd(R/I(\T))=\begin{cases} 
    \max\{\pd(R/I(\T_1)), n+\pd(R/I(\T_2))\} & \text{ if } \w(v)=1,\\
     \pd(R/I(\T_1))+1& \text{ if } \w(v)\geq 2.
    \end{cases}
    \]
\end{theorem}

\begin{proof}
    Regarding the Betti numbers, we can apply the same argument from the proof of Theorem \ref{thm:bettiforest1} (a). For the second case, note that  $m_{vv_1}$ is not in any block of $P$ with $|P|\geq 3$ as $\w(v)\geq 2$. Hence, $\sigma$ is critical for $I(\T)$ if and only if $\sigma\setminus \{m_{vv_1}\}$ is critical for $I(\T_1)$.
    
    The projective dimension part is clear from the above arguments and Corollary \ref{morseminimabetti}.
\end{proof}

One can generalize the results of this section from edge ideals of weighted naturally oriented forests to a larger class of ideals.  Let $\T$ be a tree (no vertex weights or edge orientations). We assign each edge $e=xy$ of $\T$ with a pair $(p_e(x), q_e(y)) \in \ZZ_{+}^2$ and define the following ideal associated to $\T$ with respect to this pair assignment. 
$$J(\T)= (x^{p_e(x)} y^{q_e(y)} ~|~ e=xy \in E(\T ) \text{ and } (p_e(x), q_e(y)) \in \ZZ_{+}^2)$$

One can produce almost all the results of this section (except for the blockend description, Proposition \ref{blockendforest}, and the recursive formulas for (graded and total) Betti numbers) for $J(\T)$  using the same arguments.
\begin{theorem}\label{thm:superforest}
    The ideal $J(\T)$ is bridge-friendly and has a Batzies-Welker matching.
\end{theorem}

This result can be extended to forests by Theorem \ref{thm:tensormorse}. Moreover, the edge ideals of weighted oriented forests (with any orientation) belong to this class of ideals. 

\begin{corollary}
    The edge ideals of weighted oriented forests are bridge-friendly and have Batzies-Welker matchings.
\end{corollary}

\section{Minimal Free Resolutions of Edge Ideals of Weighted Oriented Cycles}\label{sec:cycle}

This section consists of three parts. In the first one, we simplify the study of weighted oriented cycles by finding a way to reduce any weighted oriented cycle (or path) into a \emph{naturally oriented} one. In the last two subsections, we study \emph{classic} and \emph{non-classic} cycles whose disjoint union is the class of weighted oriented cycles.

Let $C=(V(C),E(C),\w)$ be a weighted oriented cycle with a vertex set $V(C)=\{x_1,x_2,\dots, x_n\}$, an edge set $E(C)=\{ x_1x_2,x_2x_3,\dots, x_{n-1}x_n,x_nx_1 \}$, and a weight function $\w$ on the vertices.  Recall that the notation $x_ix_{i+1}$ does not indicate the direction of this edge. For the remainder of this section, indices are considered modulo $n$. Set $e_i= x_ix_{i+1}$ and $m_i=m_{e_i}$ for each $i \in [n]$. 

We start this section by introducing blocks and blockends for cycles. These definitions are very similar to those of forests and the main difference is the order they are introduced. In the case of forests, blockends are defined from blocks. It is more natural to first define blockends for ~cycles. 

\begin{definition}\label{definitionblockcycle}
   We call $m_i$ a \textit{blockend} of $I(C)$ if $m_i\nmid \lcm(m_{i-1},m_{i+1})$ for $i \in [n]$.  If $\G(C)$ has no blockends, then the cycle $C$ is called  \emph{classic}. Otherwise, we call $C$ \textit{non-classic}. 
   
Blocks are defined as follows: if $\G(C)$ has exactly one blockend, then we say $I(C)$ has only one \textit{block}, namely $\G(C)$. If $\G(C)$ has at least two blockends, then a \textit{block} of $I(C)$ is $\G(P)$ where $P$ is the weighted oriented path that admits two nearest blockends as the first and last edges,  contains all the other edges in between, and inherits weights and orientations from $C$. In particular, a block is a subset of $\G(C)$. Note that none of the $m_i$ between the two blockends of a block can be a blockend of another block. 
\end{definition}

The notion of (being in the) same block is defined as in Definition \ref{inthesameblock}.  Our definition of blockends and blocks are based on Claim \ref{clm:blockend} and Proposition \ref{blockendforest}, respectively.

\begin{remark} 
Unlike forests, there may be two paths admitting $e_i$ and $e_j$ as its first and last edges. If $I(C)$ has only two blockends, say  $m_i$ and $m_{i+k}$ for $i \in [n]$ and $k\geq 1$, then $\{m_i, m_{i+1} ,\ldots, m_{i+k}\}$ and $\{m_{i+k}, m_{i+k+1}, \ldots, m_i\}$ are the only blocks of $I(C)$. 
\end{remark}

\begin{example} 

Let $C_1, C_2, C_3$ be the weighted oriented cycles given in Figure \ref{fig:cycles}. We identify blockends and blocks of these cycles below. 

\begin{figure}[h]
    \centering
    \includegraphics[width=0.75\textwidth]{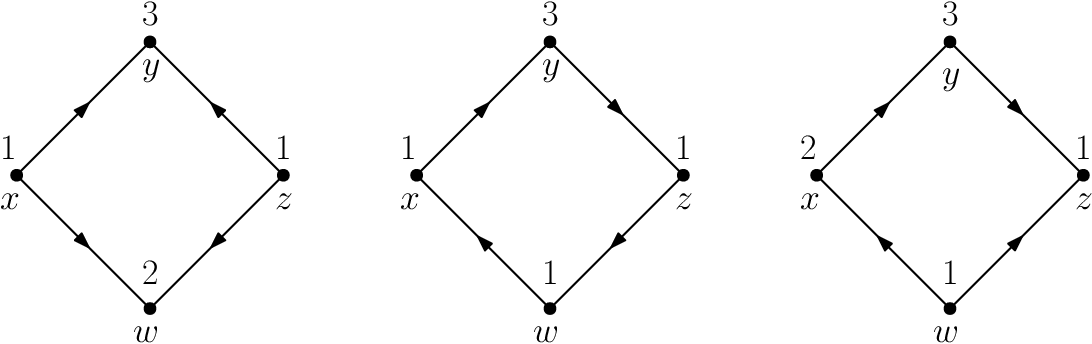}
    \label{fig:cycles}
    \caption{Weighted oriented cycles $C_1, C_2$, and $C_3$ (from left to right)}
\end{figure}

For the first cycle, we have $\G(C_1)=\{xy^3,y^3z,zw^2,w^2x\}$. It has no blockends or blocks. For the second cycle,  $\G(C_2)=\{\underline{xy^3},yz,zw,wx\}$ and it has only one blockend $xy^3$. For the last cycle, $\G(C_3)=\{\underline{xy^3},yz,zw,\underline{wx^2}\}$ has two blockends
        $xy^3$ and $wx^2$ since  $xy^3\nmid \lcm(wx^2,yz)$ and $wx^2\nmid \lcm(zw,xy^3)$. Hence, the blocks of $\G(C_3)$ are ~$\{xy^3,yz,zw,wx^2\}$ and $\{xy^3,wx^2\}$.
\end{example}

\subsection{Independence of Directions for Paths and Cycles}\label{subsec:independence}

The main result of this subsection is that given a weighted oriented cycle (resp, path), one can apply a sequence of operations to obtain a weighted \emph{naturally oriented} (all edges are are oriented in the same direction) cycle (resp, path) that shares similar properties. For the remainder of this subsection, we will work with the following total ordering $(>_I)$ on $\G(C)$:
\[ m_1>_I m_2>_I \cdots>_I m_{n-1}>_I m_n. \]
If $C$ is a non-classic cycle, then by definition it has at least one blockend. Without loss of generality, we may assume $m_n$ is one of them in this case.  Note that a blockend is neither a bridge nor a true gap of any subset of $\G(C)$.

One can produce an analogue of Proposition \ref{prop:bridgeforest} for cycles:

\begin{proposition}\label{bridgetruegapcycle}
    Let $C$ be a non-classic cycle and $\sigma$ a subset of $\G(C)$. Consider a monomial $m_i\in \G (C)$ where $1\leq i\leq n$. Then
    \begin{enumerate}
        \item[(a)] $m_i$ is a bridge of $\sigma$ iff $m_{i-1},m_i,m_{i+1}\in \sigma$ are in the same block.
        \item[(b)] $m_i$ is a gap of $\sigma$ iff $m_{i-1},m_{i+1}\in \sigma$ and $m_i\notin\sigma$ are in the same block, 
        \item[(c)] $m_i$ is a true gap of $\sigma$ iff
        \begin{itemize}
            \item $m_{i-1},m_{i+1}\in \sigma$ and $m_i\notin\sigma$ are in the same block and 
            \item if $m_{i+2}$ is in the same block as $m_{i-1},m_i,m_{i+1}$, then $m_{i+2}\notin \sigma$. In other words, $m_{i+1}$ is the only monomial that is in the same block as $m_{i-1}, m_i$ and  is divisible by $x_{i+2}$.
        \end{itemize}
    \end{enumerate}
\end{proposition}
\begin{proof}
   The proofs of $(a)$ and $(b)$ are identical to that of Proposition \ref{prop:bridgeforest} while the proof of $(c)$ is much simpler in this case as it immediately follows from the definition of a true gap.
\end{proof}

\begin{theorem}\label{morsepropertycycle}
    If $C$ is a non-classic cycle, then $I(C)$ is bridge-friendly and has a Batzies-Welker matching. In particular, if there exists $\sigma\subseteq \G(I(C))$ such that $\sigma$ has neither bridge nor true gap where $|\sigma|=r$ and $\lcm (\sigma) =\mathbf{x}^\mathbf{a}$, then $\beta_{r,\mathbf{a}}(R/I(\C))=1$. Otherwise, $\beta_{r,\mathbf{a}}(R/I(\C))=0$.

\end{theorem}
\begin{proof}
    The proof of bridge-friendliness follows from Lemma \ref{lem:morsefri2} and Proposition \ref{bridgetruegapcycle} $(c)$. The proof of the Batzies-Welker matching is similar to that of Theorem \ref{thm:morseminimalforest}. There are fewer cases to consider and one only needs to use Proposition \ref{bridgetruegapcycle} $(b)$ and $(c)$ along with the fact that a critical subset of $\G(C)$ has no bridges or true gaps. In order to avoid repetition, we will omit the details. 
\end{proof}

The next main result of this subsection is about independence of directions. This means for any given  weighted oriented cycle $C$, one can obtain a weighted \emph{naturally oriented} cycle $C^*$ from $C$  such that the Betti numbers of $I(C)$ can be expressed in terms of those of $I(C^*)$. A weighted oriented cycle $C$ is called \textit{naturally oriented} if all of its edges are oriented in the same direction, i.e., $(x_i,x_{i+1}) \in E(C)$ for each $i\in [n]$.

The process of obtaining $C^*$ from $C$ involves two operations: sinking and ironing. A vertex $x$ is called a \emph{sink} if every edge incident to $x$ is oriented towards it.

\begin{definition}
    \emph{Sinking} is an operation applied to a weighted oriented graph such that it reduces the weight of each sink vertex to 1 while keeping weights of all non-sinks the same. This operation only affects the weights and there is no change in the orientation of the graph. 
\end{definition}

\begin{notation}\label{straightenedcycle}
    Let $C_{\sink}$ denote the weighted oriented cycle obtained from $C$ via sinking and let $\w_{\sink}$ be the weight function of $C_{\sink}$. Then
   \[\w_{\sink}(x_i)\coloneqq \begin{cases}
   1 &\text{ if } x_i \text{ is a sink},\\
   \w(x_i)& \text{ otherwise}.
   \end{cases}\]  
   
   Let $m_i^{\sink}\in I(C_{\sink})$ denote the minimal generator corresponding to the edge $x_ix_{i+1}\in E(C_{\sink})$ for each $i \in [n]$. 
\end{notation}

Note that if a sink vertex appears in a minimal generator of $I(C)$, then it appears with the same exponent in each of those generators. Then $I(C)$ and $I(C_{\sink})$ only differ at the exponents of sink vertices. Hence, the minimal free resolution of $I(C)$ is exactly that of $I(C_{\sink})$, after a change of variables.

\begin{remark}\label{rem:GradedBetti}
Let $I$ be monomial ideal in $R=\Bbbk[x_1,\ldots, x_N]$ and $\textbf{v}=(\mathbf{v}_1, \ldots \mathbf{v}_n) \in \mathbb{N}^N $. All non-zero Betti numbers of $I$ occurs in $\mathbb{Z}^N$-graded degrees $\textbf{v}$ such that $x^\textbf{v} $ equals a least common multiple of some minimal generators of $I$ where $x^\textbf{v}=\prod_{i=1}^n x_i^{\mathbf{v}_i}$. 
\end{remark}

\begin{corollary}\label{cor:sinking}
    For any integer $r$ and vector $\mathbf{v}=(\mathbf{v}_1, \ldots ,\mathbf{v}_n)\in \mathbb{N}^n$, we have 
    \[\beta_{r,\mathbf{v}}(R/I(C))=\beta_{r,\mathbf{v}_{\sink}}(R/I(C_{\sink}))\]
    where 
    \[(\mathbf{v}_{\sink})_i =\begin{cases}
        \frac{\mathbf{v}_i}{\w(x_i)} & \text{if } \w_{\sink}(x_i)\neq \w(x_i),\\
        \mathbf{v}_i&  \text{otherwise.}
    \end{cases} \]
\end{corollary}

It turns out that sinking operation does not affect blockends of $I(C)$ when $C$ is a non-classic ~cycle.

\begin{proposition}\label{easyblockendcycle}
  If $C$ is a  non-classic cycle, so is $C_{\sink}$. Moreover, $m^{\sink}_i$ is a blockend if and only if $\deg m_i> 2$.
\end{proposition}

\begin{proof}
    For the first statement, it suffices to show that the blockends of $C$ and $C_{\sink}$ are in the same positions.  Notice that $m_i^{\sink}$ is a blockend of $I(C_{\sink})$ if and only if neither $x_i$ nor $x_{i+1}$ is a sink and  $\deg (m_i)>2$. This is equivalent to $m_i\nmid \lcm(m_{i-1},m_{i+1})$, i.e., $m_i$ is a blockend of $I(C)$.
\end{proof}

Notice that the second statement about blockends  is similar to Proposition \ref{blockendforest}. In fact, we have the following analogue:

\begin{corollary}\label{blockendcycles}
    If $C$ is naturally oriented, then for any $i \in [n]$, $m_i$ is a blockend if and only if $\w(x_{i+1})\geq 2$.
\end{corollary}

Next, we define the ironing operation.

\begin{definition}
    \emph{Ironing} is an operation applied to a weighted oriented cycle $C$ after sinking it. This operation produces a weighted naturally oriented cycle $C_{\iron}$ by reversing the direction of each edge of $C$ (or $C_{\sink})$ of the form $(x_i,x_{i-1})$. That way all the edges of $C_{\sink}$ are oriented in the same direction. The weight function of $C_{\iron}$ is defined as follows: 
      \begin{align*}
        \w_{\iron}(x_i)&=\begin{cases}
                    \w_{\sink}(x_{i-1}) &\text{ if } (x_i,x_{i-1})\in E(C_{\sink}),\\
                \w_{\sink}(x_i) &\text{ otherwise.}
                 \end{cases}
    \end{align*}
    Let $m_i^{\iron}$ denote the minimal generator of $I(C_{\iron})$ corresponding to edge $x_ix_{i+1} \in E(C_{\iron})$. ~Then
    \[m_i^{\iron}=\begin{cases}
    x_ix_{i+1}^{\w_{\sink}(x_i)} &\text{ if } m_i^{\sink}= x_i^{\w_{\sink}(x_i)} x_{i+1},\\\\
    x_ix_{i+1}^{\w_{\sink}(x_{i+1})} &\text{ if } m_i^{\sink}= x_ix_{i+1}^{\w_{\sink}(x_{i+1})}. 
    \end{cases}\]
\end{definition}

We are now ready to present our main result on independence of directions for cycles. 

\begin{theorem}\label{indecycle}
    Let $C$ be a weighted oriented cycle. Then
    \[\beta_{r,d}(R/I(C_{\sink}))=\beta_{r,d}(R/I(C_{\iron}))\]
    for any indices $r,d$. In particular,
    \[\beta_{r}(R/I(C))=\beta_{r}(R/I(C_{\sink}))=\beta_{r}(R/I(C_{\iron}))  \]
    for any index $r$.
\end{theorem}

\begin{proof}
    If $C$ is a classic cycle, then $C_{\sink}$ is exactly the edge ideal of an $n$-cycle. It can be viewed as a naturally oriented cycle equipped with the constant weight function $
    \w(x_i)=1$ for each $i\in [n]$. In this case, $C_{\iron}=C_{\sink}$ and both of the statements hold. 
    
     Suppose $C$ is a non-classic cycle. Let $\sigma_{\sink}=\{m_{i_1}^{\sink}, m_{i_2}^{\sink},\dots, m_{i_k}^{\sink}\}$ be a subset of $\G(C_{\sink})$ and let $\sigma_{\iron}=\{m_{i_1}^{\iron}, m_{i_2}^{\iron},\dots, m_{i_k}^{\iron}\}$ be the corresponding subset of $\G(C_{\iron})$. The key idea of the proof is the  observation that  $m_i^{\sink}$ is a bridge/true gap in $\sigma_{\sink}$  if and only if $m_i^{\iron}$ is a bridge/true gap in the corresponding  $\sigma_{\iron}$.   This follows from the characterizations given in Proposition  \ref{bridgetruegapcycle}, and, together with Proposition \ref{easyblockendcycle},  implies that $m_i^{\sink}$ is a blockend for $I(C_{\sink})$ if and only if $m_i^{\iron}$ is a blockend  for $I(C_{\iron})$. Hence, $\sigma_{\sink}$ is a critical subset of $\G(C_{\sink})$ if and only if $\sigma_{\iron}$ is a critical subset of $\G(C_{\iron})$. Note that $\deg m_i^{\sink}=\deg m_i^{\iron}$ for all $i \in [n]$. Therefore, $I(C_{\sink})$ and $I(C_{\iron})$ have the same graded Betti numbers by Corollary \ref{morseminimabetti} and Theorem \ref{morsepropertycycle}.
\end{proof}

One can sink and iron a given weighted oriented path in the same way as cycles and obtain the same result for a weighted oriented path $P, P_{\sink}$, and $P_{\iron}$. Note that $I(P)$ is always non-classic in the sense of Definition \ref{definitionblockcycle}. We will omit the proof in this case since the arguments are exactly the same. 

\begin{theorem}\label{indepath}
    Let $P$ be a weighted oriented path. Then
    \[\beta_{r,d}(R/I(P_{\sink}))=\beta_{r,d}(R/I(P_{\iron}))\]
    for any indices $r,d$. In particular,
    \[\beta_{r}(R/I(P))=\beta_{r}(R/I(P_{\sink}))=\beta_{r}(R/I(P_{\iron}))  \]
    for any index $r$.
\end{theorem}

We now illustrate the last two theorems with an example.

\begin{example}
 Let $C$ and $P$ be the weighted oriented cycle and path given in Figures \ref{fig:iron} and \ref{fig:straighten}, respectively.  It can be verified with Macaulay2 \cite{M2} that $I(C), I(C_{\iron})$, and $I(C_{\sink})$ have the same total Betti numbers. The same is true of $I(P), I(P_{\iron})$, and $I(P_{\sink})$.
  \begin{figure}[h]
        \centering
        \includegraphics[width=0.75\textwidth]{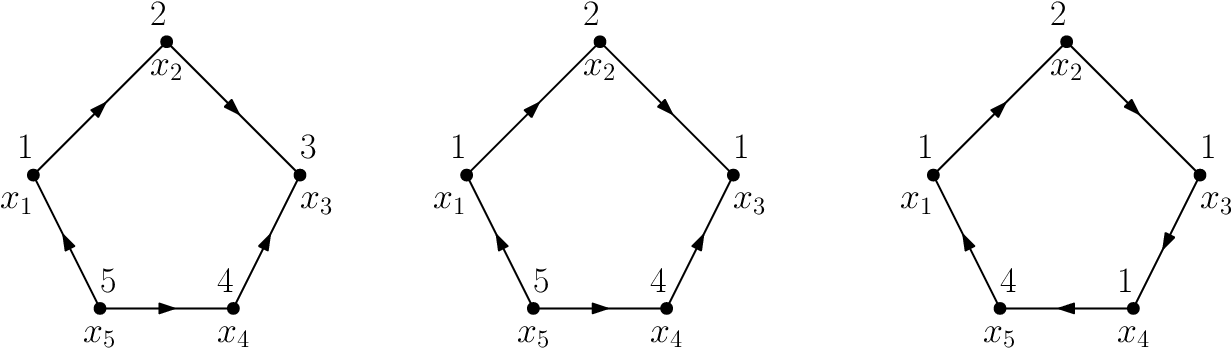}
        \caption{ $C, C_{\sink}$ and $C_{\iron}$ (from left to right)}
        \label{fig:iron}
    \end{figure}
    
 \begin{figure}[h]
        \centering
        \includegraphics[width=0.45\textwidth]{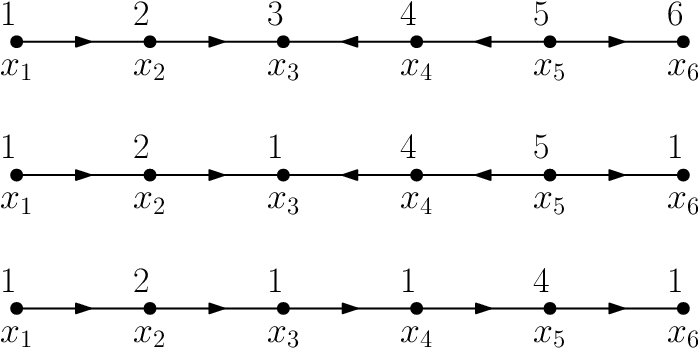}
        \caption{$P, P_{\sink},$ and $P_{\iron}$ (from top to bottom)}
        \label{fig:straighten}
    \end{figure}
\end{example}

As we shall see in the following example, it is not possible to obtain an analogue of Theorems \ref{indecycle} and \ref{indepath} for weighted oriented trees (or more generally, forests).

\begin{example}\label{notindeforest}
    Let $\T$ be the weighted oriented tree  given in Figure \ref{fig:weirdforest}.
        \begin{figure}[H]
        \centering
        \includegraphics[width=0.3\textwidth]{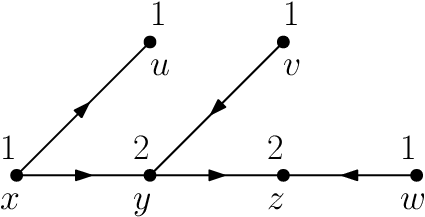}
        \caption{A weighted oriented tree}
        \label{fig:weirdforest}
    \end{figure}

   Then $I(\T)= (xy^2,xu,yz^2,y^2v,z^2w)$. As in cycles and paths, one can obtain a weighted naturally oriented tree from $\T$ by sinking and ironing operations to $\T$. After sinking $\T$, we obtain  $I(\T_{\sink})= (xy^2,xu,yz,y^2v,zw)$.  The total Betti numbers of $R/I(\T)$ and $R/I(\T_{\sink})$ both equal $(1,5,7,3)$.

    On the other hand, it is not possible to preserve the total Betti numbers after ironing $\T_{\sink}$. In fact, one can verify with Macaulay2 \cite{M2} that there are no weighted naturally oriented trees on 6 vertices with the total Betti numbers $(1,5,7,3)$. Note that there are a finite number of them as one needs to only consider the trees with vertex weights one or two, i.e., $\w(V(\T)) \subseteq \{1,2\}$ for a weighted naturally oriented tree $\T,$ due to Propositions \ref{blockendforest}, \ref{prop:bridgeforest}, and Theorem \ref{thm:morsefriendlyforest}.
\end{example}

Paths and cycles share the property of independence of directions and this shared property comes from the similarities in the total orderings inducing Barile-Macchia resolutions that are minimal. This observation leads to the next result regarding their (sometimes identical) total Betti numbers. Recall that an $n$-cycle is a cycle on $n$ vertices and a path of length $n$ is a path with $n$ edges.

\begin{theorem}\label{thm:identicalBetti} We have the following:
  \begin{enumerate}
        \item[(a)] If $P$ is a weighted oriented path of length $n$, then there exists a weighted oriented $n$-cycle $C$ such that $\beta_{i}(R/I(P))=\beta_i(R/I(C))$ for each $i\in \mathbb{Z}$. 
        \item[(b)] If $C$ is a weighted oriented $n$-cycle such that $C$ has a block of cardinality $2$ or $n\leq 4$, then there exists a weighted oriented path $P$ of length $n$ such that $\beta_{i}(R/I(P))=\beta_i(R/I(C))$ for each $i\in \mathbb{Z}$.
        \item[(c)]  For each $n\geq 5$, there exists a weighted oriented $n$-cycle $C$ such that for any weighted oriented path $P$, we have $\beta_{i}(R/I(P))\neq \beta_i(R/I(C))$ for some $i\in \mathbb{Z}$.
    \end{enumerate}
\end{theorem}

\begin{proof}
    First note that $\beta_{1}(R/I(P))=\beta_1(R/I(C))$ is equivalent to $|\G(P)|= |\G(C)|$. Hence, for all three parts, we can assume path $P$ is on the vertices $x_1, \ldots, x_{n+1}$ and cycle $C$ on the vertices $y_1, \ldots, y_n$. Then 
    $$I(P)=(m^P_1,m^P_2,\dots, m^P_n) \text{ and } I(C)=(m^C_1,m^C_2,\dots, m^C_n)$$ where $m_i^P=m_{x_ix_{i+1}}$ for each $i\in [n]$,     $m_j^C=m_{y_jy_{j+1}}$ for each $j\leq [n-1]$, and $m_n^C=m_{y_ny_1}$.  Let $\w_P$ and $\w_C$ be the weight functions of $P$ and $C$, respectively.

    We may assume $P$  and $C$ are naturally oriented  by Theorems \ref{indecycle} and \ref{indepath}. Let $x_1$ be the root of $P$. Consider the Barile-Macchia matchings induced by the total orderings $m^P_1>m^P_2>\cdots > m^P_n$ on $\G(P)$ and $m^C_1>m^C_2>\cdots > m^C_n$ on $\G(C$). Observe that $I(P)$ and $I(C)$ are both  bridge-friendly by Theorems \ref{thm:morsefriendlyforest} and \ref{morsepropertycycle} with respect to the given total orderings while noting that  $I(C)$ must be non-classic for the second theorem to hold. 
    \begin{enumerate}[label=(\alph*)]
        \item Set $\w_C(y_1)=\w_C(y_2)\coloneqq 2$  and $\w_C(y_i)\coloneqq \w_P(x_i)$ for each $3\leq i\leq n$. Then $C$ is a non-classic cycle. Note that $m^C_i$ is a blockend if and only if $m^P_i$ is a blockend for  each $1\leq i\leq n$ by Proposition \ref{blockendforest} and Corollary \ref{blockendcycles}. Let $\sigma_P=\{m^P_{i_1},m^P_{i_2},\dots,m^P_{i_k}\}$ be a subset of $\G(P)$ and $\sigma_C=\{m^C_{i_1},m^C_{i_2},\dots,m^C_{i_k}\}$ be the corresponding subset of $\G(C)$. It follows from Propositions \ref{prop:bridgeforest} and \ref{bridgetruegapcycle} that, for any $1\leq i\leq n$, $m_i^C$ is a bridge/true gap of $\sigma_C$ if and only if  $m_i^P$ is a bridge/true gap of $\sigma_P$. Thus, $\sigma_C$ is critical if and only if $\sigma_P$ is critical. Therefore, $I(P)$ and $I(C)$ have the same total Betti numbers by Corollary ~\ref{morseminimabetti}.
        \item  Suppose $C$ has a block of cardinality $2$, i.e., $C$ has two consecutive blockends. Without loss of generality, we may assume $m_1^C$ and $m_n^C$ are blockends. Set $\w_P(x_i)\coloneqq \w_C(y_i)$ for each $1\leq i\leq n$. By similar arguments as the previous part, we obtain the same ~conclusion.

        Now suppose $C$ does not have any blocks of cardinality $2$ and is of length at most $4$. There are only a small number of possibilities. Note that since a cycle is symmetric, if $C$ has a blockend,  we may assume $m_1^C$ is one.  Below, we construct the corresponding $P$ for each possible case:
        \begin{itemize}
            \item If $I(C)=(y_1y_2,y_2y_3,y_3y_1^a)$ where $a\geq 1$, set $I(P)=(x_1x_2,x_2x_3,x_3x_4)$.
            \item If $I(C)=(y_1y_2,y_2y_3,y_3y_4, y_4y_1^a)$ where $a\geq 1$, set $I(P)=(x_1x_2,x_2x_3,x_3x_4,x_4x_5)$.
            \item If $I(C)=(y_1y_2,y_2y_3^a,y_3y_4, y_4y_1^b)$ where $a,b\geq 2$,  set $I(P)=(x_1x_2,x_2x_3^2,x_3x_4,x_4x_5)$.
        \end{itemize}
        One can verify with Macaulay2 \cite{M2} that $I(C)$ and $I(P)$ have the same total Betti numbers in each case.
        \item Let $I(C)=(y_1y_2,y_2y_3,\dots, y_{n-1}y_n,y_ny_1^2)$. Let $P$ be a weighted naturally oriented path. If $P$ is not of length $n$, then $\beta_1(R/I(C))\neq\beta_1(R/I(P))$. Otherwise, we will show that $\beta_2(R/I(C))<\beta_2(R/I(P))$. By the technique of polarization (see \cite[Observation 37]{casiday2021betti}), it suffices to show that the inequality holds for $I(P)=(x_1x_2,x_2x_3,\dots, x_{n-1}x_n,x_nx_{n+1})$. Note that if $\sigma_C$ is critical, then the corresponding subset $\sigma_P$ is critical. On the other hand, while $\{m_2^P,m_n^P\}$ is critical, $\{m_2^C,m_n^C\}$ is not since $m_1^C$ is a true gap of $\sigma_C$. Therefore $\beta_2(R/I(C))<\beta_2(R/I(P))$ by Corollary \ref{morseminimabetti}. \qedhere
    \end{enumerate}
\end{proof}

\subsection{Non-Classic Cycles and Recursive Formulas For Graded Betti Numbers}\label{subsec:non-classic}

In this subsection, we obtain recursive formulas for the Betti numbers of edge ideals of weighted oriented cycles. As we have seen in the previous chapter, it suffices to consider the Betti numbers of edge ideals of weighted naturally oriented cycles by Theorem \ref{indecycle}. For the remainder of this section, we assume $C$ is a non-classic weighted naturally oriented cycle. As in the previous subsection, let $m_i$ be the monomial associated to the edge $x_ix_{i+1} \in E(C)$ for each $i\in [n]$. Without loss of generality, we may assume $m_n$ is a blockend, i.e., $\w(x_1) \geq 2$.

In the previous subsection, we proved that $I(C)$ is bridge-friendly with respect to a given total ordering $(>_I)$. This total ordering is not fruitful in our efforts to obtain recursive formulas similar to the ones from  Theorems \ref{thm:bettiforest1}, \ref{thm:bettiforest2}, and \ref{thm:bettiforest3}. For this reason, we need a new total ordering on $\G(C)$ in this subsection.  Let $B_1,\dots,B_p$ be all the blocks of $I(C)$ such that
\begin{align*}
    B_1&=\{\underline{m_{b_p}=m_n}, m_1,m_2\dots, \underline{m_{b_1}}\},\\
    B_2&=\{\underline{m_{b_1}},m_{b_1+1},\dots, \underline{m_{b_2}}\},\\
    \vdots&\\
    B_p&=\{\underline{m_{b_{p-1}}},m_{b_{p-1}+1},\cdots, \underline{m_{b_p}=m_n}\}
\end{align*}
where each blockend of a block is underlined. If $|B_i|=2$ for each $i\in [p]$, then all monomials in $\G(C)$ are blockends. In particular, it is not possible for any subset of $\G(C)$ to contain a bridge or a (true) gap. Thus, the minimal free resolution of $I(C)$ is the Taylor resolution and Betti numbers are obtained directly from the Taylor resolution. Thus, we may assume $|B_k|\geq 3$ for some $k\in [p]$. We introduce a new total ordering on $\G(C)$ for a fixed $k$ where $|B_k|\geq 3$, which we call the  \emph{k-flip ordering}. This order is similar to $(>_I)$ with the only difference that  the order of the second and third elements of $B_k$, namely $m_{b_{k-1}+1}$ and  $m_{b_{k-1}+2}$,  are switched.  Let $(>_{k})$ denote the $k$-flip ordering:
\[  m_{1}>_k  m_2 >_k \cdots >_k m_{b_{k-1}} >_k \textcolor{blue}{m_{b_{k-1}+2}}>_k \textcolor{blue}{m_{b_{k-1}+1}}>_k m_{b_{k-1}+3}>_k \cdots >_k m_{n-1}>_k m_n\] 

In order to obtain formulas for the Betti numbers using our tools, we first show that $I(C)$ is bridge-friendly  with respect to  $(>_{k})$. We will start by characterizing bridges, gaps, and true gaps of $I(C)$ with respect to $(>_{k})$. Notice that bridges and gaps do not depend on the total ordering. Hence,  we use the characterizations of bridges and gaps of $I(C)$ from Proposition \ref{bridgetruegapcycle} and only consider true gaps.

\begin{proposition}\label{bridgetruegapcycle2}
    Let $\sigma$ be a subset of $\G(C)$. Consider a monomial  $m_i\in \G(C)$. Then, $m_i$ is a true gap of $\sigma$ if and only if  $m_i$ is a gap of $\sigma$  and
    \begin{enumerate}
        \item[(a)] when $i=b_{k-1}+1$, nothing else is needed;
        \item[(b)]  when $i=b_{k-1}+2$, we have $m_{i-2}\notin \sigma$  and, whenever $m_{i+2}$ is in the same block as $m_{i-1},m_i,m_{i+1}$, then $m_{i+2}\notin \sigma$;
        \item[(c)]  when $i\notin \{ b_{k-1}+1, b_{k-1}+2\}$, whenever $m_{i+2}$ is in the same block as $m_{i-1},m_i,m_{i+1}$, then $m_{i+2}\notin \sigma$.
    \end{enumerate}
\end{proposition}

\begin{proof}
    Recall from Proposition \ref{bridgetruegapcycle} $(b)$ that $m_i$ is a gap of $\sigma$ if and only if $m_{i-1},m_i,m_{i+1}$ are in the same block while $m_{i-1},m_{i+1}\in \sigma$ and $m_i\notin\sigma$. In particular,  in this case, $m_i$ is not  a ~blockend.
    
    Suppose $i=b_{k-1}+1$. Note that $\sigma \cup m_i$ has a new bridge smaller than $m_i$ if and only if that bridge is either $m_{b_{k-1}}$ or $m_{b_{k-1}+2}$. None of these situations is possible because   $m_{b_{k-1}}$ is a blockend, thus it cannot be a bridge, and $m_{b_{k-1}+2}>_k m_{b_{k-1}+1}$. Hence $m_i$ is a true gap of $\sigma$ if and only if it is a gap of $\sigma$.
    
    Suppose $i=b_{k-1}+2$. Note that $\sigma \cup \{m_i\}$ has a new bridge dominated by $m_i$ if and only if that bridge is either $m_{b_{k-1}+1}$ or $m_{b_{k-1}+3}$. Notice that $m_{i-2}=m_{b_{k-1}}$ is already in the same block as $m_{i-1},m_i,m_{i+1}$ and the monomial $m_{i-1}=m_{b_{k-1}+1}$ is a bridge of $\sigma \cup \{m_i\}$ iff $m_{i-2} \in \sigma$. On the other had, the monomial $m_{i+1}=m_{b_{k-1}+3}$ is a bridge of $\sigma \cup \{m_i\}$ iff $m_{i+2}\in \sigma$ and it is in the same block as $m_{i-1},m_{i},m_{i+1}$.  Thus, the statement in $(b)$ holds from the definition of true ~gaps.

    The third case follows immediately from the definition of true gaps. Note that it is the same condition as the one in Proposition \ref{bridgetruegapcycle} $(c)$, which is not surprising since the two orderings  $(>_{k})$ and $(>_I)$ are very similar.
\end{proof}

The bridge-friendliness of $I(C)$ with respect to the $k$-flip ordering follows immediately from Lemma \ref{lem:morsefri2} and Proposition \ref{bridgetruegapcycle2}.

\begin{proposition}\label{thm:morseminimalcycle2}
    The ideal $I(C)$ is bridge-friendly with respect to $(>_{k})$.
\end{proposition}

The Barile-Macchia resolution  of $I(C)$ with respect to the $k$-flip ordering is then minimal and we can read off the Betti numbers of $I(C)$ from this Barile-Macchia resolution. In the following recursive formulas, the graded Betti numbers of $I(C)$ are expressed in terms of those of weighted naturally oriented paths. By combining these formulas with the ones given in Theorems \ref{thm:bettiforest1} -- \ref{thm:totalbettiforest}, one can reach to a better understanding of the graded Betti numbers, total Betti numbers, and projective dimension of edge ideals of weighted naturally oriented cycles.

As in Theorems \ref{thm:bettiforest2} and \ref{thm:bettiforest3}, the more vertices with non-simple weights in $C$, the more involved the formula for graded Betti numbers becomes. Hence, we only provide the graded Betti numbers only for two simple cases. One can continue to derive recursive formulas under other weight ~assumptions.

Consider the following weighted naturally oriented paths obtained from $C$ by deleting an edge or some vertices of $C$: 
$$ P_1  := C \setminus \{x_1x_n\}, ~~P_2  := C \setminus \{x_{n-1},x_n, x_1, x_2\} \text{ and } P_3 := P_2 \setminus \{ x_3\}.$$

\begin{theorem}\label{cor:recursive_non-classic}
Assume the vertex $x_n$ has a simple weight, i.e., $\w(x_n)=1$.
\begin{enumerate}
    \item[(a)]  If $\w(x_2)=1$, then, for all $r$ and $d$, we have
    \[ \beta_{r,d}(R/I(C))=\beta_{r,d}(R/I(P_1)) +\sum_{j=0}^{2}\binom{2}{j} \beta_{r-(j+1),d-(j+1+\w(x_1))}(R/I(P_2)). \]
    \item[(b)]  If $\w(x_2)\geq 2$ and $\w(x_3)=1$, then, for all   $r$ and $d$, we have 
    \begin{align*}
        \beta_{r,d}(R/I(C))=\beta_{r,d}(R/I(P_1)) &+\beta_{r-1,d-(1+\w(x_1))}(R/I(P_3))+\beta_{r-2,d-(2+\w(x_1))}(R/I(P_3))\\
        &+\beta_{r-2,d-(3+\w(x_1))}(R/I(P_3))+\beta_{r-3,d-(4+\w(x_1))}(R/I(P_3))\\
        &+\sum_{j=0}^{2}\binom{2}{j} \beta_{r-(j+2),d-(j+1+\w(x_1)+\w(x_2))}(R/I(P_3)).
    \end{align*}
\end{enumerate}
\end{theorem}

\begin{proof}
The assumptions of $(a)$ and $(b)$ are equivalent to $|B_1|\geq 3$ and $|B_2|\geq 3$, respectively. Hence we can use the $1$-flip ordering for $(a)$ and  $2$-flip ordering for $(b)$. Recall from Proposition \ref{thm:morseminimalcycle2} that $I(C)$ is bridge-friendly with respect to both total orderings. Thus, we can make use of Corollary \ref{morseminimabetti} and only consider critical subsets of $\G(C)$ with cardinality $r$ and degree $d$. Let $C_{r,d}$ denote the set of all such critical subsets of $\G(C)$ and $\sigma \in C_{r,d}$.  We consider the following two cases based on whether $m_n$ is contained in $\sigma$.

$(a)$ If $m_n\notin \sigma$, then $\sigma$ is critical subset of $\G(C)$ if and only if it is a  critical subset of  $\G(P_1)$. In this case, the cardinality of  $C_{r,d}$  equals  $\beta_{r,d}(R/I(P_1))$.

If $m_n\in \sigma$, then neither $m_2$ nor $m_{n-2}$ is contained in $\sigma$. Otherwise, $\sigma$ admits $m_{1}$ and $m_{n-1}$ either as a bridge or a true gap. As a result, $\sigma \setminus \{m_1,m_{n-1}, m_n \}$ is a critical subset of $\G(P_2)$. We claim that $\sigma$ can be expressed as $\tau\cup \{m_n\} \cup S$ where  $\tau$ is a critical subset of $\G(P_2)$ with  $\deg (\tau)= r-(j+1)$ for $0\leq j\leq 2$ and $S \subseteq \{m_1,m_{n-1}\}$.  Let $\tau':= \tau\cup \{m_n\} \cup S$. It suffices to show $\tau'$ is in $ C_{r,d}$. First note that $\tau'$ has no bridges since the supports of $m_1, m_{n-1}, m_n$ are disjoint from the vertices of $P_2$. Lastly, if $\tau'$ has a true gap, it must be either $m_2$ or $m_{n-2}$. In either case, we have $m_n \notin \sigma$ by Proposition \ref{bridgetruegapcycle2}, a contradiction. Hence, $\tau'$ is indeed a critical subset of $\G(C)$ and our claim holds. One can deduce the degree of $\sigma$ in terms of that of $\tau$ by considering the four possibilities of $S$.

$(b)$ This proof is based on the arguments from proofs of $(a)$ and Theorem \ref{thm:bettiforest2}. To avoid repetition, we will omit the details. 
\end{proof}

In what follows, we present recursive formulas for total Betti numbers and projective dimension as we did for forests in Theorem \ref{thm:totalbettiforest}. These formulas are derived from similar arguments as the ones used above. 

\begin{theorem}\label{cor:recursive_nonclassic2}
    Let $P'=C\setminus \{ x_{n-1}, x_n\}$  and $P''=C\setminus \{ x_1, x_2\}$. Then, for all $r$,  we have
    \[\beta_r(R/I(C))=
    \begin{cases}
        \beta_r(R/I(P))+ \sum_{j=0}^{2}\binom{2}{j} \beta_{r-(j+1)}(R/I(P'))&\text{ if } \w(x_n)=\w(w_2)=1,\\
        \beta_r(R/I(P))+ \sum_{j=0}^{1}\binom{1}{j} \beta_{r-(j+1)}(R/I(P'))  &\text{ if } \w(x_n)=1 \text{ and } \w(x_2)\geq 2,\\
        \beta_r(R/I(P))+ \sum_{j=0}^{1}\binom{1}{j} \beta_{r-(j+1)}(R/I(P''))  &\text{ if } \w(x_n)\geq 2 \text{ and } \w(x_2)=1,\\
        \beta_r(R/I(P))+\beta_{r-1}(R/I(P))&\text{ if } \w(x_n)\geq 2 \text{ and } \w(x_2)\geq 2.
    \end{cases}
    \]
    Moreover,
    \[\pd(R/I(C))=
    \begin{cases}
        \max\{\pd(R/I(P)),3+\pd(R/I(P'))\}&\text{ if } \w(x_n)=\w(w_2)=1,\\
        \max\{\pd(R/I(P)),2+\pd(R/I(P'))\}  &\text{ if } \w(x_n)=1 \text{ and } \w(x_2)\geq 2,\\
        \max\{\pd(R/I(P)),2+\pd(R/I(P''))\}  &\text{ if } \w(x_n)\geq 2 \text{ and } \w(x_2)=1,\\
        1+ \pd(R/I(P))&\text{ if } \w(x_n)\geq 2 \text{ and } \w(x_2)\geq 2.
    \end{cases}
    \]
\end{theorem}

In the same spirit as the conclusion of weighted oriented forests, our main results on edge ideals of weighted naturally oriented cycles can be generalized to a larger class of ideals. Let $C$ be an $n$-cycle (no vertex weights or edge orientations). We assign each edge $e=x_ix_{i+1}$ of $C$ with a pair $(p_e(x_i), q_e(x_{i+1})) \in \ZZ_{+}^2$ and define the following ideal associated to $C$ with respect to this pair assignment: 
$$J(C)= (x_i^{p_e(x_i)} x_{i+1}^{q_e(x_{i+1})} ~|~ e=x_ix_{i+1} \in E(C ) \text{ and } (p_e(x_i), q_e(x_{i+1})) \in \ZZ_{+}^2)$$

We adopt the definition of blockends and the notions of classic and non-classic cycles from Definition \ref{definitionblockcycle}. Furthermore, by applying arguments similar to those employed throughout this subsection to this class of ideals, one can obtain the following result.

\begin{theorem}\label{thm:supercycle}
    If $J(C)$ is non-classic, i.e., $\G(J(C))$ has a non-bridge, then $J(C)$ is bridge-friendly and has a Batzies-Welker matching.
\end{theorem}

\subsection{Classic Cycles and Eliahou-Kervaire Splitting}\label{subsec:classic}

The main result of this subsection provides an inductive way to construct the minimal free resolution of a classic cycle. Since we already considered non-classic cycles in Subsection \ref{subsec:non-classic}, the results of this subsection concludes our investigation on the minimal free resolutions of edge ideals of weighted oriented cycles. The study of edge ideals of classic cycles can be reduced to that of their underlying (unweighted and unoriented) cycles. Let $C$ be a classic cycle on $n$ vertices. Denote the underlying cycle of $C$ by $C_n$, i.e.,
$$I(C_n)=(x_1x_2,x_2x_3, \ldots, x_{n-1}x_n, x_nx_1).$$ 
Observe that the edge ideals of $C_{\sink}$ and $C$ coincide.  Then, $I(C)$ is exactly $I(C_n)$ after a change of variables, hence it suffices to find the minimal free resolution of $I(C_n)$.

We already know that the Betti numbers of $I(C_n)$ are independent of $\textrm{char} (\Bbbk) $ \cite[Theorem ~7.6.28]{Jac04}. Therefore, one can check if a Barile-Macchia resolution of $I(C_n)$ is minimal by comparing the numbers of its critical subsets with its Betti numbers. 
\begin{remark}\label{rem:cycleTable}
    Using our {\tt Macaulay2} codes from \cite{github}, we obtain the following information about the bridge-friendliness of $I(C_n)$ and whether it has a minimal Barile-Macchia resolution:
    \begin{center}
        \begin{tabular}{|c|c|c|}
        \hline
        $n$ & bridge-friendliness & minimal Barile-Macchia resolution  \\ \hline
        2,3,5,6 & \cmark      & \cmark         \\ \hline
        4,7,8,10 & \xmark      & \cmark         \\ \hline
        9 & \xmark      & \xmark         \\ \hline
    \end{tabular}
    \end{center}

In general, it is not known whether $I(C_n)$ is  bridge-friendly or has a minimal Barile-Macchia resolution, with the exception of when $n\equiv 1 \pmod 3$ from Proposition \ref{ex:notfriendlycycle}. Thus, our tools from the previous sections fall short in producing the minimal free resolution of $I(C_n)$. Our supplementary tool is the Eliahou-Kervaire splitting (abbreviated as E-K splitting) from  \cite{EK90}. The main idea behind the E-K splitting is to decompose a given monomial ideal $I$ into smaller ones and use these ideals to understand $I$. 
\end{remark}

\begin{definition}\label{def:EKsplit}
   A monomial ideal $I$ is called \textit{E-K-splittable} if it  is the sum of two non-zero monomial ideals $J$ and $K$ such that $\G(I)=\G(J) \sqcup \G(K)$  and there is a splitting function
    \begin{align*}
        \G(J \cap K ) &\to \G(J) \times \G(K)\\
        w &\mapsto (\phi(w), \varphi(w))
    \end{align*}
    satisfying the following properties:
    \begin{enumerate}[label=(\arabic*)]
        \item $w = \lcm(\phi(w), \varphi(w))$ for all $w \in \G(J\cap K)$, and
        \item for each non-empty subset $W \subseteq \G(J\cap K)$, we have that $\lcm( \phi(W))$ and $\lcm( \varphi(W))$ strictly divide $\lcm(W)$.
    \end{enumerate}
    The decomposition $I=J+K$ is called an \textit{E-K splitting}.
\end{definition}

It was shown in  \cite[Proposition 2.1]{FHV09} that if $I=J+K$ is an E-K splitting, then the minimal free resolution of $I$ can be obtained from those of $J,K$, and $J\cap K$ via the mapping cone construction. We adopt this strategy for $I(C_n)$ by setting  
$J=(x_2x_3,\dots, x_{n-1}x_n)$ and $K=(x_nx_1,x_1x_2)$
where $I(C_n)=J+K$. Note that 
\begin{multline*}
    \G(J\cap K)=\{x_1x_{n-1}x_n, x_1x_2x_3\} \cup \{ x_1x_n(x_ix_{i+1}) ~|~ 3\leq i\leq n-3 \} ~ \cup \\ \{x_1x_2(x_jx_{j+1})~|~ 4\leq j\leq n-2\}.
\end{multline*}
We may assume $n\geq 8$.

\begin{proposition}\label{prop:EKsplit}
    The decomposition $I(C_n)=J+K$ is an E-K splitting.
\end{proposition}
\begin{proof}
    Consider the following splitting function
    \begin{align*}
        \G(J \cap K ) &\to \G(J) \times \G(K)\\
        w &\mapsto (\phi(w), \varphi(w)),\\
        x_1x_{n-1}x_n&\mapsto (x_{n-1}x_n,x_1x_n),\\
        x_1x_{2}x_3&\mapsto (x_2x_3,x_1x_2),\\
        x_1x_ix_{i+1}x_n&\mapsto (x_ix_{i+1},x_1x_n),\\
        x_1x_2x_jx_{j+1}&\mapsto (x_{j}x_{j+1},x_1x_2)
    \end{align*}
    where   $3\leq i\leq n-3$ and $ 4\leq j\leq n-2$.  The first property of Definition \ref{def:EKsplit} is immediate. For the second property, consider a non-empty subset $W\subseteq \G(J\cap K)$. Observe that $\lcm(\phi(W))$ strictly divides $ \lcm(W)$ since $x_1 \mid \lcm(W)$ but  $x_1 \nmid \lcm(\phi(W))$.  Similarly, $\lcm(\varphi(W)) $ is a proper divisor of  $\lcm(W)$ because there exists $x_k$ for  $3\leq k\leq n-1$ such that $x_k \mid \lcm(W)$ but $x_k \nmid \lcm(\varphi(W))$.
\end{proof}

Notice that $J$ and $K$ are edge ideals of weighted naturally oriented paths and their minimal free resolutions are known from Theorem \ref{thm:morsefriendlyforest}. Thus, it suffices to obtain a minimal free resolution $J\cap K$. We repeat the process by showing that $J\cap K= J' +K'$ is an E-K splitting where
$$
  J'=x_1x_n(x_{n-1},x_ix_{i+1}~|~3\leq i\leq n-3)  \text{ and }
  K'=x_1x_2(x_3, x_ix_{i+1}~|~ 4\leq i\leq n-2).
$$
Then $J'\cap K'=x_1x_2x_n(x_3x_4,x_4x_5,\dots,x_{n-2}x_{n-1},x_3x_{n-1})$.

\begin{proposition}
    The decomposition $J\cap K=J'+K'$ is an E-K splitting.
\end{proposition}
\begin{proof}
  It is immediate from the following observation: $x_2x_n\mid \lcm(W)$ but $x_2\nmid \lcm(\phi(W))$ and $x_n\nmid \lcm(\varphi(W))$ any for non-empty subset $W\subseteq \G(J'\cap K')$ and any splitting function  which maps $w \in \G(J' \cap K' )$ to $(\phi(w), \varphi(w)) \in \G(J') \times \G(K')$.
\end{proof}

The ideals $J', K'$, and  $J'\cap K'$ are of very special forms: $J'$ and $K'$ resemble the edge ideal of a path of length $(n-5)$ and $J'\cap K'$ resemble that of an $(n-3)-$cycle. These ideals are obtained from edge ideals of such paths and cycles by either adding a new variable or multiplying by a monomial in new variables. Applying these two operations to a bridge-friendly and bridge-minimal ideal produce ideals with these properties:

\begin{proposition}
    Let $I$ be a monomial ideal in  $R=K[x_1,x_2,\dots, x_m]$ and $y$ a new variable. If $I$ is bridge-friendly, so are $(y)+I$ and $ yI$ as ideals in $R[y]$.
\end{proposition}

\begin{proof}
    The ideal $(y)+I$ is bridge-friendly and bridge-minimal due to Theorem \ref{thm:tensormorse}.  The same argument from that proof can be applied to $yI$.
\end{proof}

Since edge ideals of paths (unweighted and unoriented) are bridge-friendly and bridge-minimal from Theorem \ref{thm:morsefriendlyforest}. so are $J'$ and $ K'$. As a result, we can construct a minimal free resolution of $I(C_n)$ from $I(C_{n-3})$ and $I(P_{n-5})$, the edge ideal of a path of length $(n-5)$. 

\begin{corollary}\label{cor:resofclassic}
    The minimal free resolution of $I(C_n)$ can be constructed inductively from the minimal free resolutions of bridge-friendly ideals.
\end{corollary}

If a monomial ideal $I=J+K$ is E-K splittable, then it is a Betti splitting (see \cite[Proposition 3.1]{EK90} and \cite[Definition 1.1.]{FHV09}). This means  the (graded) Betti numbers and projective dimension of $I$ can be obtained from those of $J, K,$ and 
$J\cap K$. We will not derive these formulas for edge ideals of cycles since they are already obtained in \cite[Theorem 7.6.28, Corollary 7.6.30]{Jac04} using a different method. 

Finally, we note that $I(C_9)$ neither is bridge-friendly nor has a minimal Barile-Macchia resolution. Hence

\begin{corollary}\label{cor:9cycle}
    There exists a monomial ideal $I$ and an E-K splitting $I=J+K$ such that $J,K$ and  $J\cap K$ have minimal Barile-Macchia resolutions, but $I$ does not.
\end{corollary}

\section{A Comparison}\label{sec:comparison}

We devote this section to draw comparisons between Barile-Macchia resolutions and some well-known simplicial resolutions and the Scarf complex. Throughout this section, let $R$ be the polynomial ring in $N$ variables over a field, $I$ a monomial ideal of $R$ and $(>_I)$ a total ordering on $\G(I)$.

\subsection{Barile-Macchia resolutions versus simplicial resolutions}

There are several constructions that yield simplicial resolutions or complexes of a monomial ideal, e.g., Taylor resolutions \cite{Tay66} and  Lyubeznik resolutions \cite{Ly88,Nov00} among the most well-known ones. Meanwhile, the Scarf complex \cite{BPS98, Sca73} is defined to be the subcomplex of the Taylor resolution determined by the subsets with unique $\lcm$, and is not necessarily a resolution. We refer to \cite{Mer09} for the basics on these classic notions.  

In this section, we  compare Barile-Macchia resolutions with these simplicial resolutions and complexes. We begin with an observation on how Morse resolutions relate to simplicial complexes, which we believe was implicitly given in \cite{BW02}. 

\begin{proposition}\label{prop:Morsesimplicialcomplex}
    If the set of critical subsets of $\G(I)$ forms a simplicial complex, i.e., it is closed under taking subsets, then the Morse resolution coincides with the corresponding subcomplex of the Taylor resolution of $R/I$.
\end{proposition}

\begin{proof}
    We adopt the notations from Theorem \ref{thm:morseres}. Our goal is to show that the differentials in the Morse resolution are exactly those of the corresponding subcomplex of the Taylor resolution. This is equivalent to show  $\sigma''=\sigma'$ for each $\sigma'$ and $\sigma ''$ such that $\sigma'\subseteq \sigma, |\sigma'|=|\sigma|-1$, and  $\sigma''$ is a critical subset where $|\sigma|=|\sigma''|$. For the sake of contradiction, suppose the set of critical subsets form a simplicial complex and $\sigma''\neq \sigma'$ for such $\sigma'$ and $\sigma''$.  By our assumption, $\sigma'$ is ~critical.
    
    Consider a gradient path $\sigma'=\sigma_1 \to \sigma_2 \to \ldots \to \sigma_k=\sigma''$ from $\sigma'$ to $\sigma''$. By the  definition of gradient paths, we have either $(\sigma_i,\sigma_{i+1}) \notin A$  or $(\sigma_{i+1},\sigma_i) \in A$ for each $i\in[k-1]$. Observe that each directed edge $(\sigma_i,\sigma_{i+1})$ along this gradient path must be alternating between these two types of edges. If  $(\sigma',\sigma_2) \notin A$, then $(\sigma'', \sigma_{k-1}) \in A$ since $|\sigma|=|\sigma''|$. This means $\sigma''$ is not a critical subset, a contradiction. If $(\sigma_2, \sigma') \in A$, then $\sigma'$ is not a critical subset, a contradiction. This completes the proof.
\end{proof}

Lyubeznik resolutions, the Taylor resolution, and the Scarf complex of a monomial ideal are not necessarily isomorphic. These three complexes may also differ from Barile-Macchia resolutions.

\begin{theorem}\label{exampleofmultipleres}
    There exists an ideal $I$ such that its Taylor, Lyubeznik resolutions and Scarf complex are all non-isomorphic from one another, and non-isomorphic to one of its Barile-Macchia resolutions which is minimal.
\end{theorem}
\begin{proof}
    Let $I=(xy,yz,zw,wx)\subseteq R:=k[x,y,z,w]$. It turns out that all the Lyubeznik resolutions of $R/I$ are isomorphic. This can be verified  by exhausting all possible total orderings on $\G(I)$. Then, we have the following resolutions and the Scarf complex of $I$:
    \begin{align*}
        \text{Taylor resolution: } &~~ 0\leftarrow R\leftarrow R^4\leftarrow R^6 \leftarrow R^4 \leftarrow R\leftarrow 0,\\
        \text{Lyubeznik resolution: } &~~ 0\leftarrow R\leftarrow R^4\leftarrow R^5 \leftarrow R^2\leftarrow 0,\\
        \text{Scarf complex: } &~~ 0\leftarrow R\leftarrow R^4 \leftarrow R^4\leftarrow 0,\\
        \text{Barile-Macchia resolution: } &~~ 0\leftarrow R\leftarrow R^4\leftarrow R^4 \leftarrow R\leftarrow 0. \qedhere
    \end{align*}
  \end{proof}

In our experiments, Barile-Macchia resolutions performed quite well in getting closer to minimal resolutions. This suggests that whenever one of the simplicial resolutions (Taylor or Lyubeznik) or the Scarf complex of an ideal is the minimal resolution, it coincides with a Barile-Macchia resolution. This phenomena is already observed for Taylor resolutions in Remark \ref{Taylorismorseminimal}. In the next part, we investigate Scarf complexes. To the best of our knowledge, a full characterization of ideals which admit their Scarf complexes as minimal resolutions is still unknown. There are only a few classes of such ideals including strongly generic ideals \cite{BPS98,BS98}, generic ideals \cite{MSY00} and ideals that admit their Buchberger resolutions as minimal resolutions \cite{OW16}. In the late nineties, Yuzvinsky defined a large class of monomial ideals which contains strongly generic monomial ideals and proved that these ideals admit their Scarf complexes as  minimal resolutions \cite[Proposition 4.4]{Yu99} using lcm lattices. Here, we offer a much shorter and more elementary proof.

\begin{theorem}\label{YuzvinskyisMorseminimal}
    If $\lcm(\sigma)=\lcm(\tau)$ implies $\lcm(\sigma)=\lcm(\sigma\cap \tau)$ for any subsets $\sigma$ and $\tau$ of $\G(I)$, then $I$ is bridge-friendly. In particular, its minimal Barile-Macchia resolution is exactly its Scarf complex.
\end{theorem}

\begin{proof}
    We first develop some tools for this proof. Let $\lcm^{-1}(p)$ denote the set of all subsets of $\G(I)$ whose $\lcm$ equals $p$ for any monomial $p$. For the rest of this proof, we use $\P$ for   $\lcm^{-1}(p)$, i.e., $\P\coloneqq \{\sigma \subseteq \G(I) : \lcm(\sigma)=p\}$. Let  $\sigma_{\min(\P)}$ denote the smallest set in $\P$ (with respect to inclusion), if it exists. If $\P\neq \emptyset$, then  $\lcm(\cap_{\sigma\in \mathcal{P}} \sigma) = p$ by the hypotheses. In other words,   $\sigma_{\min(\P)}$ exists if and only if $\P \neq \emptyset$.

    Next, we prove the following claim: If $| \mathcal{P}|\geq 2$, then every element in $ \mathcal{P}$ either has a bridge or a true gap.  Note that  $\sigma_{\min(\P)}$ exists and $\sigma_{\min(\P)} \subseteq \sigma$ for each $\sigma\in \P$ under this assumption. Let $\sigma \in \P$ where $\sigma \neq \sigma_{\min(\P)}$. Since  $\lcm (\sigma) = \lcm (\sigma_{\min(\P)})$, any $\tau$ such that $\sigma_{\min(\P)} \subseteq \tau \subseteq \sigma$ must belong to $\P$. Thus,  each monomial in  $\sigma \setminus \sigma_{\min(\P)}$ serves as a bridge of $\sigma$. So, it suffices to show  $\sigma_{\min(\P)}$ has a true gap. Again, since any $\tau$ such that $\sigma_{\min(\P)} \subseteq \tau \subseteq \sigma$ belongs to $\P$, we can choose $\tau$ such that $|\tau|=|\sigma_{\min(\P)}|+1$. Observe that  $\tau \setminus \sbridge(\tau)=\sigma_{\min(\P)}$ because $\tau \setminus \sbridge(\tau) \in \P$ and $\sigma_{\min(\P)}$ is the smallest set in $\P$. Therefore,  $\sbridge(\sigma)$ is a true gap of $\sigma_{\min(\P)}$ by Proposition ~\ref{prop:truegap2}.

    Now we are ready to prove the bridge-friendliness of $I$ using Lemma \ref{lem:morsefri2}. Let $(>_I)$ be any total ordering on $\G(I)$, $\sigma$ a subset of $\G(I)$, and $m$ a true gap of $\sigma$. Consider a monomial $m'$ such that $m'>_I m$.  It suffices to show $m$ is a true gap of $\sigma \cup m'$. It is immediate that $m$ is a gap of $\sigma \cup m'$. Suppose $m''$ is a bridge of $\sigma \cup \{m, m'\}$ where $m>_I m''$. Then, we have 
    \[ \lcm(\sigma \cup\{m, m'\}\setminus \{m''\})=\lcm(\sigma \cup \{m, m'\})=\lcm(\sigma \cup m'). \]
    We obtain the following  chain of equalities by the hypothesis of this theorem:
    \[\lcm(\sigma \cup m')= \lcm\Big((\sigma \cup m')\cap \Big( \big( \sigma \cup m' \big)\setminus m''\Big)\Big) = \lcm\Big(\big(\sigma \cup m'\big) \setminus m'' \Big).\]
   In particular, this means $m''$ is a bridge of $\sigma \cup m'$. Hence by definition, $m$ is a true gap of $\sigma \cup m'$.
    
   Finally, we show that the corresponding Barile-Macchia resolution is indeed the Scarf complex. Since $I$ is bridge-friendly, our claim implies that any $\sigma$ where $|\lcm^{-1}(\lcm(\sigma))|\geq 2$ is not critical. On the other hand, any $\sigma$ where $|\lcm^{-1}(\lcm(\sigma))|= 1$ must be critical. Thus the set of critical subsets forms the Scarf simplicial complex. Therefore, the Barile-Macchia resolution and the Scarf complex of $I$ coincide minimally by  Proposition \ref{prop:Morsesimplicialcomplex}.
\end{proof}

The following result is immediate based on Example \ref{example:DependOnordering}, the fact that every monomial ideal in the polynomial ring in two variables is strongly generic, and Theorem \ref{YuzvinskyisMorseminimal}.

\begin{corollary}\label{cor:allidealsMorsefriendly+minimal}
    Every monomial ideal $I$ is bridge-friendly and bridge-minimal with respect to any total ordering on $\G(I)$ if and only if $\dim R\leq 2$. 
\end{corollary}

In the remainder of this subsection, we compare Barile-Macchia resolutions with Lyubeznik resolutions. These two resolutions are constructed in a similar way:  Both are obtained from  $\lcm$-homogeneous acyclic matchings using a total ordering on $\G(I)$. We recall Batzies and Welker's construction of Lyubeznik resolutions.

\begin{theorem}\cite[Theorem 3.2]{BW02}\label{def:Lyuusingmatchings}
    Let $I$ be a monomial ideal and $(>_I)$ a total ordering on $\G(I)$. 
    For any subset $\sigma=\{m_1, \ldots ,m_p\}$ of $\G(I)$ where $m_1>_I \cdots >_{I} m_p$, we define
    \[
    	v_L(\sigma)\coloneqq  \sup \big\{ k\in \mathbb{N} : \exists m \in \G(I) \text{ such that } m_k >_I m 
    	\text{ for } k\in [p] \text{ and } m\mid \lcm(m_1, \ldots, m_k) \big\}.
  \]
    If $v_L(\sigma)\neq -\infty$, define
     \[
    m_L(\sigma)\coloneqq \min_{>_I} \{m\in \G(I): m\mid \lcm(m_1,\dots, m_{v_L(\sigma)})  \}.
    \]
    For each $p\in \ZZ^N$, set
    \[
    A_p\coloneqq \{(\sigma\cup m_L(\sigma), \sigma \setminus m_L(\sigma)) :  \lcm (\sigma)=p \text{ and }v_L(\sigma)\neq -\infty \}.
    \]
    Then  $\displaystyle A=\cup_{p\in R}A_p$ is an $\lcm$-homogeneous acyclic matching. Hence, it induces a graded free resolution $\mathcal{F}_A$ of $R/I$, called a \textbf{Lyubeznik} resolution of $R/I$ (with respect to $(>_I)$).
\end{theorem}

\begin{example}\label{ex:BWmatching}
      Consider the monomial ideal $I=(m_1,m_2,m_3,m_4,m_5,m_6)$ where
\[
		m_1=x_1x_2x_3x_4, ~~m_2=x_2x_3x_5x_6, ~~m_3=x_1x_2x_5, ~~m_4=x_1x_2x_7, ~~m_5=x_2x_3x_8, ~~m_6=x_7x_8
\]
with the total ordering $ m_1 > m_2 > m_3 > m_4 > m_5 > m_6$  on $\G(I)$. 
\begin{itemize}
    \item For $\sigma_1=(m_1,m_4,m_5)$, we have $v_L(\sigma_1)=3$ and $m_L(\sigma_1)=m_6$ since  $m_6 \mid  \lcm(\sigma_1)$.
    \item For  $\sigma_2=(m_1,m_2,m_3)$, we have $v_L(\sigma_2)=2$ and $m_L(\sigma_2)=m_3$ since $\lcm(\sigma_2)$ is not divisible by $m_4,m_5$ or $m_6$  while  $m_3 \mid  \lcm(m_1,m_2)$.
    \item For $\sigma_3=(m_2,m_3,m_4)$, we have $v_L(\sigma_3)=- \infty$ since $\lcm(\sigma_3)$ is not divisible by $m_5$ or $m_6$  and  $ \lcm(m_2,m_3)$ is not divisible by $m_4,m_5$ or $m_6$.
\end{itemize}
Thus, $(\sigma_1 \cup m_6, \sigma_1)$ and $(\sigma_2, \sigma_2\setminus m_3)$ are elements of $A$ while there is no edge of $A$ involving $\sigma_3$.
\end{example}

Barile-Macchia and Lyubeznik resolutions are both induced by total orderings on $\G(I)$. Thus, it is natural to investigate how close or different these two resolutions are  under the same total  ordering.   Let $A_L$ and $A_B$ denote the matchings induced from Theorem \ref{def:Lyuusingmatchings} and Algorithm \ref{algorithm1}, respectively. Let $V_L$ (resp, $V_B$) denote the set of non-$A_L$-critical (non-$A_B$-critical) subsets of $\G(I)$, i.e.,
\[V_L=\{\sigma \mid (\sigma,\tau) \in A_L \text{ or } (\tau,\sigma)\in A_L \text{ for some subset $\tau$ of $\G(I)$} \}\]
and 
\[V_B=\{\sigma \mid (\sigma,\tau) \in A_B \text{ or } (\tau,\sigma)\in A_B \text{ for some subset $\tau$ of $\G(I)$} \}.\]
To understand the relationship between Barile-Macchia and Lyubeznik resolutions, we compare $V_L$ and $V_B$.  First, we  consider the relationship between $m_L(\sigma)$ and $\sbridge(\sigma)$ where $\sigma\subseteq\G(I)$.

\begin{lemma}\label{lem:Lyumonoisbridge}
    If $m_L(\sigma)$ exists, then it is either a bridge or a gap of $\sigma$.
\end{lemma}

\begin{proof}
   Follows from the definition of $m_L(\sigma)$.
\end{proof}

\begin{lemma}\label{lem1}
    Assume $m_L(\sigma)$ exists. If $m_L(\sigma)\in \sigma$, then $\sbridge(\sigma)$ exists. Otherwise, $\sbridge(\sigma\cup m_L(\sigma))$ ~exists.
\end{lemma}

\begin{proof}
    Follows from  the definition of $m_L(\sigma)$ and Lemma \ref{lem:Lyumonoisbridge}.
\end{proof}

\begin{lemma}\label{lem2}
    Assume $m_L(\sigma)$ exists and $m_L(\sigma)\in \sigma$ where $m_L(\sigma)\neq \sbridge(\sigma)$. Then 
    $$\sbridge(\sigma \setminus m_L(\sigma))=\sbridge(\sigma) \text{ and } m_L(\sigma\setminus \sbridge(\sigma))=m_L(\sigma).$$
\end{lemma}

\begin{proof}
    For the first equality, we first observe that $\sbridge(\sigma \setminus m_L(\sigma))\geq_I  \sbridge(\sigma)$ since the former is a subset of the latter. Hence it suffices to show that $\sbridge(\sigma)$ is a bridge of $\sigma \setminus m_L(\sigma)$. Indeed, by definition, we have
    \begin{align*}
        & \ \ \ \lcm(\sigma \setminus m_L(\sigma) \setminus \sbridge(\sigma)) \\
        &= \lcm(\{m\in \sigma : m_L(\sigma) >_I m \text{ and }m\neq \sbridge(\sigma)\} \cup \{m\in \sigma : m>_I m_L(\sigma)\})\\
        &=\lcm(\lcm(\{m\in \sigma :  m_L(\sigma) >_I m \text{ and }m\neq \sbridge(\sigma)\}),\lcm( \{m\in \sigma : m>_I m_L(\sigma)\}))\\
        &=\lcm(\lcm(\{m\in \sigma :  m_L(\sigma) >_I m \text{ and }m\neq \sbridge(\sigma)\}),\lcm( \{m\in \sigma : m\geq_I  m_L(\sigma)\}))\\
        &=\lcm(\sigma \setminus \sbridge(\sigma))\\
        &=\lcm(\sigma) \\
        & =\lcm(\sigma \setminus m_L(\sigma)).
    \end{align*}
    
    The second equality is due to  the following facts: $\sigma\setminus \sbridge(\sigma)$ and $\sigma$ only differ at the element $\sbridge(\sigma)$ and  $ m_L(\sigma) >_I\sbridge(\sigma)$. By definition, we have $v_L(\sigma\setminus \sbridge(\sigma))\geq_I  v_L(\sigma)$. Since the former subset is contained in the latter, we   have $v_L(\sigma\setminus \sbridge(\sigma))= v_L(\sigma)$. Thus  $m_L(\sigma\setminus \sbridge(\sigma))= m_L(\sigma)$ by definition.
\end{proof}

\begin{lemma}\label{lem3}
    Assume $\sigma$ has a bridge. If $m_L(\sigma\setminus\sbridge(\sigma))$ exists and $m_L(\sigma\setminus\sbridge(\sigma))\neq \sbridge(\sigma)$, then either $m_L(\sigma)=m_L(\sigma\setminus\sbridge(\sigma))$ or $\sbridge(\sigma) >_I m_L(\sigma)$.
\end{lemma}

\begin{proof}
   Since $\sigma$ contains $\sigma\setminus\sbridge(\sigma)$, we have $v_L(\sigma) \geq v_L(\sigma\setminus\sbridge(\sigma))$. Then  $m_L(\sigma)$ exists by our hypothesis. As before recall that $\sigma\setminus \sbridge(\sigma)$ and $\sigma$ only differ at the element $\sbridge(\sigma)$. We have the following two cases to consider:
    \begin{enumerate}
        \item Suppose $v_L(\sigma)=v_L(\sigma\setminus \sbridge(\sigma))$. Then, by definition, $m_L(\sigma)=m_L(\sigma\setminus\sbridge(\sigma))$.
        \item Suppose $v_L(\sigma)>_I v_L(\sigma\setminus \sbridge(\sigma))$. Then, by definition, $\sbridge(\sigma) >_I m_L(\sigma)$. \qedhere
    \end{enumerate}
\end{proof}

\begin{lemma}\label{lem4:potentialtype2}
    Let $\sigma$ be a subset of $\G(I)$ such that $m_L(\sigma)$ and  $\sbridge(\sigma)$ exist where $m_L(\sigma)\neq \sbridge(\sigma)$, $m_L(\sigma)\notin \sigma$ and $\sbridge(\sigma \cup m_L(\sigma))=\sbridge(\sigma)$. Assume $\sigma$ is potentially-type-2, but not type-2, i.e., there exists $\sigma'$ such that $\sigma'\setminus \sbridge(\sigma')=\sigma\setminus \sbridge(\sigma)$ and $\sbridge(\sigma) >_I\sbridge(\sigma')$. Then $m_L(\sigma')$ exists and either $m_L(\sigma')=m_L(\sigma)$ or $\sbridge(\sigma')>_I m_L(\sigma')$. In particular, $m_L(\sigma')\notin \sigma'$.
\end{lemma}

\begin{proof}
    Note that $m_L(\sigma \cup m_L(\sigma))=m_L(\sigma)$ and $\sbridge(\sigma)$ is a bridge of $\sigma\cup m_L(\sigma)$  by definition. It follows from Lemma \ref{lem:Lyumonoisbridge} and the hypothesis that  $m_L(\sigma)>_I \sbridge(\sigma)=\sbridge(\sigma \cup m_L(\sigma))$. Next observe that $\sigma$ and $\sigma'$ only differ at the elements  $\sbridge(\sigma)$ and $\sbridge(\sigma')$. Then, for any $m>_I  m_L(\sigma)$, $m\in \sigma$ if and only if $m\in \sigma'$. Furthermore, $v_L(\sigma')$ is finite since $v_L(\sigma') \geq v_L(\sigma)$. Hence, $m_L(\sigma')$ ~exists. 
    \begin{enumerate}
        \item If $v_L(\sigma')=v_L(\sigma)$, then  $m_L(\sigma')=m_L(\sigma)$. Since $m_L(\sigma)\notin \sigma$, we have $m_L(\sigma')\notin \sigma'$.
        \item If $v_L(\sigma')>v_L(\sigma)$, one can observe that $\sbridge(\sigma')>_I m_L(\sigma')$ under these assumptions. Hence $m_L(\sigma')\notin \sigma'$ by Lemma \ref{lem:Lyumonoisbridge}. \qedhere
    \end{enumerate}
\end{proof}

\begin{lemma}\label{lem:potentialtype2-2}
    Let $\sigma$ and $ \tau$ be  subsets of $\G(I)$ such that $\sbridge(\sigma), \sbridge(\tau), $ and $ m_L(\tau)$ exist where $m_L(\tau)\in \tau$, $\sigma \setminus \sbridge(\sigma)=\tau \setminus \sbridge(\tau)$ and $\sbridge(\tau) >_I \sbridge(\sigma)$. Then $m_L(\sigma)$ exists. Moreover, if $m_L(\tau)\neq \sbridge(\tau)$, then either $m_L(\sigma)=m_L(\tau)$ or $\sbridge(\sigma) >_I m_L(\sigma)$.
\end{lemma}

\begin{proof}
    As before, note that $\sigma$ and $\tau$ only differ at elements $\sbridge(\sigma)$ and $\sbridge(\tau)$. It follows from Lemma \ref{lem:Lyumonoisbridge} and our hypothesis that $ m_L(\tau) \geq_I\sbridge(\tau)>_I \sbridge(\sigma) $. Then, for any $m>_I m_L(\tau)$, $m\in \tau$ if and only if $m\in \sigma$. Hence, $v_L(\sigma)$ is finite and $m_L(\sigma)$ exists.  Now assume $m_L(\tau)\neq \sbridge(\tau)$. 
    \begin{enumerate}
        \item Suppose $v_L(\sigma)=v_L(\tau)$. Then $m_L(\sigma)=m_L(\tau)$.
        \item Suppose $v_L(\sigma)>v_L(\tau)$. Observe that  $\sbridge(\sigma)>_I m_L(\sigma)$ under these assumptions. \qedhere
    \end{enumerate}
\end{proof}

Now we are ready to prove the main result of this section: A Barile-Macchia resolution is at most as ``big" as the corresponding Lyubeznik resolution.  In order to illustrate possible scenarios, we will use green and blue arrows to depict directed edges in $A_L$ and $A_B$, respectively. Additionally, we use blue dashed arrows to depict directed edges of the form $(\sigma,\sigma\setminus \sbridge(\sigma))$ where $\sigma$ is potentially-type-2, but not type-2.

\begin{theorem}\label{BridgevsLyu}
    Let $I$ be a monomial ideal. Assume there exist no two subsets $\sigma$ and $\tau$ of $\G(I)$ such that $\sigma \setminus \sbridge(\sigma)=\tau \setminus \sbridge(\tau)$, $\sbridge(\tau)>_I \sbridge(\sigma)$, $m_L(\sigma)=m_L(\tau)>_I \sbridge(\tau)$, and $m_L(\tau)\in \tau$. Then $V_L\subseteq V_B$. In particular, $\rank (\mathcal{F}_L)_i\geq \rank (\mathcal{F}_B)_i$ for each $i\in \mathbb{Z}$ where $\mathcal{F}_L$ and $\mathcal{F}_B$ are the induced Lyubeznik and  Barile-Macchia resolutions, respectively.
\end{theorem}
    
\begin{proof}
    Recall that $m_L(\sigma \cup m_L(\sigma))=m_L(\sigma)$ by definition. Thus, it suffices to show that for each $\tau \in V_L$ satisfying $(\tau,\tau \setminus m_L(\tau))\in A_L$,  both $\tau$ and $\tau \setminus m_L(\tau)$ appear in some edges of $A_B$. To be more specific, these two subsets of $\G(I)$ can appear in $A_B$ in one of the following three ~patterns:     
    
    \begin{figure}[H]
        \centering
        \includegraphics[width=0.98\textwidth]{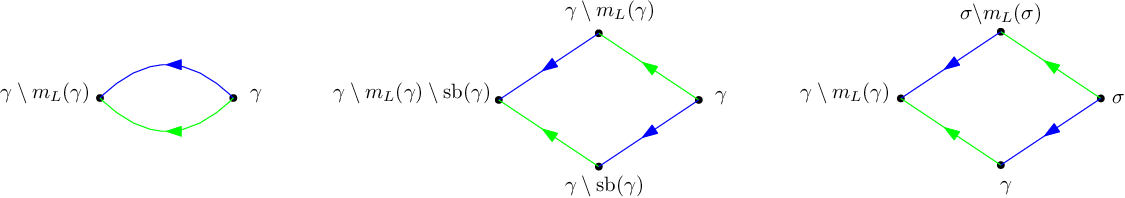}
        \caption{The only possible patterns for $\gamma$ and $\gamma \setminus m_L(\gamma)$ (from left to right)}
        \label{threepatterns}
    \end{figure}

   The first pattern occurs if $m_L(\tau)= \sbridge(\tau)$, the second pattern occurs if $m_L(\tau)\neq \sbridge(\tau)$ (the second diagram follows from Lemma \ref{lem2}), and the third pattern occurs if there exists a subset $\sigma$ satisfying the last diagram.

    Note that the assumption  $(\tau,\tau \setminus m_L(\tau))\in A_L$ for $\tau \in V_L$  implies $m_L(\tau)$ exists and $m_L(\tau)\in \tau$. Then, $\sbridge(\tau)$ exists by Lemma \ref{lem1}. If $V_L=\emptyset$, there is nothing to prove. So, we may assume $V_L\neq \emptyset$. In order to show one of the three patterns hold for $\tau$, we use descending induction on ~$|\tau|$.

    We start with the base case. Consider the largest $|\tau|$ such that $\tau \in V_L$ satisfying  $(\tau,\tau \setminus m_L(\tau))\in A_L$. Observe that if $\tau$ belongs to $V_L$, then any subset of $\G(I)$ containing $\tau$ also belongs to $V_L$ because a Lyubeznik resolution is induced by a simplicial complex. Hence, $\G(I) \in V_L$ since $V_L\neq \emptyset$. This means $m_L(\G(I))$ exists and $(\G(I), \G(I) \setminus m_L(\G(I))) \in A_L$.  Thus, we have $\tau=\G(I)$ for the base case. If $m_L(\tau)=\sbridge(\tau)$, then the first pattern occurs. So, we may assume $m_L(\tau)\neq \sbridge(\tau)$. It follows from Lemma \ref{lem2} that $\tau\setminus m_L(\tau)$ has a bridge. Therefore, $\tau\setminus m_L(\tau)$ must be type-1 or potentially-type-2. Since $m_L(\tau)\neq \sbridge(\tau)$, it is not  type-1, and hence it is potentially-type-2. In fact, it is type-2 since, by Lemma \ref{lem2}, we have $\sbridge(\tau\setminus m_L(\tau))=\sbridge(\tau)=\sbridge(\G(I))$, which is the smallest possible bridge of any subset of $\G(I)$. Therefore the second pattern occurs by Lemma \ref{lem2}. 

    Note that if $\tau$ follows the second pattern, then $\tau \setminus \sbridge(\tau)$ follows the third pattern. 
    
    Fix a natural number $n$. By induction, we can assume that whenever $(\tau,\tau \setminus m_L(\tau))\in A_L$ for a subset $\tau$ of $\G(I)$ where $|\tau|>n$, then $\tau$ satisfies one of the three patterns. 
    
    Consider $\tau\in V_L$ where  $(\tau,\tau \setminus m_L(\tau))\in A_L$ and $|\tau|=n$. We may assume $\sbridge (\tau)$ is the smallest in the following sense: If there exists  $\tau'$ where  $(\tau',\tau' \setminus m_L(\tau'))\in A_L$ with $|\tau'|=n$ and $\sbridge(\tau)>_I ~\sbridge(\tau')$, then $\tau'$ must follow one of the three patterns. We call this new assumption $(*)$. Again, since $\tau$ has a bridge, it must be type-1 or potentially-type-2. Based on this observation, we consider the following cases:
        \begin{enumerate}
            \item Suppose $\tau$ is type-1. By the definition of type-1 subsets, there exists a subset $\sigma$ of $\G(I)$ such that $\sigma\setminus \sbridge(\sigma)=\tau$. Since $ \tau \subset \sigma$ and $m_L(\tau)$ exists, so does  $m_L(\sigma)$. Note that $m_L(\sigma)\neq \sbridge(\sigma)$. Otherwise,  $(\sigma, \tau)\in A_L$ which contradicts the fact that $A_L$ is a matching.  We have two subcases.
\begin{enumerate}
            \item Suppose $m_L(\sigma)\in \sigma$. Then we have the following diagram:
                \begin{figure}[H]
                    \centering
            \includegraphics[width=0.28\textwidth]{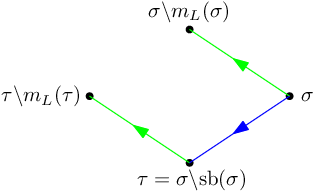}
                    \label{1a}
                \end{figure}
                By the induction hypothesis,  $\sigma$ follows one of the three patterns. Note that it does not follow the first pattern since $m_L (\sigma)\neq \sbridge(\sigma)$. So,  we have the following two possible diagrams.
                \begin{figure}[H]
                    \centering
            \includegraphics[width=0.65\textwidth]{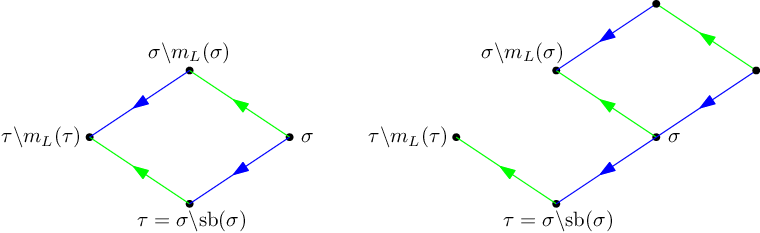}
                    \label{1a1}
                
                \end{figure}
                The second diagram contradicts the fact that $A_B$ is a matching. So, the only possible diagram is the first one which coincides with the third pattern from our list. Thus $\tau$ follows the third pattern for this case.

\item Suppose $m_L(\sigma)\notin \sigma$. Then we have the following diagram:
                \begin{figure}[H]
                    \centering
 \includegraphics[width=0.35\textwidth]{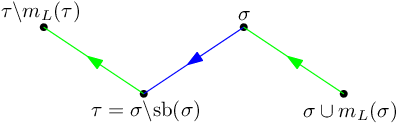}
                    \label{1b}
                \end{figure}
                By the induction hypothesis,  $\sigma\cup m_L(\sigma)$ follows one of the three patterns. Note that the first pattern contradicts the fact 
 that $A_B$ is a matching. Recall from  Lemma \ref{lem1} that $\sbridge (\sigma \cup m_L(\sigma))$ exists.  So, we have the following two possible diagrams:
                \begin{figure}[H]
                    \centering
                    \includegraphics[width=0.85\textwidth]{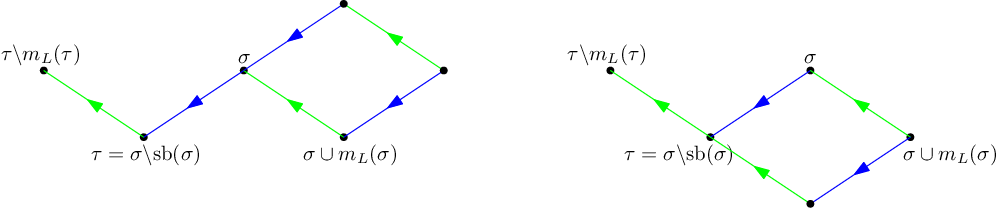}
                    \label{1b1}
                \end{figure}
                The first diagram  contradicts the fact that $A_L$ is a matching while the second one contradicts that $A_B$ is a matching. This means  $m_L(\sigma) \in \sigma$ when $\tau$ is type-1. 
     \end{enumerate}       
            \item Suppose $\tau$ is type-2. If $m_L(\tau)=\sbridge(\tau)$, then $\tau$ follows the first pattern and we are done.  So we may assume $m_L(\tau)\neq \sbridge(\tau)$. It follows from Lemma \ref{lem2} that $\tau\setminus m_L(\tau)$ has a bridge and $\sbridge (\tau\setminus m_L(\tau))=\sbridge(\tau)$. Hence,  $\tau\setminus m_L(\tau)$ is either type-1 or potential type-2. 

  \begin{enumerate}
  \item Suppose $\tau\setminus m_L(\tau)$ is type-1, i.e., there exists $\sigma$ such that $\sigma \setminus \sbridge(\sigma)=\tau\setminus m_L(\tau)$. Note that
$m_L (\sigma \setminus \sbridge(\sigma)) = m_L(\tau)$ and $m_L(\sigma)$ exists. Since $A_B$ is a matching, we have $m_L(\tau) \neq \sbridge(\sigma)$. It then follows from Lemma \ref{lem3} that $ \sbridge(\sigma) >_I m_L(\sigma)$. Hence,  $m_L(\sigma)\notin \sigma$ by Lemma \ref{lem:Lyumonoisbridge} and we have the following diagram:
                \begin{figure}[H]
                    \centering
             \includegraphics[width=0.45\textwidth]{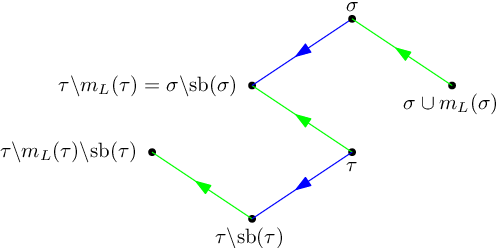}
                    \label{2a}
                \end{figure}
                By the induction hypothesis,  $\sigma\cup m_L(\sigma)$  follows one of the three patterns.  Note that the first pattern contradicts the fact that $A_B$ is a matching. Combining this with Lemma \ref{lem2}, we have the following two diagrams:
                \begin{figure}[H]
                    \centering
                    \includegraphics[width=0.95\textwidth]{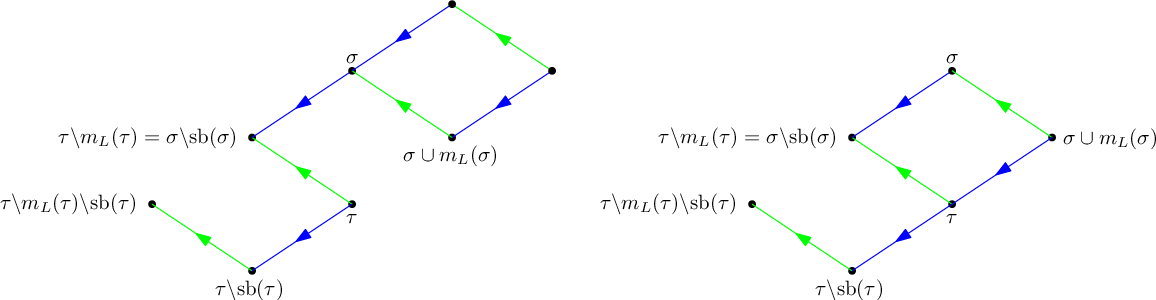}
                    \label{2a1}
                \end{figure}
                Neither of these diagrams is possible as they contradict the fact that $A_B$ is a ~matching.

  \item Suppose $\tau\setminus m_L(\tau)$ is type-2. Then, $\sigma$ follows the second pattern from our list by Lemma \ref{lem2}. 
                \begin{figure}[H]
                    \centering
                    \includegraphics[width=0.35\textwidth]{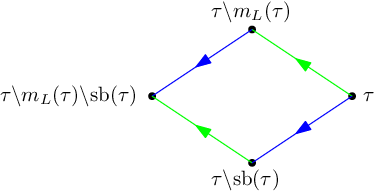}
                    \label{2b}
                \end{figure}
  \item  Suppose $\tau\setminus m_L(\tau)$ is potentially-type-2, but not type-2. To simplify the notation, set $\sigma=\tau\setminus m_L(\tau)$. Then, there exists $\sigma'$ such that $\sigma' \setminus \sbridge(\sigma')= \sigma \setminus \sbridge(\sigma)$ and $\sbridge(\sigma) >_I \sbridge(\sigma')$.  Note that $m_L (\sigma) =m_L (\tau)$.  It follows from Lemma \ref{lem4:potentialtype2} that $m_L(\sigma')$ exists and $m_L(\sigma')\notin \sigma'$. We have the following diagram: 

                \begin{figure}[H]
                    \centering
                    \includegraphics[width=0.4\textwidth]{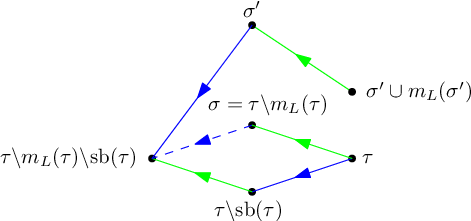}
                    \label{2c}
                \end{figure}
                By Lemma \ref{lem2}, 
                \[\sbridge(\tau)=\sbridge(\sigma \cup m_L(\sigma) )=\sbridge(\sigma)>_I \sbridge(\sigma')=\sbridge(\sigma'\cup m_L(\sigma')).\]
                Hence by $(*)$, $\sigma'\cup m_L(\sigma')$  follows one of the three patterns. Note that the first pattern is not possible.  Combining this with Lemma \ref{lem2}, we have the following two possible diagrams:
                \begin{figure}[H]
                    \centering
                    \includegraphics[width=0.85\textwidth]{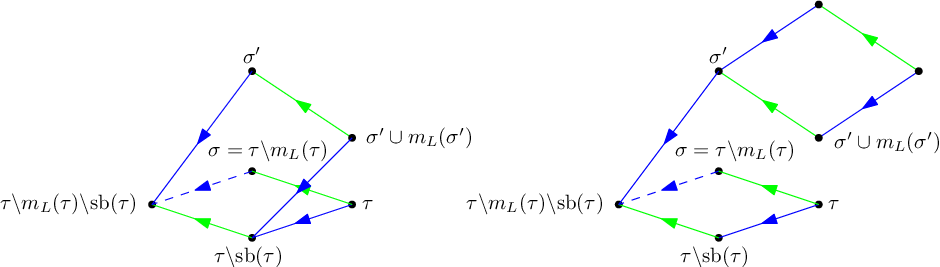}
                    \label{2c1}
                \end{figure}
                Both diagrams contradict the fact that $A_B$ is a matching.
\end{enumerate}
            \item Suppose $\tau$ is potentially-type-2, but not type-2. Then, there exists a subset $\sigma$ of $\G(I)$ such that $\sigma \setminus \sbridge(\sigma) = \tau \setminus \sbridge(\tau)$ and $\sbridge(\tau) >_I\sbridge(\sigma)$. By Lemma \ref{lem:potentialtype2-2}, $m_L(\sigma)$ exists. We have four subcases:
            \begin{enumerate}
                \item Suppose $m_L(\sigma)\in \sigma$ and $m_L(\tau)=\sbridge(\tau)$. Then we have a diagram.
                \begin{figure}[H]
                    \centering
                    \includegraphics[width=0.28\textwidth]{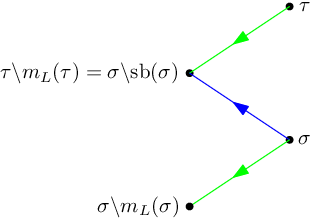}
                    \label{3a}
                \end{figure}
                By $(*)$,  $\sigma$  follows one of the three patterns. Note that it cannot follow the first and third patterns. Combining this with Lemma \ref{lem2}, we only have one possibility:
                \begin{figure}[H]
                    \centering
                    \includegraphics[width=0.28\textwidth]{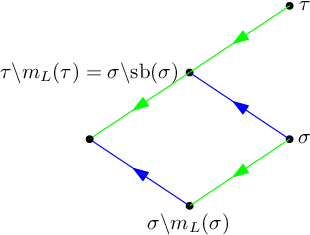}
                    \label{3a1}
                \end{figure}
                This situation, however, contradicts the fact that $A_L$ is a matching.
                
                \item Suppose $m_L(\sigma)\notin \sigma$ and $m_L(\tau)= \sbridge(\tau)$.  In this case, we have the following ~diagram: 
                \begin{figure}[H]
                    \centering
                    \includegraphics[width=0.35\textwidth]{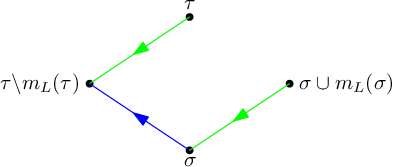}
                    \label{3b}
                \end{figure}
                By the induction hypothesis,  $\sigma\cup m_L(\sigma)$ follows one of the three patterns. However, the first pattern is not possible.  Combining this with Lemma \ref{lem2}, we have the following two possible diagrams:
                \begin{figure}[H]
                    \centering
                    \includegraphics[width=0.6\textwidth]{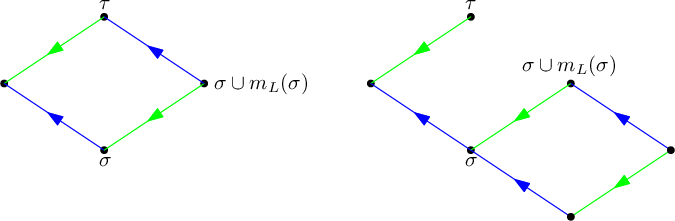}
                    \label{3b1}
                \end{figure}
                In the first diagram,   $\tau$ satisfies the third pattern. Meanwhile the second diagram is not possible since $A_B$ is a matching.

                \item Suppose $m_L(\tau)\neq \sbridge(\tau)$. Then,  either $m_L(\sigma)=m_L(\tau)$ or $\sbridge(\sigma) >_I m_L(\sigma)$ by Lemma \ref{lem:potentialtype2-2}. The first case is not possible since the existence of such $\sigma$ and $\tau$ contradicts our hypotheses.  Suppose $\sbridge(\sigma) >_I m_L(\sigma)$. Then $m_L(\sigma)\notin \sigma$ by Lemma ~\ref{lem:Lyumonoisbridge}. In this case, we have the following diagram:
                    \begin{figure}[H]
                        \centering
                        \includegraphics[width=0.35\textwidth]{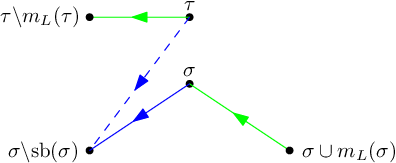}
                        \label{3c}
                    \end{figure}
                    By the induction hypothesis, $\sigma \cup m_L(\sigma)$ follows one of the three patterns. However, the first pattern is not possible. Combining this with Lemma \ref{lem2}, we have two possible diagrams:
                    \begin{figure}[H]
                        \centering
                       \includegraphics[width=0.6\textwidth]{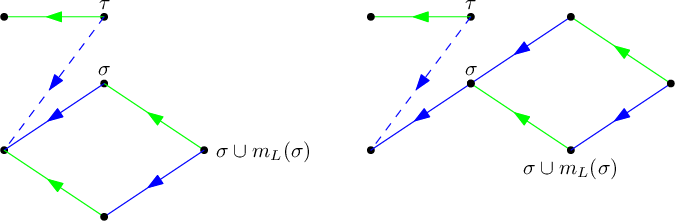}
                        \label{3c1}
                    \end{figure}
                     The second situation contradicts the fact that $A_B$ is a matching. On the other hand, since $m_L(\sigma)\geq_I\sbridge (\sigma \cup m_L(\sigma)) = \sbridge(\sigma)$ due to Lemma \ref{lem2} and the fact that $m_L(\sigma)$ is a bridge of $\sigma \cup m_L(\sigma)$, the first situation is not possible, either. \qedhere 
            \end{enumerate}
        \end{enumerate}
\end{proof}

The assumption cannot be removed, as can be seen in the next example.

\begin{example}\label{ex:CompareLyuBri}
    Let $I$ be the monomial ideal given in Example \ref{ex:BWmatching}. Consider the following total orderings $(>_1)$ and $(>_2)$ on $\G(I)$:
    \begin{align*}
        &m_1 >_1 m_2 >_1 m_3 >_1 m_4 >_1 m_5 >_1 m_6,\\
        &m_1 >_2 m_2 >_2 m_4 >_2 m_5 >_2 m_6 >_2 m_3.
    \end{align*}
    The Lyubeznik and Barile-Macchia resolutions with respect to $(>_1)$ are
    \begin{align*}
        \mathcal{F}_L^1: &~~0 \leftarrow  R \leftarrow R^6\leftarrow R^{13} \leftarrow R^{12} \leftarrow R^4\leftarrow 0\leftarrow 0,\\
         \mathcal{F}_B^1: &~~0\leftarrow R \leftarrow R^6\leftarrow R^{9} ~~\leftarrow R^{6} ~~\leftarrow R^3\leftarrow R\leftarrow 0,
    \end{align*}
    respectively. We observe that Theorem \ref{BridgevsLyu} does not hold for $(>_1)$: $\rank (\mathcal{F}_B^1)_i\leq \rank (\mathcal{F}_L^1)_i$ for each $i\neq 4$, but $\rank (\mathcal{F}_B^1)_4=1>0= \rank (\mathcal{F}_L^1)_4$. On the other hand, the Lyubeznik and Barile-Macchia resolutions with respect to $(>_2)$ are
    \begin{align*}
        \mathcal{F}_L^2: &~~0\leftarrow  R \leftarrow R^6\leftarrow R^{12} \leftarrow R^{10} \leftarrow R^3\leftarrow 0,\\
        \mathcal{F}_B^2: &~~0\leftarrow  R \leftarrow R^6 \leftarrow R^{9} ~~\leftarrow R^{5} ~\leftarrow R~\leftarrow 0,
    \end{align*}
    respectively. Hence Theorem \ref{BridgevsLyu} holds for $(>_2)$. In fact, one can check that $\mathcal{F}_B^2$ is minimal.
\end{example}

With more care, one can show that a monomial ideal that does not satisfy the assumption in Theorem \ref{BridgevsLyu} requires at least $6$ generators.  Hence, the ideal from Example \ref{ex:CompareLyuBri} is the smallest example where Theorem \ref{BridgevsLyu} does not hold. As a consequence, we obtain the following:

\begin{corollary}\label{cor:5generators}
    If $I$ is a monomial ideal with $|\G(I)|\leq 5$, then for any Lyubeznik resolution $\mathcal{F}_L$, there exists a Barile-Macchia resolution $\mathcal{F}_B$ such that $\rank (\mathcal{F}_B)_i\leq \rank (\mathcal{F}_L)_i$ for each $i\in \mathbb{Z}$.
\end{corollary}

\subsection{Generalization of Barile-Macchia resolutions}

In \cite{BW02}, Batzies and Welker provided a generalization of Lyubeznik resolutions. We state that result below in a less general form. Note that we will identify each monomial  with an element in $\mathbb{Z}^N$. 

\begin{theorem}\cite[Theorem 3.2]{BW02}\label{def:LBW}
    Let $I$ be a monomial ideal and $(X,\lcm)$ be the $\mathbb{Z}^N$-graded complex supporting the Taylor resolution of $R/I$. Let $P$ be a poset, $f$ an $\lcm$-compatible $P$-grading of $X$ and $(>_p)_{p\in P}$ be a sequence of total orderings of $\G(I)$. Assume the following condition for all $p\in P$ and all  subsets $\sigma$ of $\G(I)$ such that $f(\sigma)=p$ and  $\sigma=\{m_i>_p \cdots >_p m_q\}$: 
    
    if $m\mid \lcm(m_i,\dots,m_t)$ for some $ 0\leq t<q$ and $  m_t>_p m$, then $f(\sigma \setminus m )=f(\sigma\cup m)=f(\sigma)=p$.

    For $\sigma=\{ m_1 >_{f(\sigma)} \cdots >_{f(\sigma)} m_q \}$, we define
   \[
    	v_L(\sigma)\coloneqq  \sup \big\{ k\in \mathbb{N} : \exists m \in \G(I) \text{ such that } m_k >_{f(p)} m 
    	 \text{ for } k\in [q] \text{ and } m\mid \lcm(m_1, \ldots, m_k) \big\}.
   \]
    If $v(\sigma)\neq -\infty$, set
    \[
    m_L(\sigma)\coloneqq \min_{>_{f(p)}} \{m\in \G(I): m\mid \lcm(m_1,\dots, m_{v_L(\sigma)})  \}.
    \]
    For each $p\in P$ set
    \[
    A_p\coloneqq \{(\sigma\cup m_L(\sigma),  \sigma \setminus m_L(\sigma)) ~|~  f(\sigma)=p \text{ and }v_L(\sigma)\neq - \infty \}.
    \]
    Then $A=\cup_{p\in P}A_p$ is an $f$-homogeneous acyclic matching. Hence, it induces a graded free resolution $\mathcal{F}_A$ of $R/I$, called \textbf{generalized Lyubeznik.}
\end{theorem}

We can obtain a similar generalization for Barile-Macchia resolutions.

\begin{theorem}\label{def:generalizedbridge}
    Let $I$ be a monomial ideal and $(X,\lcm)$ be the $\mathbb{Z}^N$-graded complex supporting the Taylor resolution of $R/I$. Let $P$ be a poset, $f$ an $\lcm$-compatible $P$-grading of $X$ and $(>_p)_{p\in P}$ a sequence of total orderings of $\G(I)$. Assume $f(\sigma \setminus \sbridge_{>_p}(\sigma) )=f(\sigma)$ for all $p\in P$ and all subsets $\sigma$ of $\G(I)$ such that $f(\sigma)=p$. For each $p\in P$, let $A_p$ be the $f$-homogeneous acyclic matching obtained by applying Algorithm \ref{algorithm1} to the set $f^{-1}(p)$ imposed with the total ordering $(>_p)$. Then $A=\cup_{p\in P}A_p$ is an $f$-homogeneous acyclic matching. Hence, it induces a graded free resolution $\mathcal{F}_A$ of $R/I$.
\end{theorem}

\begin{proof}
    The proof is exactly the same as that of Theorem \ref{thm:proofalg}. One can verify with definitions that $A$ is an $f$-homogeneous matching. Hence it suffices to show that it is acyclic. Using the same arguments as in the proof of Lemma \ref{clm:directedcycle}, one can verify that all subsets $\sigma$ of $\G(I)$ in a directed cycle have the same $f(\sigma)$. Hence a directed cycle exists in $G_{X}^A$ if and only if a directed cycle exists in the restriction of $G_{X}^A$ to $f^{-1}(p)$ for some $p\in P$. Then using the same arguments as in the proof of Theorem \ref{thm:proofalg}, one can verify that the existence of a directed cycle leads to a contradiction. Therefore $A$ is an $f$-homogeneous acyclic matching.
\end{proof}

In the same spirit of naming the Lyubeznik resolutions and their generalizations, we call these resolutions   \textbf{generalized Barile-Macchia}.
We obtain a direct generalization of Theorem \ref{BridgevsLyu} for the generalized version of Barile-Macchia and  Lyubeznik resolutions.

\begin{theorem}\label{BridgevsLyu3}
    Let $I$ be a monomial ideal and $(X,\lcm)$ be the $\mathbb{Z}^N$-graded complex supporting the Taylor resolution of $R/I$. Let $P$ be a poset, $f$ an $\lcm$-compatible $P$-grading of $X$ and $(>_p)_{p\in P}$ be a sequence of total orderings of $\G(I)$. Let $\mathcal{F}_L$ and $\mathcal{F}_B$ denote the generalized Lyubeznik and generalized Barile-Macchia resolutions induced by Theorem \ref{def:LBW} and Theorem \ref{def:generalizedbridge}, respectively. Assume the following:
    \begin{enumerate}
        \item For all $p\in P$ and all  subsets $\sigma$ of $\G(I)$ such that $f(\sigma)=p$ and $\sigma=\{m_1 >_p \cdots >_p m_q\}$, we have the condition: if $m\mid \lcm(m_1,\dots,m_t)$ for some $0\leq t<q$ and $ m_t>_p m$, then $ f(\sigma \setminus m )=f(\sigma\cup m)=f(\sigma)=p$. 
        \item For all $p\in P$ and all subsets $\sigma$ of $\G(I)$ such that $f(\sigma)=p$ the condition
         \[
            f(\sigma \setminus \sbridge_{>_p}(\sigma) )=f(\sigma).
         \]
        \item For each $p\in P$, the condition in Theorem \ref{BridgevsLyu} holds for $f^{-1}(p)$.
    \end{enumerate}
    Then $\rank (\mathcal{F}_B)_i\leq \rank (\mathcal{F}_L)_i$ for each $i\in \mathbb{Z}$.
\end{theorem}

\section{Concluding Remarks, Conjectures and Questions}\label{sec:final}

In this section, we conclude our study of the Barile-Macchia resolutions by presenting several questions and conjectures.  In previous sections, we identified several classes of ideals that are bridge-friendly (or have a minimal Barile-Macchia resolution)). In general, it is difficult to tell whether a given monomial ideal has these properties. One way to complete this task is to exhaust all the possible total orderings, which is  not practical for ideals with many generators. So, we raise the following natural question:

\begin{question}
   What characterizes bridge-friendly monomial ideals or when a monomial ideal has a minimal Barile-Macchia resolution?  
\end{question}
 
Given the quite general nature of the previous question, one can pose a more tractable one: 

\begin{question}\label{ques:nonfriendlynorminimal}
    What is the smallest example, in terms of the number of minimal generators, of a monomial ideal that is not bridge-friendly or does not have a minimal Barile-Macchia resolution?
\end{question} 

It is straightforward that any monomial ideal with at most three generators is always bridge-friendly with respect to any total ordering. For bridge-friendliness, the answer is four (Example \ref{notmorseminimalexample}). The story is more complicated when it comes to having a minimal Barile-Macchia resolution. One can prove the following, recovering parts of a recent result \cite[Theorem 4.5]{faridi22}.

\begin{theorem}
    Let $I$ be a monomial ideal with at most four generators. Then $I$ has a minimal Barile-Macchia resolution. In particular, any monomial ideal with at most four generators has a minimal cellular free resolution.
\end{theorem}\label{thm:4generators}
\begin{proof}[Sketch of the proof.] 
    Set $\G(I)=\{m_1,m_2,m_3,m_4\}$. Consider all subsets $S$ of $\G(I)$ such that $\lcm(S)=\lcm(\G(I))$. Abusing notations, let $S_0$ denote such a subset with minimum cardinality (there could be many such $S_0$). Consider the cases when $|S_0|=2,3$ or $4$. 
    
    Without loss of generality, set $S_0=\{m_1,\dots,m_{|S_0|}\}$. Let $(>)$ denote the total ordering $m_1>m_2>m_3>m_4$ and $\mathcal{F}$ be the induced Barile-Macchia resolution. Suppose $\mathcal{F}$ is not minimal, then there exist subsets $\sigma$ and $\tau$ of $\G(I)$ such that $\lcm(\sigma)=\lcm(\tau), |\sigma| =|\tau|+1$ and there is a gradient path from $\sigma$ to $\tau$. By investigating all the possible Barile-Macchia matchings induced by $(>)$, one can always find a contradiction.
\end{proof}

One may ask whether this result holds for monomial ideals with more generators. The following example shows that there exists a monomial ideal with six generators that does not have a Barile-Macchia resolution.

\begin{example} 
    This example was introduced in \cite{KTY09}. Consider $$I=(m_1,m_2,m_3,m_4,m_5,m_6)$$ where $m_1=x_1x_2x_8x_9x_{10},~~
    	m_2=x_2x_3x_4x_5x_{10},~~
    	m_3=x_5x_6x_7x_8x_{10}, ~~
    	m_4=x_1x_4x_5x_6x_{9}, ~~$  
    $~~~~~$  $m_5=x_1x_2x_3x_6x_{7}$ and $	m_6=x_3x_4x_7x_8x_{9}$.

    The total Betti numbers of $R/I$ depend on $\textrm{char} (\Bbbk) $ \cite[3.3]{BMSW22}.  Hence, $I$ does not have a minimal Barile-Macchia resolution since Morse resolutions, in general, are cellular \cite[Proposition 2.2]{BW02}.
\end{example}

Therefore to get a full picture, the question now is whether Theorem \ref{thm:4generators} holds for monomial ideals with five generators. We conjecture that the answer is affirmative.

\begin{conjecture}
    All monomial ideals with at most five generators are bridge-minimal.
\end{conjecture}

One may ask a similar question:

\begin{question}\label{ques:nonfriendlynorminimal2}
    What is the smallest example, in terms of the number of variables, of a monomial ideal that is not bridge-friendly or does not have a minimal Barile-Macchia resolution?
\end{question} 

Corollary \ref{cor:allidealsMorsefriendly+minimal} provide a lower bound. For bridge-friendliness, an upper bound was obtained in Example \ref{ex:notfriendlycycle}. On the other hand, to find a monomial ideal that does not have a minimal Barile-Macchia resolution, it is sufficient find an ideal whose Betti numbers depend on $\textrm{char} (\Bbbk) $ or minimal free resolution is not cellular, i.e., not supported by any CW-complex (\cite[Proposition ~2.2]{BW02}). The smallest example of the former is the Stanley-Reisner ideal of a minimal triangulation of $\mathbb{R}P^2$  in a polynomial ring of dimension $6$ (see \cite{BH98}) while that of the latter is  a nearly Scarf ideal  in a polynomial ring of dimension $284$ \cite[Lemma 3]{Vel08}. Therefore, these two examples provide an upper bound for Question \ref{ques:nonfriendlynorminimal2}. 

Let $I$ be one of the ideals considered in this paper, namely edge ideals of weighted oriented forests or cycles. Notice that if $I$ has a minimal Barile-Macchia resolution, then  so does a monomial ideal $J$ where $\G(J)\subseteq \G(I)$. Motivated by this observation, we pose the following question:

\begin{question}
    Let $I$ and $J$ be monomial ideals where $\G(J)\subseteq \G(I)$. If $I$ has a minimal Barile-Macchia resolution, does $J$ have one as well? 
\end{question}

Free resolutions of an ideal may depend on $\textrm{char} (\Bbbk) $. However, for all monomial ideals, one can find a free resolution that does not (e.g., Taylor, Lyubeznik or Barile-Macchia resolutions). This is true even for monomial ideals whose minimal free resolutions depend on $\textrm{char} (\Bbbk) $. Then, it is natural to ask whether one can identify the  shortest length of such a resolution.  This question is one of the motivations behind our work in Section \ref{sec:comparison}. Recall that Lyubeznik resolutions are closer to minimal than Taylor resolutions. So, in Theorem \ref{BridgevsLyu}, we compared the Lyubeznik and Barile-Macchia resolutions with respect to a fixed ordering with the hope of obtaining an insight into this last question.  As the next step, we ask whether similar results as in Theorem \ref{BridgevsLyu} hold for those with respect to different orderings. Numerous examples suggest the following ~conjecture:

\begin{conjecture}
    Let $I$ be a monomial ideal. If $\mathcal{F}_L$ is a Lyubeznik resolution of $R/I$, then there exists a Barile-Macchia resolution $\mathcal{F}_B$ of $I$ such that $\rank (\mathcal{F}_B)_i\leq \rank (\mathcal{F}_L)_i$ for each $i\in \mathbb{Z}$.
\end{conjecture}

\begin{question}
    Let $I$ be a monomial ideal. Assume a Lyubeznik resolution or the Scarf complex of $R/I$ is the minimal free resolution. What conditions are sufficient for $I$ to have a minimal Barile-Macchia resolution?
\end{question}

Another question of the same theme is whether Corollary \ref{cor:5generators} can be extended to monomial ideals with more than $5$ generators.

In \cite{BW02}, the original statements of Theorems \ref{thm:morseres} and \ref{def:LBW}  are given in a more general way by replacing the Taylor simplicial complex with a compactly graded regular CW-complex. Therefore, the introduction of Barile-Macchia resolutions opens up various directions to obtain resolutions that one can hope to be minimal. In particular, one can use Algorithm \ref{algorithm1} to ``trim" the corresponding Lyubeznik resolution $\mathcal{F}_L$ and obtain a resolution $\mathcal{F}_B$ such that  $\rank (\mathcal{F}_B)_i\leq \rank (\mathcal{F}_L)_i$ for each $i\in \mathbb{Z}$. In fact, Barile and Macchia in \cite{BM20} used this method to obtain the minimal free resolutions of edge ideals of nonweighted nonoriented forests, which is a special case of our study in Section \ref{sec:forests}. This fact suggests the following question:

\begin{question}
    Does applying Algorithm \ref{algorithm1} to a Lyubeznik resolution of $R/I$ with respect to some total ordering on $\G(I)$ always produce resolution that is isomorphic to a Barile-Macchia resolution?
\end{question}

Finally, we discuss the generalization of Barile-Macchia resolutions. Generalized Lyubeznik resolutions are  meaningful generalizations of Lyubeznik resolutions in the sense that while the latter are rarely minimal, the former have been showed to be minimal for large classes of ideals such as generic and shellable ideals \cite[Proposition 4.1, Proposition 4.3]{BW02}. Barile-Macchia resolutions, meanwhile, are already minimal for large classes of ideals, and it begs the question of whether their generalization is meaningful. We found an example that provides an affirmative answer to this question.

\begin{example}
Let $I=(x_1x_2,x_2x_3,\ldots, x_8x_9,x_9x_1)$ and set $m_i=x_ix_{i+1}$ for each $1\leq i\leq 8$ and $m_9=x_9x_1$. Then, by Macaulay2 \cite{M2}, one can show that none of the Barile-Macchia resolutions of $I$ are minimal.  Next, we construct a minimal generalized Barile-Macchia resolution of $R/I$. Consider the following total orderings:
\[m_9 >_1 m_8 >_1 m_7 >_1 m_6 >_1 m_5 >_1 m_4 >_1 m_3 >_1 m_2 >_1 m_1,
\]
\[m_9 >_2 m_8 >_2 m_7 >_2 m_6 >_2 m_5 >_2 m_4 >_2 m_2 >_2 m_3 >_2 m_1.
\]

Let $X$ denote the Taylor resolution of $R/I$. Note that  $(X, \lcm)$ is $\ZZ^9$-graded. Let $P=\ZZ^9 \cup \{p_0\}$ be a poset where  $p_0 < p$ if $(1,1,1,1,1,0,0,1,1)< p$ and $p_0 > p$ if $(1,1,1,1,1,0,0,1,1)> p$  for any $p\in \ZZ^9$.  Define a map $f$ as follows:
\begin{align*}
    f: X^{(*)} &\to P\\
    \{m_1,m_2,m_3,m_4,m_8\} &\mapsto p_0,\\
    \{m_1,m_2,m_4,m_8\} &\mapsto p_0,\\
    \sigma &\mapsto \lcm(\sigma),
\end{align*}
where $\sigma \neq \{m_1,m_2,m_3,m_4,m_8\} , \{m_1,m_2,m_4,m_8\}$. One can verify that $f$ is an $\lcm$-compatible $P$-grading of $X$ by considering a map $g:P\rightarrow \ZZ^9$ where $g(p)=p$ for any $p\neq p_0$ and $g(p_0)=(1,1,1,1,1,0,0,1,1)$. By applying Algorithm \ref{algorithm1} to $f^{-1}(p_0)$ imposed with $(>_2)$ and $f^{-1}(p)$ imposed with $(>_1)$ for any $p\neq p_0$, we obtain an $f$-homogeneous acyclic matching such that for each $n\in \ZZ$, the number of critical subsets of cardinality $n$ equals the total Betti number $\beta_n(R/I)$.  Therefore the corresponding generalized Barile-Macchia resolution is minimal.
\end{example}

We conclude the paper with a question on whether more  examples of this kind can be found. 

\begin{question}\label{ques:generalizedbutnotbridge}
    What class of monomial ideals admits minimal generalized Barile-Macchia resolutions? What class of monomial ideals admits no minimal Barile-Macchia resolution, but a minimal generalized Barile-Macchia resolution?
\end{question}

\textbf{Acknowledgements.} The first author was supported by NSF grants DMS 1801285 and 2101671. We thank Srikanth Iyengar for his comments and suggestions.


\bibliographystyle{amsplain}
\bibliography{refs}
\end{document}